\pdfoutput=1
\documentclass[11pt]{amsart}
\usepackage[backend=biber,
	    url=false,
	    doi=false,
	    isbn=false,
	    date=year,
	    maxbibnames=99,
	    giveninits,
	    sortcites=true,
	    sorting=nty,
	    style=numeric]{biblatex}
\renewbibmacro{in:}{}

 \usepackage[utf8]{inputenc}
\usepackage[T1]{fontenc}
\usepackage[text={445pt,630pt}, centering, a4paper, marginparwidth=2cm]{geometry}
\usepackage[hidelinks,pdfusetitle]{hyperref}
\usepackage[activate={true,nocompatibility},final,tracking=true,kerning=true,factor=1100,stretch=10,shrink=10]{microtype}
\usepackage{braket}    
\usepackage{amsthm}    
\usepackage{amssymb}   
\usepackage{mathrsfs}  
\usepackage{bbold}
\usepackage[shortlabels]{enumitem}
\usepackage[justification=centering]{caption}
\usepackage[margin=5pt,justification=centering,labelformat=simple,labelfont=sc]{subcaption}
\usepackage{graphicx}

\theoremstyle{plain}
\newtheorem{theorem}{Theorem}[section]
\newtheorem{corollary}[theorem]{Corollary}
\newtheorem{proposition}[theorem]{Proposition}
\newtheorem{lemma}[theorem]{Lemma}

\newtheorem{definition}[theorem]{Definition}

\theoremstyle{remark}
\newtheorem{remark}[theorem]{Remark}
\newtheorem{example}[theorem]{Example}

\newcommand{\cf}{\ensuremath{\mathscr{F}}}

\newcommand{\ind}[1]{\mathbb{1}_{#1}}
\newcommand{\un}{\mathbb{1}}
\newcommand{\zero}{\mathbb{0}}

\date{\today}

\author{Jean-François Delmas}
\address{Jean-François Delmas,
  CERMICS, \'{E}cole des Ponts, France}
\email{jean-francois.delmas@enpc.fr}

\author{Dylan Dronnier}
\address{Dylan Dronnier,
  CERMICS, \'{E}cole des Ponts, France}
\email{dylan.dronnier@enpc.fr}

\author{Pierre-André Zitt}
\address{Pierre-André Zitt, LAMA, Université Gustave Eiffel, France}
\email{pierre-andre.zitt@univ-eiffel.fr}

\newcommand{\abs}[1]{\left\lvert\,#1\,\right\rvert}
\newcommand{\norm}[1]{\left\lVert\,#1\,\right\rVert}
\newcommand{\ceil}[1]{\left\lceil\,#1\,\right\rceil}
\newcommand{\floor}[1]{\left\lfloor\,#1\,\right\rfloor}
\newcommand{\R}{\ensuremath{\mathbb{R}}}

\newcommand{\N}{\ensuremath{\mathbb{N}}}

\newcommand{\Z}{\ensuremath{\mathbb{Z}}}
\newcommand{\E}{\ensuremath{\mathbb{E}}}

\newcommand{\Diag}{\ensuremath{\mathrm{Diag}}}

\newcommand{\cpu}{\mathcal{S}^\mathrm{uni}}
\newcommand{\cpa}{\mathcal{P}^\mathrm{Anti}}
\newcommand{\F}{\ensuremath{\mathcal{F}}}
\newcommand{\AF}{\ensuremath{\mathcal{F}^\mathrm{Anti}}}
\newcommand{\kk}{\ensuremath{\mathrm{k}}}
\newcommand{\rd}{\ensuremath{\mathrm{d}}}

\newcommand{\cp}{\ensuremath{\mathcal{P}}}

\newcommand{\etau}{\ensuremath{\eta^\mathrm{uni}}}

\newcommand{\FF}{\ensuremath{\mathbf{F}}}

\newcommand{\lb}{[\![}
\newcommand{\rb}{]\!]}
 
\newcommand{\cmir}{\ensuremath{c_\star}}
\newcommand{\cmar}{\ensuremath{c^\star}}

\newcommand{\mir}{\ensuremath{R_{e\star}}}
\newcommand{\mar}{\ensuremath{R_e^\star}}

\newcommand{\regular}{constant degree} 
\newcommand{\Regular}{Constant degree} 
\newcommand{\ho}{\ensuremath{\eta^{\mathrm{h}}}}
\newcommand{\ver}{\ensuremath{\eta^{\mathrm{v}}}}
\newcommand{\hoPath}{\mathscr{P}_{\mathrm{h}}}
\newcommand{\verPath}{\mathscr{P}_{\mathrm{v}}}
\newcommand{\etag}{\ensuremath{\eta_0}}
\newcommand{\etad}{\ensuremath{\eta_{1}}}

\newcommand{\hg}{\ensuremath{h_0}}
\newcommand{\hd}{\ensuremath{h_1}}
\newcommand{\gag}{\ensuremath{\gamma_0}}
\newcommand{\gad}{\ensuremath{\gamma_1}}
\newcommand{\fg}{\ensuremath{\phi_0}}
\newcommand{\fd}{\ensuremath{\phi_1}}
\newcommand{\Vg}{\ensuremath{V_0}}
\newcommand{\Vd}{\ensuremath{V_1}}
\newcommand{\Mg}{\ensuremath{M_0}}
\newcommand{\Md}{\ensuremath{M_1}}
\newcommand{\Sd}{\ensuremath{{\mathbb{S}^{d-1}}}}

\newcommand{\So}{\mathcal{S}^{\bot \mathrm{id}}}
\newcommand{\Sball}{\mathcal{S}^{\mathrm{balls}}}
\newcommand{\leftStrats}{\mathcal{S}_0}
\newcommand{\rightStrats}{\mathcal{S}_1}
\newcommand{\orthStrats}{\mathcal{S}^{\bot\alpha}}

\title{Optimal vaccination: various (counter) intuitive examples}

\begin{document}

\thanks{This work is partially supported by Labex B\'ezout reference ANR-10-LABX-58}

\subjclass[2010]{92D30, 47B34, 47A10, 58E17, 34D20}

\keywords{Kernel operator, vaccination strategy,
  effective reproduction number, multi-objective optimization, Pareto frontier}

\begin{abstract}
  In previous articles, we formalized the problem of optimal allocation strategies for a
  (perfect) vaccine in an infinite-dimensional metapopulation model. The aim of the
  current paper is to illustrate this theoretical framework with multiple examples where
  one can derive the analytic expression of the optimal strategies. We discuss in
  particular the following points: whether or not it is possible to vaccinate optimally
  when the vaccine doses are given one at a time (greedy vaccination strategies); the
  effect of assortativity (that is, the tendency to have more contacts with similar
  individuals) on the shape of optimal vaccination strategies; the particular case where
  everybody has  the same number of neighbors.
\end{abstract}

\maketitle

\section{Introduction}

\subsection{Motivation}

The basic reproduction number, denoted by $R_0$, plays a fundamental role in epidemiology
as it determines the long-term behavior of an epidemic. For a homogeneous model, it is
defined as the number of secondary cases generated by an infected individual in an
otherwise susceptible population. When this number is below $1$, an infected individual
causes less than one infection before its recovery in average; the disease therefore
declines over time until it eventually dies out. On the contrary, when the reproduction
number is greater than $1$, the disease  invades the population.
It follows from this property  that a proportion equal to $1 - 1/R_0$ of the population
should be immunized in order to stop the outbreak. We refer the reader to the monograph of
Keeling and Rohani~\cite{keeling_modeling_2008} for a reminder of these basic properties
on the reproduction number.

In heterogeneous generalizations of classical compartmental models, also called
metapopulation models, see \cite{lajmanovich1976deterministic, beretta_global_1986,
delmas_infinite-dimensional_2020}, the population is stratified into homogeneous groups
sharing the same characteristics (time to recover from the disease, interaction with the
other groups, \ldots). For these models, it is still possible to define a meaningful
reproduction number~$R_0$, as the number of secondary cases generated by a typical
infectious individual when all other individuals are uninfected; see~\cite{Diekmann1990}.
The reproduction number can then be identified as the spectral radius of the so-called
next generation matrix, see~\cite{VanDenDriessche}.  With this definition, it is still
true that the outbreak dies out if~$R_0$ is smaller than~$1$ and invades the population
otherwise; see~\cite{LocalStabilityThieme1985, StabilityOfThHethco1985, VanDenDriessche,
thieme_global_2011, delmas_infinite-dimensional_2020} for instance.

Suppose now  that we have at  our disposal a vaccine  with \emph{perfect efficacy}, that
is, vaccinated individuals are completely immunized to the  disease.  After  a 
vaccination  campaign,  let~$\eta$  denote  the proportion of \textbf{non-vaccinated}
individuals  in the population: in inhomogeneous models,  $\eta$ depends  \emph{a priori} 
on the  group as different groups may  be vaccinated differently.  We will  call $\eta$ a
\emph{vaccination strategy}. For  any strategy $\eta$, let  us denote by $R_e(\eta)$ the
corresponding reproduction  number of the non-vaccinated population, also called the
\emph{effective reproduction number}. In the metapopulation model, it can also be
expressed as the spectral radius of the  effective next  generation matrix,  see
Equation~\eqref{eq:Re-meta} below. The  choice of  $\eta$ naturally  raises a  question
that  may be expressed as the following informal constrained optimization problem:
\begin{equation}\label{eq:informal_optim1}%
  \begin{cases}
    \textbf{Minimize: } & \text{the quantity of vaccine to administrate} \\
    \textbf{subject to: } & \text{herd immunity is reached, that is, $R_e \leq 1$.}
  \end{cases}
\end{equation}
For practical reasons, we will instead look at the problem the other way around. If the
vaccine is only available in limited quantities, the decision makers could try to allocate
the doses so as to maximize efficiency; a natural indicator of this efficiency is the
effective reproduction number. This reasoning leads to the following constrained problem:
\begin{equation}\label{eq:informal_optim2}%
  \begin{cases}
    \textbf{Minimize: } & \text{the effective reproduction number~$R_e$} \\
    \textbf{subject to: } & \text{a given quantity of available vaccine.}
  \end{cases}
\end{equation}
In accordance with~\cite{ddz-theo}, we will denote by $\mir$ the value of this problem: it
is a function of the quantity of available vaccine. The graph of this function is called
the \emph{Pareto frontier}. In order to measure how bad a vaccination strategy can be, we
will also be interested in maximizing the effective reproduction number given a certain
quantity of vaccine:
\begin{equation}\label{eq:informal_max}%
  \begin{cases}
    \textbf{Maximise: } & \text{the effective reproduction number~$R_e$} \\
    \textbf{subject to: } & \text{a given quantity of available vaccine.}
  \end{cases}
\end{equation}
The value function corresponding to this problem is denoted by $\mar$ and its graph is
called the \emph{anti-Pareto frontier}. We will quantify the ``quantity of available
vaccine'' for the vaccination strategy $\eta$ by a cost $C(\eta)$. Roughly speaking the
``best'' (resp. ``worst'') vaccination strategies are solutions to
Problem~\eqref{eq:informal_optim2} (resp. Problem~\eqref{eq:informal_max}). Still
following~\cite{ddz-theo}, they will be called \emph{Pareto optimal} (resp.
\emph{anti-Pareto optimal}) strategies.

The problem of optimal vaccine allocation has been studied mainly in the metapopulation
setting where the population is divided into a finite number of subgroups with the same
characteristics. Longini, Ackerman and Elverback were the first interested in the question
of optimal vaccine distribution given a limited quantity of vaccine supply
\cite{AnOptimizationLongin1978}. Using the concept of next-generation matrix introduced by
Diekmann, Heesterbeek and Metz \cite{Diekmann1990}, Hill and Longini reformulated this
problem thanks to the reproduction number \cite{hill-longini-2003}. Several theoretical
and numerical studies followed focusing on Problem~\eqref{eq:informal_optim1} and/or
Problem~\eqref{eq:informal_optim2} in the metapopulation setting
\cite{DistributionOfGoldst2010, poghotanyan_constrained_2018, DoseOptimalVaDuijze2018,
EvaluatingVaccHaoL2019}. We also refer the reader to the introduction of~\cite{ddz-theo}
for a detailed review of the bibliography.

\medskip
 
In two previous works~\cite{ddz-theo, ddz-Re}, we provided an
infinite-dimensional framework generalizing the metapopulation model where
Problems~\eqref{eq:informal_optim2} and~\eqref{eq:informal_max} are well-posed, justified
that the optimizers are indeed Pareto optimal and studied in detail the Pareto and
anti-Pareto frontiers. Since there is no closed form for the effective reproduction
number, Problems~\eqref{eq:informal_optim2} and~\eqref{eq:informal_max} are hard to solve
in full generality: our goal here is to exhibit examples where one can derive analytic
expressions for the optimal vaccination strategies. The simple models we study give a
gallery of examples and counter-examples to natural questions or conjectures, and may help
understanding common rules of thumb for choosing vaccination policies. We will in
particular be interested in the following three notions.

\begin{enumerate}[(i)]
  \item \textbf{Greedy parametrization of the frontiers}. For the decision maker it is
    important to know if global optimization and sequential optimization are the same as
    one cannot unvaccinated people and redistribute the vaccine once more doses become
    available. More precisely, there is a natural order on the vaccination strategies: let
    us write $\eta'\leq \eta$ if all the people that are vaccinated when following the
    strategy~$\eta$ are also vaccinated when following the strategy~$\eta'$. Let $\eta$ be
    an optimal solution of~\eqref{eq:informal_optim2} for cost $c=C(\eta)$, that is,
    $R_e(\eta) = \mir(c)$. If, for $c'>c$, we can find a strategy $\eta'\leq \eta$ such
    that $R_e(\eta') = \mir(c')$, then the optimization may be, at least in principle,
    found in a greedy way: giving sequentially each new dose of vaccine so as to minimize
    $R_e$ gives, in the end, an optimal strategy for any quantity of vaccine. By analogy
    with the corresponding notion for algorithms we will say in this case that there
    exists a \emph{greedy parametrization} of the Pareto frontier. The existence of such a
    greedy parametrization was already discussed by Cairns in
    \cite{EpidemicsInHeCairns1989} and is examined for each model throughout this paper.
  \item \textbf{Assortative/Disassortative network}. The second notion is a property of
    the network called \emph{assortativity}: a network is called assortative when the
    nodes tend to attach to others that are similar in some way and \emph{disassortative}
    otherwise. The assortativity or disassortativity of a network is an important property
    that helps to understand its topology. It has been oberved that social networks are
    usually assortative while biological and technological networks are disassortative,
    see for example~\cite{AssortativeMixNewman2002}. The optimal vaccination strategies
    can differ dramatically in the case of assortative versus disassortative mixing, see
    Galeotti and Rogers  \cite{StrategicImmunGaleot2013} for a study in a population
    composed of two groups. This question is in particular addressed in
    Section~\ref{sec:ass-disass}  for an elementary  model with an
    arbitrary number of groups.

\item \textbf{How to handle individuals with the same level of
    connection}. Targeting individuals that are the most connected is a common approach used to prevent an
epidemic in a complex network \cite{pastor-satorras_immunization_2002}. In
\cite{ddz-mono}, we show that these strategies are optimal for the so-called monotonic
kernel models, in which the  individuals may be naturally
ordered by a score related to their connectivity.
  When many individuals or groups  are tied for the best score,
  either from the beginning or after some vaccine has been distributed,
the optimal way of vaccinating them
 may be surprisingly varied according to the 
 situation.
This variety of
answers  appears already in the treatment of such individuals in the assortative/disassortative toy model
developed in  Section~\ref{sec:ass-disass}. To go further in this
direction, a large part of the current paper, see
Sections~\ref{sec:constant_degree}-\ref{sec:geometric},  is devoted to
regular or ``\regular'' models 
where  all individuals share the same
degree.
We shall in particular ask whether uniform vaccination
strategies are either the ``best'' or the ``worst'' or even neither the ``best''
nor the ``worst'' possible strategies. 
    
\end{enumerate}

\subsection{Main results}

Section~\ref{sec:discrete}  is dedicated  to  classical finite-dimensional metapopulation
models.  We present two  simple models that,  despite being seemingly  very similar,
display totally different behaviors: the asymmetric and symmetric circle graphs. For  the
first one, where individuals of the group $i$ can only be infected by individuals of the
group $i-1$ and   which corresponds to a next generation matrix given by:
\[
  K_{ij}=\ind{\{i=j+1 \mod{N}\}},
\]
with $N$ the number of groups or nodes in the circle, we derive  a greedy parametrization 
of the
Pareto  frontier. On  the second  one,  where individuals of the group $i$ can only be
infected by individuals of the group $i-1$ or $i+1$ and  which corresponds to a next
generation matrix given by:
\[
  K_{ij}=\ind{\{i=j\pm 1 \mod{N}\}},
\]
we  observe numerically  that the Pareto frontier  is much more  complicated, and in
particular  cannot be parametrized  greedily. Those  two  models are  in  fact \regular{}
models;  the uniform vaccination  strategies are the  ``worst'' for the  first model, and
neither the ``best'' nor the ``worst'' strategies for the second.

After Section~\ref{sec:setting}, where we recall the kernel setting used in  
\cite{ddz-theo}  for   infinite   dimensional   models,  we   focus
in~Section~\ref{sec:ass-disass}  on  the   effect  of  assortativity  on optimal
vaccination strategies. We define a simple kernel model that may be  assortative  or 
disassortative  depending   on  the  sign  of a parameter. In  the discrete 
metapopulation model, the  next generation matrix can be written (up to a multiplicative
constant) as:
\[
  K_{ij}= \left(1 + \varepsilon \ind{\{i\neq j\}}\right)\, \mu_j ,
\]
where $\mu_j\geq 0$ represents the  proportional size of group $j$.  The model  is
assortative  if $\varepsilon<0$  (and $\varepsilon\geq  -1$ so that   the   matrix  $K$  
is   non-negative)   and  disassortative   if $\varepsilon>0$.   We   describe  completely
the   optimal  vaccination strategies, see Theorem~\ref{th:dis-assortative}, and show
that the best strategies for  the assortative case  are the  worst ones if  the mixing
pattern is  disassortative, and vice-versa.  We also prove that  all the Pareto and
anti-Pareto frontiers admit greedy parametrizations, and that Pareto optimal strategies
prioritize individuals that in some sense have the highest degree, that is, are the most
connected.

In  Section~\ref{sec:constant_degree},  we consider  \regular{}  models,
which are  the analogue  of regular  graphs in  the infinite-dimensional
setting.  In the  discrete metapopulation model, the sums  over each row
and  the  sums over  each  column  of  the  next generation  matrix  are
constant.   We prove,  see  Proposition~\ref{prop:regular},  that if  the
effective  reproduction  function  $R_e$  is  convex  then  the  uniform
strategies are the  ``best'' and they give a  greedy parametrization of
the  Pareto  frontier;  and  that  if  $R_e$  is  concave,  the  uniform
strategies are the ``worst''.
Section~\ref{sec:rank-2-reg} is then devoted to a particular model of
rank two, which corresponds in the
discrete metapopulation model to a  next 
generation matrix of the form:
\[
  K_{ij}= (1+ \varepsilon \alpha_i \alpha_j)\, \mu_j
  \quad\text{with} \quad \sum_j \alpha_j \, \mu_j=0,
\]
where $\varepsilon$ may be $+1 $ or $-1$,  and $\sup_i \alpha_i^2\leq
1$, so that the matrix $K$ is non-negative.  The condition
$\sum_j \alpha_j\mu_j=0$ ensures that the model has a constant
degree. In those cases, we give a complete description of the ``best''
and the ``worst'' vaccination strategies, the uniform one being
``best'' for $\varepsilon=+1$ and ``worst'' otherwise, see
Proposition~\ref{prop:rank2}.  In Section~\ref{ssec:infinite}, we also
provide an example of kernel (in infinite dimension) for which the set
of optimal strategies has an infinite number of connected
components. In this particular case, there is no greedy
parametrization of the Pareto frontier.

As another application of the results of Section~\ref{sec:constant_degree}, we investigate
in Section~\ref{sec:geometric}  geometric \regular{} kernels
defined on the unit sphere $\Sd\subset \R^d$. Intuitively an individual at point $x$ on the
sphere is infected by an individual at point $y$ with an intensity
$\kk(x,y)$ depending on the distance between $x$ and $y$. Those 
kernels appear in the graphon theory as limit of large dense random
geometric graphs.
We give 
 a particular attention  to the affine model in
 Section~\ref{sec:sphere-aff}, where:
 \[
   \kk(x,y)=1+ \varepsilon \langle
   x,y \rangle, \qquad \varepsilon \geq -1,
 \]
 where $\langle x,y \rangle$ is the usual scalar product in
 the ambient space~$\R^d$. 
Intuitively, for 
 $\varepsilon>0$,  the infection propagates through the
 nearest neighbors: this may be seen as a  kind of spatial assortativity. 
 By contrast, for  $\varepsilon<0$   the infection propagates through the
 furthest individuals neighbors, in a spatially disassortative way. 
 For this affine model, we  completely describe the   ``best''   and  the  ``worst''
vaccination   strategies, see Proposition~\ref{prop:rank2-sphere}.  

\section{First examples in the discrete setting}\label{sec:discrete}

In this section, we use the framework developped by Hill and Longini in
\cite{hill-longini-2003} for metapopulation models and provide optimal vaccination
strategies for two very simple examples. Despite their simplicity, these examples showcase
a number of interesting behaviors, that will occur a in much more general setting, as we
will see in the rest of the paper.

\subsection{The reproduction number in metapopulation models}\label{sec:discrete-reprod}

In metapopulation models, the population is divided into $N \geq 2$ different
subpopulations and we suppose that individuals within a same subpopulation share the same
characteristics. The different groups are labeled $1$, $2$, \ldots, $N$. We denote by
$\mu_1$, $\mu_2$, \ldots, $\mu_{N}$ their respective size (in proportion with respect to
the total size) and we suppose that those do not change over time. By the linearization of
the dynamic of the epidemic at the disease-free equilibrium, we obtain the so-called
\emph{next-generation matrix} $K$, see~\cite{VanDenDriessche}, which is a $N \times N$
matrix with non-negative coefficients. For a detailed discussion on the biological
interpretation of the coefficients of the next-generation matrix, we refer the reader to
\cite[Section~2]{ddz-Re}. We also refer to \cite{ddz-2pop} for an extensive treatment of
the two-dimensional case. 

The basic reproduction number is equal to the spectral radius of the next-generation
matrix:
\begin{equation}
  R_0 = \rho(K),
\end{equation}
where $\rho$ denotes the spectral radius. Since the matrix $K$ has non-negatives entries,
the Perron-Frobenius theory implies that $R_0$ is also an eigenvalue of $K$. If $R_0 > 1$,
the epidemic process grows away from the disease-free equilibrium while if $R_0 < 1$, the disease
cannot invade the population; see \cite[Theorem~2]{VanDenDriessche}.

We now introduce the effect of vaccination. Suppose that we have at our disposal a
vaccine with perfect efficacy, \textit{i.e.}, vaccinated individuals are completely
immunized to the infection. We denote by $\eta=(\eta_1, \ldots, \eta_{N})$ the vector of
the proportions of \textbf{non-vaccinated} individuals in the different groups. We shall
call $\eta$ a vaccination strategy and denote by $\Delta=[0, 1]^N$ the set of all possible
vaccination strategies. According to \cite{ddz-Re,ddz-theo}, the next-generation matrix
corresponding to the dynamic with vaccination is equal to the matrix $K$ multiplied by the
matrix $\Diag(\eta)$ on the right, where $\Diag(\eta)$ is the $N \times N$ diagonal matrix
with coefficients $\eta \in \Delta$. We call the spectral radius of this matrix the
\textit{effective reproduction number}:
\begin{equation}
  \label{eq:Re-meta}
  R_e(\eta) = \rho\left( K \cdot \Diag(\eta) \right).
\end{equation}
The effective reproduction number accounts for the vaccinated (and immunized) people in
the population, as opposed to the basic reproduction number, which corresponds to a fully
susceptible population. When nobody is vaccinated, that is $\eta = \un = (1,\ldots,1)$, 
$\Diag(\eta)$ is equal to the identity matrix, the next-generation matrix is unchanged and
$R_e(\eta)=R_e(\un)=R_0$.

We suppose that the \emph{cost} of a vaccination strategy is, up to an irrelevant
multiplicative constant, equal to the total proportion of vaccinated people and is
therefore given by:
\begin{equation}
  C(\eta) = \sum_{i=1}^{N} (1 - \eta_i) \mu_i = 1 - \sum_{i=1}^{N} \eta_i \mu_i,
\end{equation}
where $\eta=(\eta_1, \ldots, \eta_{N})\in \Delta$. We refer to
\cite[Section~5.1, Remark 5.2]{ddz-theo}
for considerations on more general cost functions.

\begin{example}[Uniform vaccination]\label{ex:uniform}
  The uniform strategy of cost $c$ consists in vaccinating the same proportion of people
  in each group:~$\eta = (1-c)\un$. By homogeneity of the spectral radius, the reproduction number~$R_e(\eta)$ is
  then equal to~$(1 - c) R_0$.
\end{example}

\subsection{Optimal allocation of vaccine doses}\label{sec:discrete-prob}

As mentioned in the introduction and recalled in Section~\ref{sec:discrete-reprod},
reducing the reproduction number is fundamental in order to control and possibly eradicate
the epidemic. However, the vaccine may only be available in a limited quantity, and/or the
decision maker may wish to limit the cost of the vaccination policy. This motivates our
interest in the following related problem:
\begin{equation}\label{eq:prob-min-Re}
  \left\{
    \begin{array}{cc}
      \min \, & R_e(\eta), \\
      \text{such that} & C(\eta) = c.
    \end{array}
  \right.
\end{equation}
According to \cite{ddz-theo}, one can replace the constraint~$\{C(\eta)=c\}$
by~$\{C(\eta)\leq c\}$ without modifying the solutions.  The opposite problem consists in
finding out the \emph{worst possible way} of allocating vaccine. While this does not seem
at first sight to be as important, a good understanding of bad vaccination strategies may
also provide rules of thumb in terms of anti-patterns.  In order to estimate how bad a
vaccination strategy can be, we therefore also consider the following problem:
\begin{equation}\label{eq:prob-max-Re}
  \left\{
    \begin{array}{cc}
      \max \, & R_e(\eta), \\
      \text{such that} & C(\eta) = c.
    \end{array}
  \right.
\end{equation}
According to \cite{ddz-theo}, one can replace the constraint ~$\{C(\eta)=c\}$
by~$\{C(\eta)\geq c\}$ without modifying the solutions. \medskip
 
Since the coefficients of the matrix~$K\cdot \Diag(\eta)$ depend continuously on~$\eta$, it is
classical that its eigenvalues also depend continuously on~$\eta$ (see for example
\cite[Appendix~D]{horn2012matrix}) and in particular the function~$R_e$ is continuous
on~$\Delta=[0,1]^N$. Since the function~$C$ is also continuous on~$\Delta$, the
compactness of~$\Delta$ ensures the existence of solutions for
Problems~\eqref{eq:prob-min-Re} and~\eqref{eq:prob-max-Re}. For ~$c \in [0,1]$,~$\mir(c)$
(resp.~$\mar(c)$) stands for the minimal (resp. maximal) value taken by~$R_e$ on the set
of all vaccination strategies~$\eta$ such that~$C(\eta) = c$:
\begin{align}
  \label{eq:mir}
  \mir(c) &= \min \{ R_e(\eta)\,\colon\, \eta\in \Delta \text{ and } C(\eta) = c \},\\
  \label{eq:mar}
  \mar(c) &= \max \{ R_e(\eta)\,\colon\, \eta\in \Delta \text{ and } C(\eta) = c \}.
\end{align}
It is easy to check that the functions~$\mir$ and~$\mar$ are non increasing. Indeed,
if~$\eta^1$ and~$\eta^2$ are two vaccination strategies such that~$\eta^1 \leq \eta^2$
(where~$\leq$ stands for the pointwise order), then~$R_e(\eta^1) \leq R_e(\eta^2)$
according to the Perron-Frobenius theory. This easily implies that~$\mir$ and~$\mar$ are
non-increasing. We refer to \cite{ddz-theo,ddz-Re} for more properties on those functions;
in particular they are also continuous. For the vaccination strategy~$\eta = \zero =
(0,...,0)$ (everybody is vaccinated) with cost~$C(\mathbb{0}) = 1$, the transmission of
the disease in the population is completely stopped, \textit{i.e.}, the reproduction
number is equal to~$0$. In the examples below, we will see that for some next-generation
matrices~$K$, this may be achieved with a strategy~$\eta$ with cost~$C(\eta) < 1$. Hence,
let us denote by~$\cmir$ the minimal cost required to completely stop the transmission of
the disease:
\begin{equation}
  \label{eq:def-cmir}
  \cmir = \inf \{ c \in [0,1] \, \colon \, \mir(c) = 0 \}.
\end{equation}
In a similar fashion, we define by symmetry the maximal cost of totally inefficient
vaccination strategies:
\begin{equation}
  \label{eq:def-cmar}
  \cmar = \sup \{ c \in [0,1] \, \colon \, \mar(c) = R_0 \}.
\end{equation}
According to \cite[Lemma~5.13]{ddz-theo}, we have~$\cmar = 0$ if the matrix~$K$ is
irreducible, \textit{i.e.}, not similar via a permutation to a block upper triangular
matrix. The two matrices considered below in this section are irreducible. \medskip

Following \cite{ddz-theo}, the \emph{Pareto frontier} associated to the ``best''
vaccination strategies, solution to Problem~\eqref{eq:prob-min-Re}, is defined by:
\begin{equation}
  \label{eq:def-P-front}
  \F = \{(c, \mir(c)) \, \colon \, c \in [0,\cmir]\}. 
\end{equation}
The set of ``best'' vaccination strategies, called \emph{Pareto optimal} strategies, is
defined by:
\begin{equation}\label{eq:def-P-set}
  \cp=\{\eta\in \Delta\, \colon\, (C(\eta), R_e(\eta))\in \F\}.
\end{equation}
When~$\cmar=0$ (which will be the case for all the examples considered in this paper), the
\emph{anti-Pareto frontier} associated to the  ``worst'' vaccination strategies, solution
to Problem~\eqref{eq:prob-max-Re}, is defined by: 
\begin{equation}\label{eq:def-AP-front}
  \AF = \{(c, \mar(c)) \, \colon \, c \in [0, 1]\}. 
\end{equation}
The set of ``worst'' vaccination strategies, called \emph{anti-Pareto optimal} strategies,
is defined by:
\begin{equation}\label{eq:def-AP-set}
  \cpa=\{\eta\in \Delta\, \colon\, (C(\eta), R_e(\eta))\in \AF\}.
\end{equation}
The set of uniform strategies will play a role in the sequel:
\begin{equation}\label{eq:def-cpu}
  \cpu=\{ t\un\, \colon\, t \in [0,1]\}.
\end{equation}
We denote by~$\FF=\{(C(\eta), R_e(\eta))\, \colon\, \eta\in \Delta\} $ the set of all
possible outcomes. According to \cite[Section~6.1]{ddz-theo}, the set~$\FF$ is a subset of
$[0,1] \times [0,R_0]$ delimited below by the graph of~$\mir$ and above by the graph
of~$\mar$; it is compact, path connected and its complement is connected in~$\R^2$.

\medskip

A \emph{path} of vaccination strategies is a measurable function~$\gamma \, \colon \,
[a,b] \to \Delta$ where~$a<b$. It is \emph{monotone} if for all~$a \leq s \leq t \leq b$
we have~$\gamma(s) \geq \gamma(t)$, where~$\leq$ denotes the pointwise order. A
\emph{greedy parametrization} of the Pareto (resp.\ anti-Pareto) frontier is a monotone
continuous path~$\gamma$ such that the image of $(C \circ \gamma, R_e \circ \gamma)$ is
equal to~$\F$ (resp.~$\AF$). If such a path exists, then its image can be browsed by a
greedy algorithm which performs infinitesimal locally optimal steps.

\begin{remark}\label{rem:R0-unknown}
  Let~$K$ be the next-generation matrix and let~$\lambda \in \R_+ \backslash \{0\}$. By
  homogeneity of the
  spectral radius, we have~%
  \(
  \rho(\lambda K\cdot \Diag(\eta))
  = \lambda \rho(K\cdot \Diag(\eta))
  \).
  Thus, the solutions of
  Problems~\eqref{eq:prob-min-Re} and~\eqref{eq:prob-max-Re} and the
  value of~$\cmir$ are invariant by scaling of the matrix~$K$. As for
  the functions~$\mir$ and~$\mar$, they are scaled by the same
  quantity. Hence, in our study, the value of~$R_0$ will not
  matter. Our main concern will be to find the best and the worst
  vaccination strategies for a given cost and compare them to the
  uniform strategy.
\end{remark}

\subsection{The fully asymmetric circle model}\label{sec:asym-circle}

We consider a model of~$N\geq 2$ equal subpopulations (\textit{i.e.}
$\mu_1 = \cdots =\mu_{N} = 1/N$) where each subpopulation only
contaminates the next one. The next-generation matrix, which is equal to
the cyclic permutation matrix, and the effective next generation matrix
are given by:
\begin{equation}
  K =
  \begin{pmatrix}
    0 & 1 & & & \\
      & 0 & 1 & & \\
      & & \ddots & \ddots & \\
    0 & & & 0 & 1\\
    1 & 0 & & & 0
  \end{pmatrix}
  \quad\text{and}\quad
  K \cdot \Diag(\eta) =
  \begin{pmatrix}
    0 & \eta_2 & & & \\
      & 0 & \eta_3 & & \\
      & & \ddots & \ddots & \\
    0 & & & 0 & \eta_{N} \\
    \eta_{1} & 0 & & & 0
  \end{pmatrix},
\end{equation}
where~$\eta=(\eta_1,  \ldots,   \eta_{N})\in  \Delta=[0,   1]^N$.  The
next-generation matrix can be interpreted as the adjacency matrix of the
fully asymmetric cyclic graph; see Figure~\ref{fig:graph-asy}.

\begin{figure}
  \begin{subfigure}[T]{.5\textwidth}
    \centering
    \includegraphics[page=1]{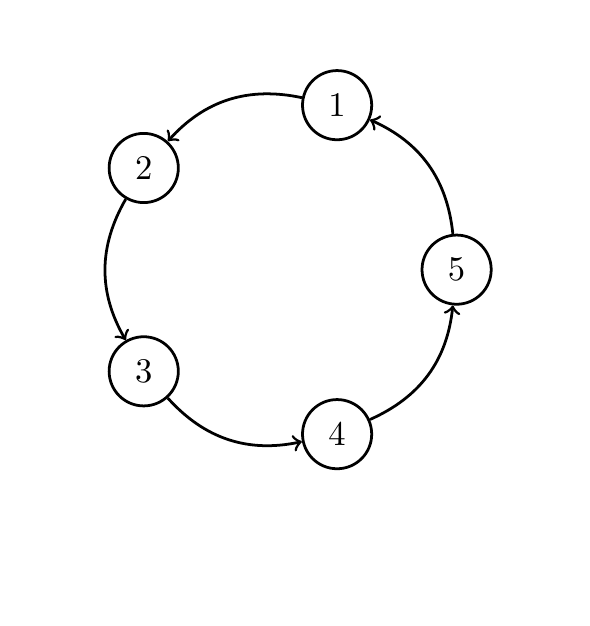}
    \caption{Graphical representation of the transmission
    of the disease.}\label{fig:graph-asy} 
  \end{subfigure}%
  \begin{subfigure}[T]{.5\textwidth}
    \centering
    \includegraphics[page=2]{asym-circle}
    \caption{Solid line: the Pareto frontier~$\F$; dashed line: the
      anti-Pareto frontier~$\AF$ (which corresponds to the uniform strategies); blue
    region: all possible outcomes~$\FF$.}\label{fig:pareto-asy}
  \end{subfigure}%
  \caption{Example of optimization for the fully asymmetric circle model with~$N=5$
  subpopulations.}\label{fig:optim-asy}
\end{figure} 

By an elementary computation, the characteristic polynomial of the matrix~$K
\cdot \Diag(\eta)$ is equal to~$X^N - \prod_{1 \leq i \leq N} \eta_i$. Hence, the
effective reproduction number can be computed via an explicit formula; it corresponds to
the geometric mean:
\begin{equation}\label{eq:geom-mean}
  R_e(\eta)= \left(\prod_{i=1}^{N} \eta_i\right)^{1/N}.
\end{equation}

The Pareto and anti-Pareto frontier are totally explicit for this elementary example, and
given by the following proposition. For additional comments on this example see also
Example~\ref{ex:asym-cyclic} below.

\begin{proposition}[Asymmetric circle]\label{prop:asym-circle}
  For the fully asymmetric circle model, we have:
  \begin{enumerate}[(i)]
  \item\label{pt:asym-circle-best} The least quantity of vaccine necessary to completely
    stop the propagation of the disease is~$\cmir=1/N$. Pareto optimal strategies have a
    cost smaller than~$\cmir$, and correspond to giving all the available vaccine to one
    subpopulation: \[ \cp=\left\{\eta=(\eta_1, \ldots, \eta_{N})\in [0, 1]^N\, \colon
    \eta_i=1 \text{ for all~$i$ but at most one}\right\}. \] The Pareto frontier is given
    by the graph of the function~$\mir$ on~$[0, \cmir]$, where~$\mir$ is given by:
    \[ 
      \mir(c) = (1-Nc )_+^{1/N} \quad\text{for}\quad c\in [0,1].
    \]
  \item\label{pt:asym-circle-worst} The maximal cost of totally inefficient vaccination
    strategies is~$\cmar=0$. The anti-Pareto optimal strategies consist in vaccinating
    uniformly the population, \emph{i.e.}:
    \[
      \cpa=\cpu.
    \]
    The anti-Pareto frontier is given by the graph of the
    function~$\mar: c \mapsto 1-c$ on $[0, 1]$. 
  \end{enumerate}
\end{proposition}

In Figure~\ref{fig:pareto-asy}, we have plotted the Pareto and the anti-Pareto frontiers
corresponding to asymmetric circle model with~$N=5$ subpopulations.

\begin{remark}[Greedy parametrization]
  From Proposition~\ref{prop:asym-circle}, we see that there exists a greedy
  parametrization of the Pareto frontier, which consists in giving all the available
  vaccine to one subpopulation until its complete immunization. Similarly, the anti-Pareto
  frontier is greedily parametrized by the uniform strategies.
\end{remark}

\begin{proof}
  We first prove \ref{pt:asym-circle-best}. Suppose that~$c\geq 1/N$. There is enough
  vaccine to protect entirely one of the groups and obtain~$R_e(\eta) = 0$ thanks to
  Equation~\eqref{eq:geom-mean}. This gives~$\cmir\leq 1/N$ and~$\mir(c)=0$ for~$c\geq
  1/N$.

  Let~$0 \leq c < 1/N$. According to \cite[Section~3.1.5]{bv2009}, the map~$\eta\mapsto
  R_e(\eta)$ is concave. According to Bauer's maximum principle
  \cite[Corollary~A.3.3]{ConvexFunctionNicule2006},~$R_e$ attains its minimum on~$\{ \eta
  \in [0, 1]^N \, \colon \, \, C(\eta) = c\}$ at some extreme point of this set. These
  extreme points are strategies~$\eta \in [0, 1]^N$ such that~$\eta_i = 1 - Nc$ for some
  $i$ and~$\eta_j = 1$ for all~$j \neq i$. Since~$R_e$ is a symmetric function of its~$N$
  variables, it takes the same value~$(1-Nc)^{1/N}$ on all these strategies, so they are
  all minimizing, which proves Point~\ref{pt:asym-circle-best}.

  \medskip

  We give another elementary proof of \ref{pt:asym-circle-best} when~$c < 1/N$. Let~$\eta$
  be a solution of Problem~\eqref{eq:prob-min-Re}. Assume without loss of generality
  that~$\eta_1 \leq \cdots\leq \eta_{N}$. Suppose for a moment that~$\eta_2 < 1$, and
  let~$\varepsilon > 0$ be small enough to ensure~$\eta_1>\varepsilon$ and~$\eta_2 < 1 -
  \varepsilon$. Then the vaccination strategy~$\tilde{\eta} = (\eta_1 - \varepsilon,
  \eta_2 + \varepsilon, \eta_3, \ldots,\eta_{N})$ is admissible, and:
  \[
    R_e(\tilde{\eta})^N =
    R_e(\eta)^N - (\varepsilon(\eta_2 - \eta_1) + \varepsilon^2)
    \prod_{i = 3}^{N} \eta_i< R_e(\eta)^N, 
  \]
  contradicting the optimality of~$\eta$. Therefore the Pareto-optimal strategies have
  only one term different from~$1$, and must be equal to \( ((1-Nc),1,\ldots, 1) \), up to
  a permutation of the indices. 

  \medskip

  Now, let us prove \ref{pt:asym-circle-worst}. Let~$\eta$ such that~$C(\eta) = c$.
  According to the inequality of arithmetic and geometric means:
  \begin{equation*}
    R_e(\eta) \leq \frac{\eta_1 + \cdots + \eta_{N}}{N} = 1-c. 
  \end{equation*}
  By Example~\ref{ex:uniform}, the right hand side is equal to the effective reproduction
  number of the uniform vaccination at cost~$c$.  This ends the proof of the proposition.
\end{proof}

\subsection{Fully symmetric circle model}\label{sec:fully-sym-circle}

We now consider the case where each of the~$N$ subpopulation may infect both of their
neighbours. The next-generation matrix and the effective next-generation matrix are given
by:
\begin{equation}
  \renewcommand{\arraystretch}{1.2}
  K =
  \begin{pmatrix}
    0 & 1 & & 0 & 1 \\
    1 & 0 & 1 & & 0 \\
      & 1 & \ddots & \ddots & \\
    0 & & \ddots & 0 & 1\\
    1 & 0 & & 1 & 0
  \end{pmatrix} \quad \text{and} \quad
  K\cdot \Diag(\eta) =
  \begin{pmatrix}
    0 & \eta_2 & & 0 & \eta_{N} \\
    \eta_1 & 0 & \eta_3 & & 0 \\
	   & \eta_2 & \ddots & \ddots & \\
    0 & & \ddots & 0 & \eta_{N} \\
    \eta_{1} & 0 & & \eta_{N-1} & 0
  \end{pmatrix} .
\end{equation}
Again, we can represent this model as a graph; see Figure~\ref{fig:graph-sy}.

\begin{figure}
  \begin{subfigure}[T]{.5\textwidth}
    \centering
    \includegraphics[page=1]{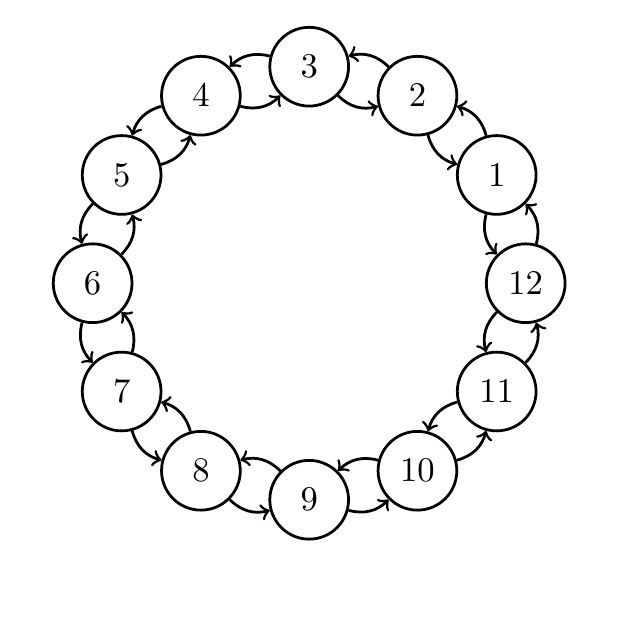} \caption{Graphical representation
    of the transmission of the disease.}
    \label{fig:graph-sy} 
  \end{subfigure}%
  \begin{subfigure}[T]{.5\textwidth} \centering
    \includegraphics[page=2]{sym-circle}
    \caption{Solid line: the Pareto frontier~$\F$; dashed line: the
      anti-Pareto frontier~$\AF$; dotted line: outcomes of the
    uniform strategies; blue region: all possible outcomes~$\FF$.}
    \label{fig:pareto-sy}
  \end{subfigure}%
  \caption{Example of optimization for the fully symmetric circle model with~$N=12$
  subpopulations.}
  \label{fig:optim-sy}
\end{figure} 

There is no closed-form formula to express~$R_e$ for~$N \geq 5$ and the optimization is
way harder than the asymmetric case. Since~$K$ is irreducible, we have~$\cmar=0$. Our only
analytical result for this model is the computation of~$\cmir$.

\begin{proposition}[Optimal strategy for stopping the
  transmission]\label{prop:cmir-full-circle}
  For the fully symmetric circle model, the strategy~$\eta' = \ind{i \,
  \text{even}}$ is Pareto optimal for the fully symmetric circle and~$R_e(\eta')=0$. In
  particular,~$\cmir$ is equal to~$C(\eta')=\ceil{N/2}/N$.
\end{proposition}

\begin{proof}
  The term~$X^{N-2}$ of the characteristic polynomial of~$K\cdot \Diag(\eta)$ has a coefficient
  equal to the sum of all principal minors of size~$2$:
  \begin{equation}
    - (\eta_1 \eta_2 + \eta_2 \eta_3 + \ldots + \eta_{N-1} \eta_{N} + \eta_{N} \eta_1). 
  \end{equation}
  If~$\eta$ is such that~$N C(\eta) < \ceil{N/2}$, then at least one of the term above is
  not equal to~$0$, proving that the sum is negative. Hence, there is at least one
  eigenvalue of~$K \cdot \Diag(\eta)$ different from~$0$, and~$R_e(\eta) > 0$. We deduce
  that~$\cmir \geq \ceil{N/2}/N$.

  Now, let~$\eta'$ such that~$\eta_i' = 0$ for all odd~$i$ and
  $\eta_i' = 1$ for all even~$i$, so that~$C(\eta') = \ceil{N/2}/N$. The matrix~$K\cdot
  \Diag(\eta')$ is nilpotent as its square is 0. Since the spectral radius of a nilpotent
  matrix is equal~$0$, we get~$R_e(\eta') = 0$. This ends the proof of the proposition.

  \medskip

  We can give another proof of the proposition: it is enough to notice that the nodes
  labelled with an odd number form a maximal independent set of the cyclic graph.
  Taking~$\eta'$ equal to the indicator function of this set, we deduce from
  \cite[Section~6.4]{ddz-Re} that~$\eta'$ is Pareto optimal,~$R_e(\eta')=0$
  and~$\cmir=C(\eta')$.
\end{proof}

We pursue the analysis of this model with numerical computations. We choose~$N=12$
subpopulations, and compute an approximate Pareto frontier, using
the Borg multiobjective evolutionary algorithm%
\footnote{%
  The algorithm is described in \cite{HR13}; we use the version coded in the
\texttt{BlackBoxOptim} package for the Julia programming language.}.
The results are plotted in Figure~\ref{fig:symmetric_circle}. We represent additionally
the curves~$(c, R(\eta(c)))$ where the vaccination strategy~$\eta(c)$ for a given cost~$c$
are given by deterministic path of ``meta-strategies'':
\begin{itemize}
  \item \textbf{Uniform strategy:} distribute the vaccine uniformly to all~$N$
    subpopulations; 
  \item \textbf{``One in~$j$'' strategy:} vaccinate one in~$j$ subpopulation,
    for~$j=2,3,4$.
\end{itemize} 

\begin{figure}
  \begin{subfigure}[T]{.5\textwidth}
    \centering
    \includegraphics[page=1]{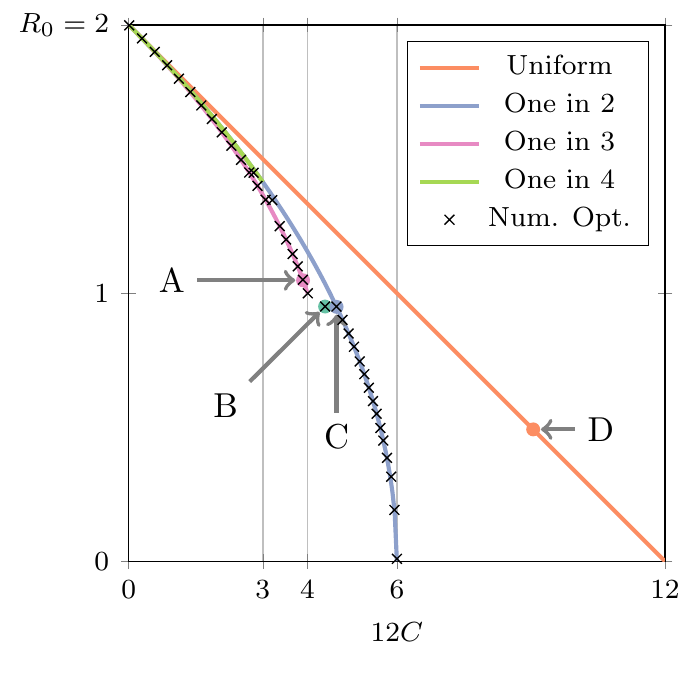}
    \caption{Effective reproduction number~$R_e$ against
    vaccination cost~$c$ for various meta-strategies. }
    \label{fig:symmetric-circle1}
\end{subfigure}%
\begin{subfigure}[T]{.5\textwidth}
  \centering
  \includegraphics[page=2]{cycle}
   \caption{Vaccination strategies corresponding to
   the four labelled points.}
 \label{fig:symmetric-circle2}
\end{subfigure}%
  \caption{Pareto frontier and computation of the outcomes for the paths of the four
    meta-strategies. Some meta-strategies~$\{ \eta_A, \eta_B, \eta_C, \eta_D \}$ are
    represented on the right with their corresponding outcome points~$\{ A, B, C, D \}$
    on the left.}\label{fig:symmetric_circle}%
\end{figure}

Let us follow the scatter plot of~$\mir$ in Figure~\ref{fig:symmetric-circle1}, starting
from the upper left.

\begin{enumerate}
  \item In the beginning nobody is vaccinated, and~$R_0$ is equal to~$2$.
  \item For small costs all strategies have similar
    efficiency. Zooming in  shows that the
    (numerically) optimal strategies split the available vaccine equally between four
    subpopulations that are separated from each other by two subpopulations. This
    corresponds to the ``one in~$3$'' meta-strategies path. As represented in
    Figure~\ref{fig:symmetric-circle2},~$\eta_A$ with outcome point~$A = (C(\eta_A),
    R_e(\eta_A))$ belongs to this path. In particular, note that disconnecting the graph
    is not Pareto optimal for~$12c = 3$ as the disconnecting ``one in~$4$'' strategy gives
    values~$R_e = \sqrt{2} \simeq 1.41$ opposed to the value~$R_e \simeq 1.37$ for the
    ``one in~$3$'' strategy with same cost. However, note that, in agreement with
    \cite[Proposition~6.5]{ddz-Re}, this disconnecting ``one in~$4$'' strategy is also not
    anti-Pareto optimal, since it performs better than the uniform strategy with the same
    cost.
   \item When~$12c = 4$ the circle has been split in four ``islands'' of two interacting
     subpopulations. There is a small interval of values of~$c$ for which it is
     (numerically) optimal to split the additional vaccine uniformly between the four
     ``islands'', and give it entirely to one subpopulation in each island: see point B
     and the associated strategy~$\eta_B$.
  \item Afterwards (see point C), it is in fact better to try and vaccinate all the (say)
    even numbered subpopulations. Therefore, the optimal vaccinations \emph{do not vary
    monotonously} with respect to the amount of available vaccine; in other words,
    distributing vaccine in a greedy way is not optimal. This also suggests that, even
    though the frontier is continuous (in the objective space~$(c,r)$), the set of optimal
    \emph{strategies} may not be connected: the ``one in two'' vaccination strategy of
    point C cannot be linked to ``no vaccination'' strategy by a continuous path of
    optimal strategies. In particular, the Pareto frontier cannot be greedily
    parametrized. The disconnectedness of the set of optimal strategies will be
    established rigorously in Section~\ref{sec:rank-2-reg} for another model.
  \item For~$12c = 6$, that is~$c = \cmir$ as stated in
    Proposition~\ref{prop:cmir-full-circle}, it is possible to vaccinate completely all
    the (say) odd numbered subpopulations, thereby disconnecting the graph completely.
    The infection cannot spread at all.
  \item Even though the problem is symmetric and all subpopulations play the same role,
    the proportional allocation of vaccine is far from optimal; on the contrary, the
    optimal allocations focus on some subpopulations.
\end{enumerate}

Using the same numerical algorithm, we have also computed the anti-Pareto frontier for
this model; see the dashed line in Figure~\ref{fig:pareto-sy}. Although we do not give a
formal proof, the anti-Pareto frontier seems to be perfectly given by the following greedy
parametrization:

\begin{enumerate}
  \item Distribute all the available vaccine supply to one group until it is completely
    immunized.
  \item Once this group is fully vaccinated, distribute the vaccine doses to one of its
    neighbour.
  \item Continue this procedure by vaccinating the neighbour of the last group that has
    been immunized.
  \item When there are only two groups left, split the vaccine equally between
    these two.
\end{enumerate}

\section{The kernel model}\label{sec:setting}

In order to get a finer description of the heterogeneity, we could divide the population
into a growing number of subgroups~$N \to \infty$. The recent advances in graph limits
theory \cite{lovasz_large_2012, backhausz_action_2020} justify describing the transmission
of the disease by a kernel defined on a probability space. We already used this type of
model in \cite{delmas_infinite-dimensional_2020,ddz-theo,ddz-Re}, in particular for an SIS
dynamics, see also \cite[Section~2]{ddz-Re} for other epidemic models. 

Let~$(\Omega, \mathscr{F}, \mu)$ be a probability space that represents the population:
the individuals have features labeled by~$\Omega$ and the infinitesimal size of the
population with feature~$x$ is given by $\mu(\mathrm{d}x)$. Let~$L^2(\mu)$ ($L^2$ for
short) be the space of real-valued measurable functions~$f$ defined on~$\Omega$ such
that~$\norm{f}_2=(\int_\Omega f^2\, \mathrm{d} \mu)^{1/2}$ is finite, where functions
which agree $\mu$-a.s.\ are identified. Let~$L^2_+=\{f\in L^2\, \colon\, f\geq 0\}$ be the
subset of non-negative functions of~$L^2$.  We define a \emph{kernel} on~$\Omega$ as
a~$\R_+$-valued measurable function defined on~$(\Omega^2, \mathscr{F}^{\otimes 2})$. We
will only consider kernels with finite double-norm on~$L^2$:
\begin{equation}\label{eq:2}
  \norm{\kk}_{2,2} =
  \left(
    \int_{\Omega \times \Omega} \kk(x,y)^2 \, \mu(\mathrm{d}x) \mu(\mathrm{d}y)
    \right)^{1/2}< +\infty.
\end{equation}
To a kernel~$\kk$  with finite double norm on~$L^2$, we associate the integral
operator~$T_\kk$ on~$L^2$ defined by:
\begin{equation}\label{eq:def-Tkk}
  T_\kk (g) (x) = \int_\Omega \kk(x,y) g(y)\,\mu(\mathrm{d}y)
  \quad \text{for } g\in L^2 \text{ and } x\in \Omega.
\end{equation}
The operator~$T_\kk$ is bounded, and its  operator norm~$\norm{T_\kk}_{L^2}$ satisfies:
\begin{equation}\label{eq:double-norm}
  \norm{ T_\kk }_{L^2} \leq \norm{\kk}_{2,2}.
\end{equation}
According to~\cite[p. 293]{grobler},~$T_\kk$ is actually compact (and even
Hilbert-Schmidt). A kernel is said to be symmetric if~$\kk(x,y) = \kk(y,x)$,
$\mu(\mathrm{d}x) \mu(\mathrm{d}y)$-almost surely. It is said to be \emph{irreducible} if
for all~$A \in \mathscr{F}$, we have:
\begin{equation}\label{eq:irr}
  \int_{A \times A^c} \kk(x,y) \, \mu(\mathrm{d}x) \mu(\mathrm{d}y) =0 \implies \mu(A) \in
  \{ 0,1 \}.
\end{equation}
If~$\kk$ is not irreducible, it is called \emph{reducible}.

By analogy with the discrete setting and also based on
\cite{delmas_infinite-dimensional_2020, ddz-Re}, we define the basic reproduction number
in this context thanks to the following formula:
\begin{equation}
  R_0 = \rho(T_\kk),
\end{equation}
where~$\rho$ stands for the spectral radius of an operator. According to the Krein-Rutman
theorem,~$R_0$ is an eigenvalue of~$T_\kk$. Besides, there exists left and right
eigenvectors associated to this eigenvalue in~$L^2_+$; such functions are called Perron
eigenfunctions. 

For~$f,g$ two non-negative bounded measurable functions defined on~$\Omega$ and~$\kk$ a
kernel on~$\Omega$ with finite double norm on~$L^2$, we denote by~$f\kk g$ the kernel
on~$\Omega$ defined by:
\begin{equation}\label{eq:def-fkg}
  (f\kk g)(x,y) = f(x)\, \kk(x,y) g(y).
\end{equation}
Since~$f$ and~$g$ are bounded, the kernel~$f \kk g$ has also a finite double norm
on~$L^2$.

Denote by~$\Delta$ the set of measurable functions defined on~$\Omega$ taking values
in~$[0, 1]$. A function~$\eta$ in~$\Delta$ represents a vaccination strategy: $\eta(x)$
represents the proportion of \textbf{non-vaccinated} individuals with feature~$x$. In
particular~$\eta=\un$ (the constant function equal to 1) corresponds
to the absence of vaccination
and~$\eta=\zero$ (the constant function equal to 0) corresponds to the whole population
being vaccinated. The uniform strategies are given by:
\[
  \etau=t \un\quad 
\]
for some~$t\in [0, 1]$, and we denote by~$ \cpu=\{ t\un\, \colon\, t\in [0,1]\}$ the set
of uniform strategies.  \medskip

The (uniform) cost of the vaccination strategy~$\eta \in \Delta$ is given by the total
proportion of vaccinated people, that is:
\begin{equation}
  \label{def:cost}
  C(\eta) = \int_\Omega (1 - \eta) \, \mathrm{d}\mu= 1- \int_\Omega \eta \, \mathrm{d}\mu.
\end{equation}
The measure~$\eta \, \mathrm{d} \mu$ corresponds to the \textit{effective population},
that is the individuals who effectively play a role in the dynamic of the epidemic. The
effective reproduction number is defined by:
\begin{equation}
  \label{eq:def-Re}
  R_e(\eta) = \rho(T_{\kk \eta}),
\end{equation}

We consider the weak topology on~$\Delta$, so that with a slight abuse of notation we
identify~$\Delta$ with $\{\eta\in L^2\, \colon\, 0\leq \eta\leq 1\}$. According to
\cite[Theorem~4.2]{ddz-theo}, the function~$R_e: \eta \mapsto R_e(\eta)$ is continuous
on~$\Delta$ equipped with the weak topology. The compactness of~$\Delta$ for this topology
implies the existence of solutions for Problems~\eqref{eq:prob-min-Re}
and~\eqref{eq:prob-max-Re}. We will conserve the same notation and definitions as in the
discrete setting for: the value functions~$\mir$ and~$\mar$, the minimal/maximal
costs~$\cmir$ and~$\cmar$, the various sets of strategies~$\cp$ and~$\cpa$, and
the various frontiers $\F$ and~$\AF$, see Equations~\eqref{eq:mir}-\eqref{eq:def-cpu} in
Section~\ref{sec:discrete-prob}.

We shall also use the following result from  \cite[Corollary~6.1]{ddz-Re} (recall that a
vaccination strategy is defined up the a.s.\ equality). 

\begin{lemma}\label{lem:k>0-c}
  Let~$\kk$ be a kernel on~$\Omega$ with finite double norm on~$L^2$ such that
  a.s.~$\kk>0$. Then, we have~$\cmar=0$,~$\cmir=1$ and the strategy~$\un$ (resp.~$\zero$)
  is the only Pareto optimal as well as the only anti-Pareto optimal strategy with
  cost~$c=0$ (resp.~$c=1$).
\end{lemma}
 
\begin{example}[Discrete and continuous representations of a metapopulation model]
  We recall the natural correspondence between metapopulation models (discrete models) and
  kernel models (continuous models) from \cite[Section~7.4.1]{ddz-theo}. Consider a
  metapopulation model with~$N$ groups given by a finite set~$\Omega_{\mathrm{d}} = \{ 1,
  2, \ldots, N \}$ equipped with a probability measure~$\mu_{\mathrm{d}}$ giving the
  relative size of each group and a next generation matrix~$K=(K_{ij}, \, i,j\in
  \Omega_{\mathrm{d}})$. The corresponding discrete kernel~$\kk_{\mathrm{d}}$
  on~$\Omega_{\mathrm{d}}$ is defined by:
  \begin{equation}\label{eq:next-kernel}
    K_{ij} = \kk_{\mathrm{d}}(i,j) \mu_j \quad
    \text{where}\quad \mu_i = \mu_{\mathrm{d}}(\{i\}). 
  \end{equation}
  Then, the matrix~$K\cdot \Diag(\eta)$ is the matrix representation of the
  endomorphism~$T_{\kk_{\mathrm{d}} \eta}$ in the canonical basis of~$\R^N$.

  Following \cite{ddz-theo}, we can also consider a continuous representation on the state
  space~$\Omega_{\mathrm{c}} = [0,1)$ equipped with the Lebesgue
  measure~$\mu_{\mathrm{c}}$. Let~$I_1 = [0, \mu_1)$, $I_2 = [\mu_1, \mu_1 + \mu_2)$,
  \ldots, $I_{N} = [1 - \mu_{N}, 1)$, so that the intervals~$(I_n,\, 1 \leq n \leq  N)$
  form a partition of~$\Omega$. Now define the kernel:
  \begin{equation}
    \kk_{\mathrm{c}}(x,y) = \sum_{1\leq i,j \leq  N} \kk_{\mathrm{d}}(i,j) \ind{I_i
    \times I_j}(x,y).
  \end{equation}
  Denote by~$R_e^\mathrm{d}$ and~$R_e^\mathrm{c}$ the effective reproduction number in the
  discrete and continuous representation models. In the same manner,
  the uniform cost in each model is denoted by~$C^\mathrm{d}$ and~$C^\mathrm{c}$.
  According to~\cite{ddz-theo}, these functions are linked through the following relation:
  \[ R_e^\mathrm{d}(\eta^\mathrm{d}) = R_e^\mathrm{c}\left( \eta^\mathrm{c} \right), \quad
  \text{and} \quad C^\mathrm{d}(\eta^\mathrm{d}) = C^\mathrm{c}(\eta^\mathrm{c}), \]
  for all~$\eta^\mathrm{d} \, \colon \, \Omega_\mathrm{d} \to [0,1]$ and~$\eta^\mathrm{c}
  \, \colon \, \Omega_\mathrm{c} \to [0,1]$ such that:
  \[
    \eta^\mathrm{d}(i) =
    \frac{1}{\mu_i} \int_{I_i} \eta^\mathrm{c}\,
    \mathrm{d}\mu_\mathrm{c} \quad\text{for all}\quad 
    i \in \Omega_{\mathrm{d}} .
  \]
  Let us recall that the Pareto and anti-Pareto frontiers for the two models are the same.

  In Figure~\ref{fig:equiv}, we have plotted the kernels of the continuous models
  associated to the asymmetric and symmetric circles models from
  Sections~\ref{sec:asym-circle} and \ref{sec:fully-sym-circle}.
\end{example}

\begin{figure}
  \begin{subfigure}[c]{.5\textwidth}
    \centering
    \includegraphics[page=3]{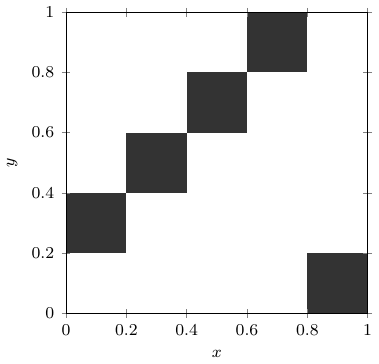}
    \caption{The asymmetric circle.}\label{fig:kass} 
  \end{subfigure}%
  \begin{subfigure}[c]{.5\textwidth}
    \centering
    \includegraphics[page=3]{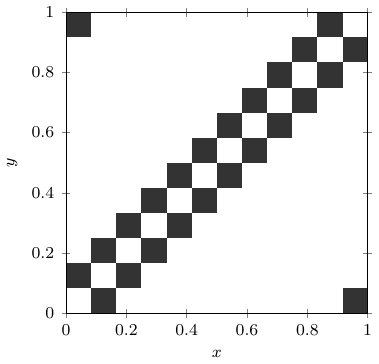}
    \caption{The symmetric circle.}\label{fig:ass}
  \end{subfigure}
  \caption{Kernels~$\kk_\mathrm{c}$ (equal to 0 in the white zone and to
    1 in the black zone) on~$\Omega_\mathrm{c}=[0, 1)$ and~$\mu_\mathrm{c}$ the
    Lebesgue measure of the continuous model associated to 
    discrete metapopulation models.}
    \label{fig:equiv}
\end{figure}

\section{Assortative versus disassortative mixing}\label{sec:ass-disass}

\subsection{Motivation}

We consider a population divided into an at most countable number of groups. Individuals
within the same group interact with intensity~$a$ and individuals in different groups
interact with intensity~$b$. Hence, the model is entirely determined by the
coefficients~$a$ and~$b$ and the size of the different groups. This simple model allows to
study the effect of assortativity, that is, the tendency for individuals to connect with
individuals belonging to their own subgroup. The mixing pattern is called
\emph{assortative} (higher interaction in the same subgroup) if~$a>b$, and
\emph{disassortative} (lower interaction in the same subgroup) when~$b > a$. Our results
illustrate how different the optimal vaccination strategies can be between assortative and
disassortative models, an effect that was previously studied by Galeotti and
Rogers~\cite{StrategicImmunGaleot2013} in a population composed of two groups.

\medskip

When the population is equally split in a finite number of subgroups,  and~$a$ is equal to~$0$, the next-generation matrix
of this model corresponds, up to a multiplicative constant, to the adjacency matrix of a
complete multipartite graph. Recall that an~$m$-partite graph is a graph that can be
colored with~$m$ different colors, so that all edges have their two endpoints colored
differently.  When $m = 2$ these are the so-called bipartite graphs. A complete
multipartite graph is a~$m$-partite graph (for some~$m\in \N^*$) in which there is an edge
between every pair of vertices from different colors.

The complete multipartite graphs have interesting spectral properties. Indeed, Smith
\cite{smith1970some} showed that a graph with at least one edge has its spectral radius as
its only positive eigenvalue if and only if its non-isolated vertices induce a complete
multipartite graph. In \cite{esser_spectrum_1980}, Esser and Harary proved that two
complete~$m$-partite graphs with the same number of nodes are isomorphic if and only if
they have the same spectral radius. More precisely, they obtained a comparison of the
spectral radii of two complete~$m$-partite graphs by comparing the sizes of the sets in
their partitions through majorization; see~\cite[Lemma~3]{esser_spectrum_1980}.

\medskip

The goal of this section is to generalize and complete these results and give a full
picture of the Pareto and anti-Pareto frontiers for the assortative and the disassortative
models.

\subsection{Spectrum and convexity}

We will use an integer intervals notation to represent the considered kernels. For~$i,
j \in \N \cup \{ + \infty \}$, we set~$\lb i,j\rb$ (resp.~$\lb i, j\lb$) for~$[i,
j]\cap (\N \cup \{ + \infty \})$ (resp.~$[i, j) \cap \N$). Let~$N \in \lb 2, + 
\infty \rb$ and~$\Omega = \lb 1, N\rb$ if $N$ is finite and $\Omega =
\lb 1, +\infty \lb$ otherwise. 
 The set~$\Omega$ is endowed with the
discrete~$\sigma$-algebra~$\cf=\cp(\Omega)$ and a probability measure~$\mu$. To 
simplify the notations, we write~$\mu_i$ for~$\mu(\{ i \})$ and
$f_i=f(i)$ for a function~$f$ defined on~$\Omega$. Without loss of generality, we can
suppose that  $\mu_i \geq \mu_j >0$ for all~$ i \leq j$ elements of $\Omega$. We consider the
kernel~$\kk$ defined for~$i,j \in \Omega$ by:
\begin{equation}\label{eq:kernel_mul} 
  \kk(i,j) =
  \begin{cases}
    a \quad \textnormal{if} \quad i=j, \\
    b \quad \textnormal{otherwise},
  \end{cases}
\end{equation}
where~$a$ and~$b$ are two non-negative real numbers.

\medskip  

If~$b = 0$, then the kernel is reducible, see \cite[Section~7]{ddz-Re}, and the effective
reproduction number is given by the following formula: $R_e(\eta) = a \max_{i\in \Omega}
\eta_i \,\mu_i$, for all $\eta=(\eta_i, i\in \Omega)\in \Delta$. This is sufficient to
treat this case and we have~$\cmar = 1 - \mu_1$.

\medskip

From now on, we assume that~$b > 0$. 
The next two results describe the spectrum of~$T_\kk$ in both the assortative and
disassortative case. Notice the spectrum of~$T_\kk$ is real as~$\kk$ is symmetric. 
Recall that $R_0=\rho(T_\kk)$.

\begin{proposition}[Convexity/concavity of~$R_e$]\label{prop:disass-ass-eig}
 Let~$\kk$ be given by \eqref{eq:kernel_mul}, with~$b> 0$ and~$a \geq
 0$.
  \begin{enumerate}[(i)]
    \item\label{pt:ass-Re} \textbf{Assortative model.}
If~$a\geq b> 0$, then the
  operator~$T_\kk$ is positive semi-definite and the function~$R_e$ is convex.
    \item\label{pt:disass-Re} \textbf{Disassortative model.}
    If~$b\geq a\geq 0$ and~$b>0$, then~$R_0$
  is the only positive eigenvalue of~$T_\kk$, and it has multiplicity one. Furthermore,
  the function~$R_e$ is concave.
  \end{enumerate}
\end{proposition}

 In the following proof, we shall  consider the symmetric matrix~$M_n$
of size~$n \times n$, with~$n\in \N^*$, given by:
\[
  M_n(i,j) =
  \begin{cases}
    a \quad \textnormal{if} \quad i=j, \\
    b \quad \textnormal{otherwise}.
  \end{cases}
\]
The matrix~$M_n$ is the sum of~$b$ times the all-ones matrix and~$a - b$ times the
identity matrix. Thus,~$M_n$ has two distinct eigenvalues:~$nb + a$ with multiplicity~$1$
and~$a-b$ with multiplicity~$n-1$.

\begin{proof}
   We first prove~\ref{pt:ass-Re}.   For any~$g\in L^2$, we have:
  \begin{equation*}
    \int_{\Omega \times \Omega} g(x) \kk(x,y) g(y) \, \mu(\mathrm{d}x) \mu(\mathrm{d} y) =
    a \sum_{i \in \Omega} g_i^2 \mu_i^2 + b \sum_{i \neq
    j} g_i g_j \, \mu_i \mu_j \geq b \norm{g}_2^2.
  \end{equation*}
  This implies that~$T_\kk$ is positive semi-definite. Thus, as~$\kk$ is symmetric, the
  function~$R_e$ is convex, thanks to
  \cite[Theorem~5.5]{ddz-Re}. \medskip

We now prove~\ref{pt:disass-Re}. 
   We give a direct proof when~$N$ is finite, and use an approximation procedure
  for~$N=\infty$. We first assume that~$N$ is finite. For~$n \leq  N$, let~$v_n =
  \ind{\lb 1, n \rb}$ and set~$T_n = T_{v_n \kk v_n}$. The non-null eigenvalues
  of~$T_n$ (with their multiplicity) are the eigenvalues of the matrix~$M_n \cdot
  \Diag_n(\mu)$, where~$\Diag_n (\mu)$ is the diagonal~$n\times n$-matrix
  with~$(\mu_1, \ldots, \mu_n)$ on the diagonal. Thanks to
  \cite[Theorem~1.3.22]{horn2012matrix}, these are also the eigenvalues of the
  matrix~$Q_n=\Diag_n (\mu)^{1/2} \cdot M_n \cdot \Diag_n(\mu)^{1/2}$. By Sylvester's law
  of inertia \cite[Theorem~4.5.8]{horn2012matrix}, the matrix~$Q_n$ has the same signature
  as the symmetric matrix~$M_n$. In particular, since we have supposed~$a-b\leq 0$,~$M_n$
  has only one positive eigenvalue. Thus~$Q_n$ has only one positive eigenvalue: thanks to
  the Perron-Frobenius theory, it is its spectral radius. This concludes the proof
  when~$N$ is finite by choosing~$n=N$.

  If~$N=\infty$, we consider the limit~$n\to N$. Since:
  \[
    \lim_{n \to \infty } \norm{\kk - v_n \kk v_n}_{2,2} = 0,
  \]
 the spectrum of~$T_n$ converges to the spectrum of~$T_\kk$, with respect to the Hausdorff
  distance, and the multiplicity on the non-zero eigenvalues also converge, see
  \cite[Corollary~3.2]{ddz-Re}. This shows that~$\rho(T_\kk)$ is the only positive
  eigenvalue of~$T_\kk$, and it has multiplicity one. Since~$\kk$ is symmetric, we deduce
  the concavity of the function~$R_e$ from \cite[Theorem~5.5]{ddz-Re}.
\end{proof}

\subsection{Explicit description of the Pareto and anti-Pareto frontiers}

For~$c\in[0,1]$, we define an ``\emph{horizontal vaccination}'' $\ho(c)\in\Delta$ with
cost~$c$ in the following manner. Rather than defining directly the proportion of
non-vaccinated people in each class, it will be convenient to define first the resulting
effective population size, which will be denoted by $\xi$. For all~$\alpha\in[0,\mu_1]$,
let $\xi^\mathrm{h}(\alpha)\in \Delta$ be defined by:
\begin{equation}\label{eq:xih}
  \xi_i^\mathrm{h}(\alpha) = \min(\alpha, \mu_i), \quad i \in \Omega.
\end{equation}
For all~$i \in \Omega$,~$\xi_i^\mathrm{h}(\alpha)$ is a non-decreasing and continuous
function of~$\alpha$. The map~$\alpha\mapsto \sum_i \xi_i^\mathrm{h}(\alpha)$ is
continuous and increasing from~$[0,\mu_1]$ to~$[0,1]$, so for any~$c\in[0,1]$, there
exists a unique~$\alpha_c$ such that~$\sum_i \xi_i^\mathrm{h}(\alpha_c) = 1-c$. We then
define the horizontal vaccination profile~$\ho(c)\in \Delta$ by:
\begin{equation}\label{eq:ho}
  \ho_i(c) = \xi_i^\mathrm{h}(\alpha_c)/ \mu_i, \quad i \in \Omega.
\end{equation}
In words, it consists in vaccinating in  such a way that the quantity of
the non-vaccinated individuals~$\xi_i^\mathrm{h}  = \eta_i\mu_i$ in each
subpopulation is always less than the ``horizontal'' threshold~$\alpha$:
see  Figure~\ref{fig:p1}. The cost of the vaccination strategy $
\ho(c)$ is indeed $c$. 
Note  that  $\ho(0) =  \un$ (no  vaccination),
whereas~$\ho(1)   =   \zero$   (full    vaccination),   and   that   the
path~$c \mapsto \ho(c)$ is greedy. We denote its range by~$\hoPath$.

\medskip

For~$c\in[0,1]$, we define similarly a ``\emph{vertical vaccination}''~$\ver(c)\in\Delta$
with cost~$c$.  First let us define for~$\beta\in [0,N]$:
\begin{equation}\label{eq:xiv}
  \xi^{\mathrm{v}}_i(\beta) = \mu_i \cdot \min(1, (\beta+1 - i)_+), \quad i \in \Omega.
\end{equation}
The map~$\beta\mapsto \sum_i \xi^{\mathrm{v}}_i(\beta)$ is increasing  and continuous
from~$[0,N]$ to~$[0,1]$, so for any~$c\in[0,1]$  there exists a unique~$\beta_c$ such
that~$\sum_i \xi^{\mathrm{v}}_i(\beta_c) = 1-c$. We then define the vertical vaccine
profile~$\ver(c)$ by:
\begin{equation}\label{eq:ver}
  \ver_i(c) = \xi^{\mathrm{v}}_i(\beta_c)/\mu_i, \quad i \in \Omega.
\end{equation}
In words, if~$\lceil \beta\rceil = \ell$, this consists in vaccinating all 
subpopulations~$j$ for~$j>\ell$, and a fraction~$\lceil \beta\rceil-\beta$ of the 
subpopulation~$\ell$, see Figure~\ref{fig:p2} for a graphical representation.
The cost of the vaccination strategy $
\ver(c)$ is by construction equal to $c$.

For all~$i \in \Omega$,~$\ver_i(c)$ is a non-increasing and continuous function of~$c$.
Just as in the horizontal case, we have~$\ver(0)=\un$ (no vaccination), $\ver(1)=\zero$
(full vaccination), and the path~$c \mapsto \ver(\beta(c))$ is also greedy. We denote its
range by~$\verPath$.

\begin{figure}
  \begin{subfigure}[T]{.5\textwidth}
    \centering
    \includegraphics[page=2]{parametrization}
    \caption{Representation of the greedy path~$\hoPath$.}\label{fig:p1}
  \end{subfigure}%
  \begin{subfigure}[T]{.5\textwidth}
    \centering
    \includegraphics[page=1]{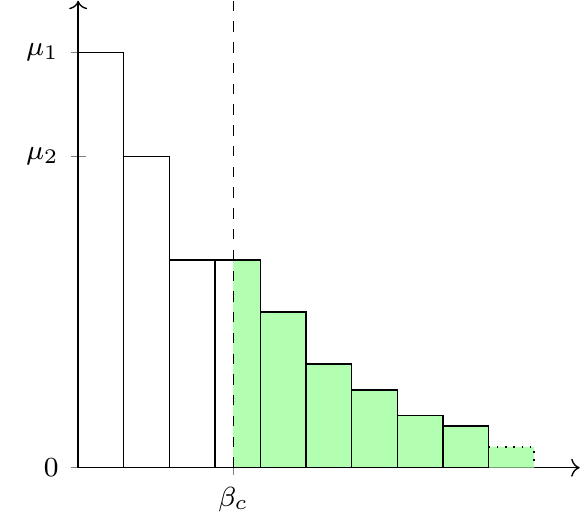}
    \caption{Representation of the greedy path~$\verPath$.}\label{fig:p2}
  \end{subfigure}%
  \caption{Greedy parametrization of the (anti-)Pareto front. The bar plot represents the
  measure~$\mu$. The proportion of green in each bar correspond to the proportion of
  vaccinated individuals in each subpopulation.}\label{fig:p}
\end{figure}

\medskip

These two paths give a greedy parametrization of the Pareto and anti-Pareto frontiers for
the assortative and disassortative models:  more explicitly, we have the following result,
whose proof can be found in~Section~\ref{sec:proof-assortative}. 

\begin{theorem}[Assortative vs disassortative]
  \label{th:dis-assortative}
  Let~$\kk$ be given by \eqref{eq:kernel_mul}, with~$b> 0$ and~$a \geq 0$.
  \begin{enumerate}[(i)]
    \item\label{pt:ass}\textbf{Assortative model.} If~$a\geq b>0$, then~$\verPath$
      and~$\hoPath$ are greedy parametrizations of the anti-Pareto and Pareto frontiers
      respectively.
    \item\label{pt:disass}\textbf{Disassortative model.} If~$b\geq a>0$, then~$\verPath$
      and~$\hoPath$ are greedy parametrizations of the Pareto and anti-Pareto frontiers
      respectively.
    \item\label{pt:multi}\textbf{Complete multipartite model.} If~$a =0$ and~$b > 0$,
      then~$\hoPath$ is a greedy parametrization  of the anti-Pareto frontier and the
      subset of strategies~$\eta \in \verPath$ such that~$C(\eta) \leq 1 - \mu_0$ is a
      greedy parametrization of the Pareto frontier. In particular, we have~$\cmir = 1 -
      \mu_1$ and~$\cmar=0$.
  \end{enumerate}
\end{theorem}

Notice that~$\cmar=0$ and~$\cmir=1$ in cases (i) and (ii) as~$\kk$ is
positive thanks to Lemma~\ref{lem:k>0-c}. 

\begin{remark}[Highest Degree vaccination]
  The effective degree function of a symmetric kernel~$\kk$ at~$\eta \in \Delta$ is the
  function~$\mathsf{deg}_\eta$  defined  on~$\Omega$ by:
  \begin{equation}\label{eq:def-degh1}
    \mathsf{deg}_\eta(x) = \int_\Omega \kk(x,y) \eta(y) \, \mu(\mathrm{d}y).
  \end{equation}
  When~$\eta = \un$, it is simply called the degree of~$\kk$ and is denoted
  by~$\mathsf{deg}$. In our model, the effective degree of the subgroup~$i$ is given by
  \begin{equation}\label{eq:def-degh}
    \mathsf{deg}_\eta(i) = a\eta_ i \mu_i + b \sum_{\ell\neq i} \eta_\ell
    \mu_\ell ,
  \end{equation}
  and    thus    the   degree    of    the    subgroup~$i$   is    given
  by~$\mathsf{deg}(i)  =   (a-b)  \mu_i  +  b$.    As~$\mu_i\geq  \mu_j$
  for~$ i< j$  elements of~$\Omega$, we deduce that  the degree function
  in   monotone:   non-increasing   in   the   assortative   model   and
  non-decreasing  in  the  disassortative  model.  The  group  with  the
  highest  degree therefore  corresponds  to the  largest  group in  the
  assortative  model  and the  smallest  group  (if  it exists)  in  the
  disassortative model.

  Consider the assortative model where all the groups have different size,
  \textit{i.e.},~$\mu_1 > \mu_2 > \ldots$ Following the parametrization~$c\mapsto \ho(c)$,
  starting from~$c=0$, will first decrease the effective size of the group~$1$ (the group
  with the highest degree) until it reaches the effective degree of group~$2$ (with the
  second highest degree). Once these two groups share the same effective degree which
  corresponds to reaching~$\mu_1 \ho_1=\mu_2$, they are vaccinated uniformly (that is,
  ensuring that they keep the same effective degree: using \eqref{eq:def-degh} this
  corresponds to ~$\mu_1 \ho_1=\mu_2 \ho_2$) until their effective degree is equal to the
  third highest degree, and so on and so forth.

  In the disassortative model, the function~$\mathsf{deg}_\eta$ remains (strictly)
  increasing when the vaccination strategies in~$\verPath$ are applied. In particular,
  if~$\mu_1 > \mu_2 > \ldots$, then the optimal strategies prioritize the groups with the
  higher effective degree until they are completely immunized. If multiple groups share
  the same degree, it is optimal to give all available doses to one group. \medskip
  
  In conclusion, in both models, the optimal vaccination consists in vaccinating the
  groups with the highest effective degree in priority if this group is unique. But if
  multiple groups share the same degree (\textit{i.e.}, have the same size), the optimal
  strategies differ between the assortative and the disassortative case. In the
  assortative case, groups with the same size must be vaccinated uniformly while in the
  disassortive case, all the vaccine doses shall be given to one group until it is
  completely vaccinated.
\end{remark}

\begin{example}[Group sizes following a dyadic distribution]
  Let~$N = \infty$,~$\Omega=\N^*$ and~$\mu_i = 2^{-i}$ for all~$i \in \Omega$. Following
  \cite[Section~7.4.1]{ddz-theo}, we will couple this discrete model with a continuum
  model for a better visualization on the figures. Let~$\Omega_c = [0,1)$ be equipped with
  the Borel~$\sigma$-field~$\mathscr{F}_c$ and the Lebesgue measure~$\mu_c$.  The
  set~$\Omega_c$ is partitionned into a countable number of intervals~$I_i = [ 1 - 2^{-i+1},
  1 - 2^{-i})$, for~$i \in \N^*$, so that~$\mu_c(I_i) = \mu_i$. The kernel of the
  continuous model corresponding to~$\kk$ in \eqref{eq:kernel_mul} is given by:
  \begin{equation}\label{eq:def:kc}
    \kk_c = (a-b) \sum_{i\in \N^*} \ind{I_i \times I_i} + b \un.
  \end{equation}
  The kernel~$\kk_c$ is plotted in Figures~\ref{fig:kernel-ass}, \ref{fig:kernel-disass}
  and~\ref{fig:kernel-multi} for different values of~$a$ and~$b$ corresponding
  respectively to the assortative, the disassortative and the complete multipartite case
  corresponding to points~\ref{pt:ass}, \ref{pt:disass} and \ref{pt:multi} of
  Theorem~\ref{th:dis-assortative} respectively. Their respective Pareto and anti-Pareto
  frontiers are plotted in Figures~\ref{fig:pareto-ass}, \ref{fig:pareto-disass} and
  \ref{fig:pareto-multi}, using a finite-dimensional approximation of the kernel~$\kk$ and
  the power iteration method. In Figure~\ref{fig:pareto-multi}, the value of~$\cmir$ is
  equal to~$1-\mu_1=1/2$. With this continuous representation of the population, the
  set~$\verPath$ corresponds to the strategies of the form~$\ind{[0, t)}$ for~$t
  \in [0,1]$.

  Notice that the Pareto frontier in the assortative case is convex. This is consistent
  with \cite[Proposition~6.6]{ddz-theo} since the cost function is affine and~$R_e$ is
  convex when~$a \geq b$; see  Proposition~\ref{prop:disass-ass-eig}~\ref{pt:ass-Re}. In the same
  manner, the anti-Pareto frontier in the disassortative and the multipartite cases is
  concave. Once again, this is  consistent with \cite[Proposition~6.6]{ddz-theo} since the
  cost function is affine and~$R_e$ is concave when~$b \geq a$; see
  Proposition~\ref{prop:disass-ass-eig}~\ref{pt:disass-Re}.
\end{example}

\begin{figure}
  \begin{subfigure}[T]{.5\textwidth}
    \centering
    \includegraphics[page=1]{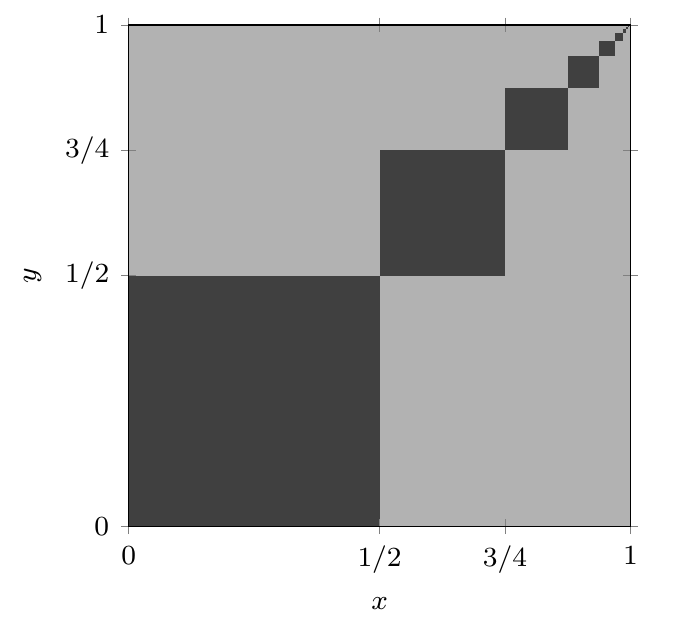}
    \caption{Grayplot of the kernel~$\kk_c$ from \eqref{eq:def:kc} on~$\Omega_c = [0,1)$ with~$\kk_c=a=5$
      on the dark gray zone, and~$\kk_c=b=2$ on the light gray zone.}
    \label{fig:kernel-ass} 
  \end{subfigure}%
  \begin{subfigure}[T]{.5\textwidth}
    \centering
    \includegraphics[page=2]{assortative}
    \caption{Solid line: the Pareto frontier~$\F$; dashed line: the anti-Pareto frontier~$\AF$;
      dotted line: path of the uniform strategies; blue
    region: all possible outcomes~$\FF$.}
    \label{fig:pareto-ass}
  \end{subfigure}%
  \caption{An assortative model.}\label{fig:optim-ass}
\end{figure}

\begin{figure}
  \begin{subfigure}[T]{.5\textwidth}
    \centering
    \includegraphics[page=1]{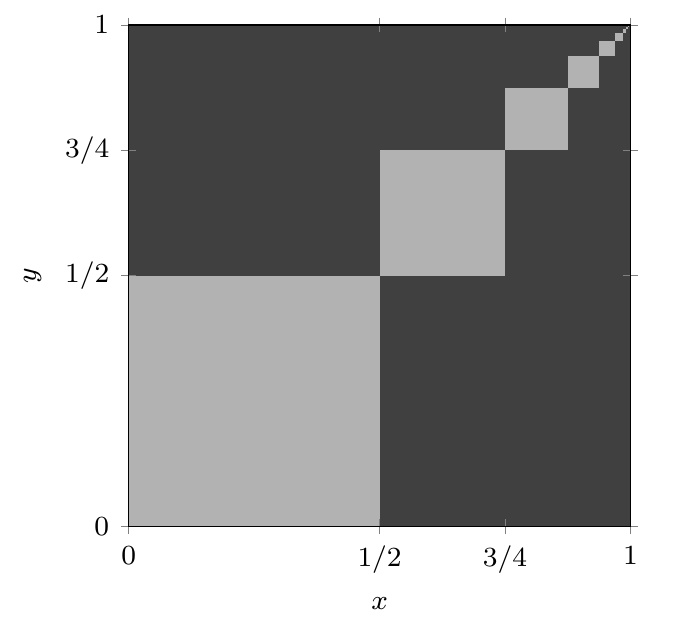}
    \caption{Grayplot of the kernel~$\kk_c$ from \eqref{eq:def:kc}
      on~$\Omega_c = [0,1)$ with~$\kk_c=a=2$ on the light gray
    zone, and~$\kk_c=b=5$ on the dark gray
    zone.}
  \label{fig:kernel-disass} 
  \end{subfigure}%
  \begin{subfigure}[T]{.5\textwidth}
    \centering
    \includegraphics[page=2]{disassortative}
    \caption{Solid line: the Pareto frontier~$\F$; dashed line: the anti-Pareto frontier~$\AF$;
      dotted line: path of the uniform strategies; blue
    region: all possible outcomes~$\FF$.}\label{fig:pareto-disass}
  \end{subfigure}%
  \caption{A disassortative model.}\label{fig:optim-disass}
\end{figure}

\begin{figure}
  \begin{subfigure}[T]{.5\textwidth}
    \centering
    \includegraphics[page=1]{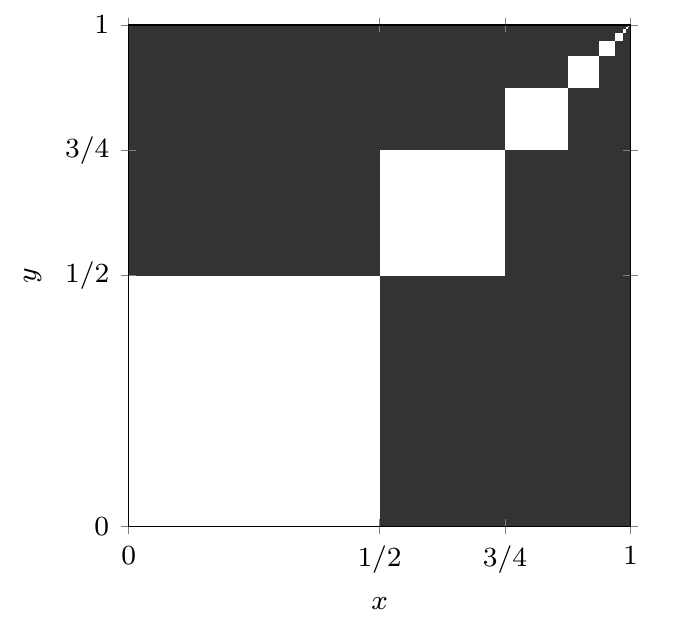}
    \caption{Grayplot of the kernel~$\kk_c$ from \eqref{eq:def:kc} on~$\Omega_c = [0,1)$
      with~$\kk_c=a=0$ (on the white zone) and~$\kk_c=b=6$ (on the black zone).}
    \label{fig:kernel-multi} 
  \end{subfigure}%
  \begin{subfigure}[T]{.5\textwidth}
    \centering
    \includegraphics[page=2]{multipartite}
    \caption{Solid line: the Pareto frontier~$\F$; dashed line: the anti-Pareto frontier~$\AF$;
      dotted line: path of the uniform strategies; blue
    region: all possible outcomes~$\FF$.}
    \label{fig:pareto-multi}
  \end{subfigure}%
  \caption{An example of the complete multipartite model.}\label{fig:optim-multi}
\end{figure}

\subsection{Proof of Theorem~\ref{th:dis-assortative}}\label{sec:proof-assortative}

After recalling known facts of majorization theory, we first consider the finite dimension
models, and then the general case by an approximation argument. 

\subsubsection{Majorization}

In this section, we recall briefly some definitions and results from majorization theory,
and refer to~\cite{MajorizationAnArnold1987, InequalitiesTMarsha2011} for an extensive
treatment of this topic. 

Let~$n   \geq    1$   and~$\xi,\chi    \in   \R^{n}_+$.    We   denote
by~$\xi^\downarrow$   and  $\chi^\downarrow$   their  respective   order
statistics,  that   is  the   vectors  in~$\R_+^{n}$  with   the  same
components,  but  sorted  in  descending order.  We  say  that~$\xi$  is
\emph{majorized} by~$\chi$, and write~$\xi \prec \chi$, if:
\begin{equation}\label{eq:majorization1}
  \sum_{j=1}^{i} \xi_j^\downarrow \leq \sum_{j=1}^i \chi_j^\downarrow
  \quad\text{for all~$i\in \{1, \ldots, n\}$, and}\quad
  \sum_{j=1}^{n} \xi_j = \sum_{j=1}^{n} \chi_j.
\end{equation}
Among the various characterizations of majorization, we will use the following by Hardy,
Littlewood and P\'{o}lya; see~\cite[Proposition~I.4.B.3]{InequalitiesTMarsha2011}:
\begin{equation}\label{eq:majorization_(x-a)+}
  \xi \prec \chi \iff
    \sum_i (\xi_i - t)_+ \leq \sum_i (\chi_i - t)_+ \quad\text{for all}\quad t\in \R_+,
\end{equation}
where~$u_+  =   \max(u,  0)$,   for  all~$u   \in  \R$.   A  real-valued
function~$\Theta$ defined on~$\R_+^{n}$  is called \emph{Schur-convex}
if    it   is    non-decreasing    with    respect   to~$\prec$,    that
is,~$\xi   \prec  \chi$   implies~$\Theta(\xi)  \leq   \Theta(\chi)$.  A
function~$\Theta$  is  called   \emph{Schur-concave}  if~$(-\Theta)$  is
Schur-convex.

\subsubsection{Shur convexity and concavity of the spectral radius in finite dimension}

We define the function~$\Theta_n$ on~$\R_+^{n}$ by:
\begin{equation*}
  \Theta_n(\xi) = \rho(M_n\cdot \Diag(\xi)),
\end{equation*}
where~$\Diag(\xi)$ is the diagonal~$n\times n$-matrix with~$\xi$ on the diagonal.
By construction, for~$\eta=(\eta_1, \ldots, \eta_n, 0, \ldots)$, we have:
\begin{equation}\label{eq:ch-variable}
  R_e(\eta) = \Theta_n(\eta_1 \mu_1, \ldots, \eta_n \mu_n).
\end{equation}
The key property below will allow us to identify the optimizers.

\begin{lemma}[Schur-concavity and Schur-convexity]\label{lem:q=schur-cave}
  Let~$b >0$ and~$a \geq 0$. The function~$\Theta_n$ is Schur-convex if~$a \geq b$, and
  Schur-concave if~$a \leq b$. 
\end{lemma}

\begin{proof}
  Let us consider the disassortative case where~$a \leq b$. By a classical result of
  majorization theory \cite[Proposition~I.3.C.2.]{InequalitiesTMarsha2011}, it is enough
  to show that~$\Theta_n$ is symmetric and concave. 

  To prove that~$\Theta_n$ is symmetric, consider~$\sigma$ be a permutation of~$\{1, 2,
  \ldots, n \}$ and~$P_\sigma$ the associated permutation matrix of size~$n \times
  n$. Since~$P_\sigma M_n P_\sigma^{-1} = M_n$, we deduce that~$\Theta_n(\xi_{\sigma})
  = \Theta_n(\xi)$, where~$\xi_\sigma$ is the~$\sigma$-permutation of~$\xi \in
  \R_+^{n}$. Thus~$\Theta_n$ is symmetric.

  We now prove that~$\Theta_n$ is concave on~$\R_+^{n}$. Since~$R_e$ is concave thanks
  to Proposition~\ref{prop:disass-ass-eig}~\ref{pt:disass-Re}, we deduce from \eqref{eq:ch-variable}, that the
  function~$\Theta_n$ is concave on~$[0, \mu_1] \times \ldots \times [0, \mu_n]$.
  Since~$\Theta_n$ is homogeneous, it is actually concave on the whole domain~$\R_+^n$.
  This concludes the proof when~$a \leq b$.

  The proof is the same for the assortative case~$a \geq b$, replacing the reference to
  Proposition~\ref{prop:disass-ass-eig}~\ref{pt:disass-Re}  by
  \ref{pt:ass-Re}. 
\end{proof}

\subsubsection{Extreme vaccinations for fixed cost}

Let us show that the horizontal and vertical vaccinations give extreme points for the
preorder~$\prec$ on finite sets, when the quantity of vaccine is fixed. Recall that
$\xi^\mathrm{h}$ and~$\xi^\mathrm{v}$ are defined in~\eqref{eq:xih} and~\eqref{eq:xiv}
respectively.

\begin{proposition}[Extreme vaccinations]\label{prop:extreme}
  Let~$n\in\Omega$,~$\beta\in[0,n)$ and~$\alpha\in[0, \mu_1]$. Let
  $\xi^{\mathrm{v},n} = (\xi^\mathrm{v}_1(\beta), \ldots,\xi^{\mathrm{v}}_n(\beta))$, and
  $\xi^{\mathrm{h},n} = (\xi^{\mathrm{h}}_1(\alpha), \ldots, \xi^{\mathrm{h}}_n(\alpha))$.
  For any~$\xi = (\xi_1,\ldots, \xi_n)\in [0,\mu_1]\times\cdots \times [0,\mu_n]$, we
  have:
  \begin{equation*}
    \left( \sum_{i=1}^n \xi_i = \sum_{i=1}^n \xi^{\mathrm{v},n}_i \right) \implies
    \xi \prec\xi^{\mathrm{v},n}, \quad \text{and} \quad
    \left( \sum_{i=1}^n \xi_i = \sum_{i=1}^n \xi^{\mathrm{h},n}_i \right) \implies
    \xi^{\mathrm{h},n} \prec\xi.
  \end{equation*}
\end{proposition}

\begin{proof}
  Let~$\xi \in [0,\mu_1]\times\cdots \times [0,\mu_n]$ be such that~$\sum_{i=1}^n \xi_i =
  \sum_{i=1}^n \xi^{\mathrm{v},n}_i$. The reordered vector~$\xi^\downarrow$ clearly
  satisfies the same conditions, so without loss of generality we may assume that~$\xi$ is
  sorted in descending order. Using Equation~\eqref{eq:xiv}, we get:
  \[
    \sum_{i=1}^\ell \xi_i \leq \sum_{i=1}^\ell \mu_i = \sum_{i=1}^\ell \xi^{\mathrm{v},n}_i,
    \quad \text{for} \quad 1 \leq \ell \leq  \floor{\beta}.
  \]
  We also have:
  \[
    \sum_{i=1}^\ell \xi_i \leq \sum_{i=1}^n \xi_i = \sum_{i=1}^n
    \xi^{\mathrm{v},n}_i = \sum_{i=1}^\ell \xi^{\mathrm{v},n}_i, \quad
    \text{for} \quad \ell > \floor{\beta}. 
  \] 
  Therefore, we get~$\xi\prec \xi^{\mathrm{v},n}$, by the definition of~$\prec$.

  \medskip

  Similarly, let~$\xi \in [0,\mu_1]\times\cdots \times [0,\mu_n]$ be such
  that~$\sum_{i=1}^n \xi_i = \sum_{i=1}^n
  \xi^{\mathrm{h},n}_i$. If~$t\geq \alpha$ then:
  \[
    \sum_i( \xi^{\mathrm{h},n}_i - t)_+ = 0 \leq \sum_i (\xi_i-t)_+, \]
  while if~$t\in
  [0, \alpha)$, using the fact that~$\sum_{i=1}^n\xi_i = \sum
  _{i=1}^n\xi^{\mathrm{h},n}_i$, the expression~$\xi_i^{\mathrm{h},n} = \min(\alpha,\mu_i)$,
  and the inequalities~$\xi_i\leq \mu_i$, we get:
  \begin{align*}
    \sum_{i=1}^n (\xi^{\mathrm{h},n}_i - t)_+
    &= \sum_{i=1}^n(\xi^{\mathrm{h},n}_i - t) + \sum_{i=1}^n (t-\xi^{\mathrm{h},n}_i)_+ \\
    &= \sum_{i=1}^n(\xi_i - t) + \sum_{i=1}^n (t-\mu_i)_+ \\
    &\leq \sum_{i=1}^n(\xi_i -t) + \sum_{i=1}^n (t-\xi_i)_+ \\
    &= \sum_{i=1}^n (\xi_i - t)_+.
  \end{align*}
  This gives~$\xi^{\mathrm{h},n} \prec \xi$, by the
  characterization~\eqref{eq:majorization_(x-a)+}.
\end{proof}

\subsubsection{``Vertical'' Pareto optima in the disassortative case}

We   consider   here   the    disassortative   model~$b\geq   a\geq   0$
and~$b>0$.                Let~$c               \in                (0,1)$
and~$D(c)= \{ \eta \in \Delta \, \colon \, C(\eta) = c \}$ be the set of
vaccination  strategies with  cost $c$.  We will  solve the  constrained
optimization Problem~\eqref{eq:prob-min-Re} that corresponds to:
\begin{equation}\label{eq:prob-min-conc}
  \left\{
    \begin{array}{cc}
      \min \, & R_e(\eta), \\
      \text{such that} & \eta \in D(c).
    \end{array}
  \right.
\end{equation}
Recall the definitions of~$\beta_c$ and~$\ver(c)$ given page~\pageref{eq:ver}.
Let~$\eta\in D(c)$. Let~$n$ be large enough so
that~$\sum_{j>n} \mu_j< 1-c$ so that~$\sum_{j\leq n} \eta_j \mu_j>0$, and assume
that~$n>\beta$.
Let~$\eta^{(n)}\in \Delta$ be defined by:
\[
  \eta^{(n)}_i = \frac{\sum_{j\leq n} \ver_j(c) \mu_j}{\sum_{j\leq n} \eta_j \mu_j}
  \ind{\{i\leq n\}} \, \eta_i.
\]
Note that since~$C(\ver(c)) = c = C(\eta)$, we have~$\lim_{n \to N} \eta^{(n)} = \eta$
(pointwise and in~$L^2$). Let~$\xi^n = (\eta_1^{(n)} \mu_1, \ldots, \eta_n^{(n)} \mu_n)$
and~$\xi^{\mathrm{v}, n}$ be defined as in Proposition~\ref{prop:extreme} with~$\beta =
\beta_c$. By construction, we have~$\sum_{i=1}^n \xi_i^n = \sum_{i=1}^n
\xi_i^{\mathrm{v},n}$, so by Proposition~\ref{prop:extreme}, we get~$\xi^n \prec
\xi^{\mathrm{v},n}$.
This implies that:
\[
  R_e(\eta^{(n)} ) = \Theta_n(\xi^n)
  \geq \Theta_n(\xi^{\mathrm{v},n}) = R_e(\ver(c)),
\]
where the inequality follows from the Schur concavity of~$\Theta_n$ in the disassortative
case (see Lemma~\ref{lem:q=schur-cave}) and where the last equality holds as~$n \geq
\ceil{\beta_c}$. Since~$R_e$ is continuous and $\eta^{(n)}$ 
converges pointwise and in~$L^2$ to~$\eta$, we
get~$R_e(\eta)\geq R_e(\ver)$. This implies that~$\ver$ is a solution of
Problem~\eqref{eq:prob-min-conc}.

\medskip

If~$a > 0$, then~$\kk$ is positive everywhere, and we deduce from Lemma~\ref{lem:k>0-c}
that~$\cmir = 1$. If~$a=0$, it is easy to prove that~$\{0\}$ is a maximal independant set
of~$\kk$; this gives that $\cmir=1-\mu_1$, thanks to \cite[Section~6.4]{ddz-Re}. Since for
all $c\in [0,\cmir)$ there exists~$\eta \in \verPath$ such that $C(\eta)=c$, we also get
that $\verPath \cap \{ \eta \in \Delta \, \colon \, C(\eta) \leq \cmir \}$ is a
parametrization of the Pareto frontier. This gives the parametrization of the Pareto
frontier using~$\verPath$ from Theorem~\ref{th:dis-assortative}~\ref{pt:disass}
and~\ref{pt:multi}.

\subsubsection{``Horizontal'' anti-Pareto optima in the disassortative case}

We still consider~$b \geq a \geq 0$ and~$b>0$. Let~$c \in (0,1)$. We now turn to the
anti-Pareto frontier by studying the constrained maximization
Problem~\eqref{eq:prob-max-Re} that corresponds to:
\begin{equation}\label{eq:prob-max-conc} 
  \left\{
    \begin{array}{cc}
      \max \, & R_e(\eta), \\
      \text{such that} & \eta \in D(c).
    \end{array}
  \right.
\end{equation}
Recall the definitions of~$\alpha_c$ and~$\ho(c)$ given
page~\pageref{eq:ho}. Let~$\eta\in D(c)$. Let~$n$ be large enough so that~$\sum_{j>n}
\mu_j< 1-c$ and thus~$\sum_{j\leq n} \eta_j
\mu_j>0$. Define~$\eta^{(n)}\in \Delta$ by:
\[
  \eta^{(n)}_i = \frac{\sum_{j\leq n} \ho_j(c) \mu_j}{\sum_{j\leq n} \eta_j \mu_j}
  \ind{\{i\leq n\}} \, \eta_i.
\]
Let~$\xi^n = (\eta_1^{(n)} \mu_1, \ldots, \eta_n^{(n)} \mu_n)$ and
let~$\xi^{\mathrm{h}, n}$ be defined as in Proposition~\ref{prop:extreme} with~$\alpha =
\alpha_c$. By construction, we have~$\sum_{i=1}^n \xi_i^n = \sum_{i=1}^n
\xi_i^{\mathrm{h},n}$, so by Proposition~\ref{prop:extreme}, we obtain~$\xi^{\mathrm{h},n}
\prec \xi^n$. This implies that:
\[
  R_e(\eta^{(n)} ) = \Theta_n(\xi^n)
  \leq \Theta_n(\xi^{\mathrm{h},n}) = R_e(\ho(c) \, \ind{\lb 1, n \rb}),
\]
where the inequality follows from the Schur concavity of~$\Theta_n$. \medskip

Now, as~$n$ goes to infinity~$\eta^{(n)}$ converges pointwise and in $L^2$ to~$\eta$,
and~$\ho(c) \, \ind{\lb 1, n \rb}$ converges pointwise and in~$L^2$ to~$\ho(c)$, so by
continuity of~$R_e$ we get $R_e(\eta) \leq R_e(\ho(c))$, and~$\ho(c)$ is solution of the
Problem~\eqref{eq:prob-max-conc} and is thus anti-Pareto optimal for $c\in (0, 1)$
as~$\cmar=0$. Since~$\cmar=0$, we also deduce from \cite[Propsotion~5.8~(iii)]{ddz-theo}
that~$\zero$ and~$\un$ are anti-Pareto optimal. Since for all~$c\in [0, 1]$ there
exists~$\eta\in\hoPath$ such that~$C(\eta)=c$, we deduce that~$\hoPath$ is a
parametrization of the anti-Pareto frontier.

\subsubsection{The assortative case}

The case~$a \geq b>0$, corresponding to point~\ref{pt:ass} in
Proposition~\ref{th:dis-assortative}, is handled similarly, replacing concavity by
convexity, minima by maxima and vice versa.

\section{\Regular{} kernels and unifom vaccinations}\label{sec:constant_degree}

\subsection{Motivation}

We have seen in the previous section an example of model where vaccinating individuals
with the highest degree is the best strategy.  A similar phenomenon is studied
in~\cite{ddz-mono}, where under monotonicity arguments on the kernel, vaccinating
individuals with the highest (resp.\ lowest) degree is Pareto (resp.\ anti-Pareto) optimal.
However, in case multiple individuals share the same maximal degree, the optimal
strategies differ completely between the assortative and the disassortative models: the
Pareto optimal strategies for one model correspond to the anti-Pareto optimal strategies
for the other and vice versa.

Motivated by this curious symmetry, we investigate in this section
\regular{} kernels, that is, the situation where all the individuals
have the same number of connections. In Section~\ref{sec:unif-reg}, we
define these kernels formally and give the main result on the
optimality of the uniform strategies when $R_e$ is either convex or
concave, see Proposition~\ref{prop:regular}.
Section~\ref{sec:proof-regular} is devoted to the proof of this main
result. We study in more detail the optimal strategies in an example
of \regular{} symmetric kernels of rank two in
Section~\ref{sec:rank-2-reg}. Eventually, we study in Section
\ref{sec:geometric} geometric kernels on the sphere, which are
\regular{} kernels.

\subsection{On the uniform strategies for \regular{} kernels}\label{sec:unif-reg}

In graph theory, a regular graph is a graph where all vertices have the same number of
in-neighbors, and the same number of out-neighbors. In other words all vertices have the
same in-degree and the same out-degree.  Limits of undirected regular graphs have been
studied in details by Backhausz and Szegedy \cite{backhausz_action_2020} and
Kunszenti-Kovács, Lovász and Szegedy \cite{kunsz}. When the graphs are dense, their limit
can be represented as a regular graphon, that is a symmetric kernel with a constant degree
function.

Since we do not wish to assume symmetry, we give the following general
definition. 
For a kernel~$\kk$ on~$\Omega$, we set, for all~$z\in \Omega$ and~$A\in \cf$:
\[
  \kk(z, A)=\int_A \kk(z,y)\, \mu(\mathrm{d} y)
  \quad\text{and}\quad
  \kk(A,z)=\int_A \kk(x,z)\, \mu(\mathrm{d} x).
\]
For~$z\in \Omega$, its in-degree is~$\kk(z, \Omega)$ and its out-degree is~$\kk(\Omega,
z)$. 

\begin{definition}[\Regular{} kernel]\label{def:regular}
  A kernel~$\kk$ with a finite~$L^2$ double-norm and a positive spectral radius~$R_0>0$ is
  called \emph{\regular}  if all the in-degrees and all the out-degrees have the same
  value, that is, the maps~$x\mapsto \kk(x, \Omega)$ and~$y\mapsto \kk(\Omega, y)$ defined
  on~$\Omega$ are constant, and thus equal.
\end{definition}

\begin{remark}
  Let~$\kk$ be a \regular{} kernel with spectral radius~$R_0 > 0$. Notice the condition
  ``all the in-degrees and out-degrees have the same value'' is also equivalent to~$\un$
  being a left and right eigenfunction of~$T_\kk$. We now check that the corresponding
  eigenvalue is~$R_0$. 

  Let~$h \in L^2_+(\Omega)
  \backslash \{ \zero\}$ be a left Perron-eigenfunction. Denote by~$\lambda$ the
  eigenvalue associated to~$\un$. Then, we have:
  \begin{equation*}
    \lambda \int_\Omega h(x) \, \mu(\mathrm{d}x) = \int_\Omega h(x) \kk(x,y)
    \mu(\mathrm{d}x) \mu(\mathrm{d}y) = R_0 \int_\Omega h(y) \, \mu(\mathrm{d}y),
  \end{equation*}
  where the first equality follows from the regularity of~$\kk$ and from the fact that~$h$
  is a left Perron-eigenfunction of~$T_\kk$. Since~$h$ is non-negative and not equal
  to~$\zero$ almost everywhere, we get that~$\lambda = R_0$ and~$\un$ is a right
  Perron-eigenvector of~$T_\kk$. With a similar proof, we show that~$\un$ is a left
  Perron-eigenvector of~$T_\kk$. In particular, if~$\kk$ is \regular{}, then the
  reproduction number is given by:
  \begin{equation}\label{eq:r0-regular}
    R_0 = \int_{\Omega \times \Omega} \kk(x,y) \, \mu(\mathrm{d}x) \mu(\mathrm{d}y).
  \end{equation}
\end{remark}

\begin{example}\label{ex:regular}
  We now give examples of \regular{} kernels.
  \begin{enumerate}[(i)]
    \item Let~$G=(V, E)$ be a finite non-oriented simple graph, and~$\mu$ the uniform
      probability measure on the vertices~$V$. The degree of a vertex~$x\in V$ is given by
      \[
	\deg(x)=\sharp \{y\in V\, \colon\, (x,y)\in E\}.
      \]
      The graph~$G$ is \regular{} if all its vertices have the same degree, say~$d\geq 1$.
      Then the kernel defined on the finite space~$\Omega=V$ by the adjacency matrix is
      \regular{} with~$R_0=d$. Notice it is also symmetric.
    \item\label{ex:regular-asym-circle} Let~$G=(V, E)$ be a finite directed graph,
      and~$\mu$ be the uniform probability measure on the vertices~$V$. The in-degree of a
      vertex~$x \in V$ is given by
      \[
	\deg_{\rm{in}}(x) = \sharp \{ y \in V \, \colon \, (y,x)\in E\},
      \]
      and the out-degree is given by
      \[
	\deg_{\rm{out}}(x)= \sharp \{y\in V\, \colon\, (x,y)\in E\}.
      \]
      The graph~$G$ is regular if all its vertices have the same in-degree and out-degree,
      say~$d\geq 1$. Then the kernel defined on the finite space~$\Omega=V$ by the
      adjacency matrix is regular with~$R_0=d$. Notice it might not be symmetric.
    \item\label{ex:circle}  Let~$\Omega  =  (\R  /  (2  \pi  \Z))^m$  be
      the~$m$-dimensional       torus       endowed       with       its
      Borel~$\sigma$-field~$\mathscr{F}$  and  the  normalized  Lebesgue
      measure~$\mu$.   Let~$f$   be   a   measurable   square-integrable
      non-negative  function   defined  on~$\Omega$.  We   consider  the
      geometric kernel on $\Omega$ defined by:
      \[
	\kk_f(x,y)=f(x-y).
      \]
      The kernel $\kk_f$ has a finite double norm as $f\in L^2$. 
      The operator~$T_{\kk_f}$ corresponds to the convolution by~$f$, and its spectral
      radius is given by~$R_0 = \int_\Omega f \, \rd\mu$. Then the kernel~$\kk_f$ is
      \regular{} as soon as~$f$ is not equal to~$0$ almost surely. This example is
      developed in Section~\ref{sec:geometric} in the case~$m=1$
      (corresponding to $d=2$ therein), see in particular
      Examples~\ref{ex:sphere=circle} and~\ref{ex:square}. 
    \item More generally, let~$(\Omega, \cdot)$ be a compact topological group and
      let~$\mu$ be its left Haar probability measure. Let~$f$ be non-negative
      square-integrable function on~$\Omega$. Then the kernel~$\kk_f(x,y) = f(y^{-1} \cdot
      x)$ is \regular.
  \end{enumerate}
\end{example} 

We summarize our main result in the next proposition, whose proof is given in
Section~\ref{sec:proof-regular}. We recall that a strategy is called uniform if it is
constant over~$\Omega$. 

\begin{proposition}[Uniform strategies for \regular{} kernels]\label{prop:regular}
  Let~$\kk$ be a \regular{} kernel on~$\Omega$.
  \begin{enumerate}[(i)]
    \item\label{prop:reg-cvxe} If the map~$R_e$ defined on~$\Delta$ is convex, then all
      uniform strategies are Pareto optimal (\emph{i.e.}~$\cpu\subset \cp$).
      Consequently,~$\cmir = 1$, the Pareto frontier is the segment joining~$(0, R_0)~$
      to~$(1,0)$, and for all~$c \in [0,1]$: \[ \mir(c) = (1 - c) R_0. \]
    \item\label{prop:reg-cave} If  the map~$R_e$ defined  on~$\Delta$ is
      concave,  then  the  kernel  $\kk$  is  irreducible and all  uniform
      strategies            are           anti-Pareto            optimal
      (\emph{i.e.}~$\cpu\subset \cpa$).  Consequently,~$\cmar =  0$, the
      anti-Pareto   frontier   is   the  segment   joining~$(0,   R_0)~$
      to~$(1,0)$, and for all~$c \in [0,1]$:
      \[
	\mar(c) = (1-c) R_0.
      \]
  \end{enumerate}
\end{proposition} 

In \cite[Section~5.2]{ddz-Re}, we give sufficient condition on the spectrum of~$T_\kk$ to
be either concave or convex. Combining this result with Proposition~\ref{prop:regular}, we
get the following corollary.

\begin{corollary}\label{cor:regular}
  Let~$\kk$ be a \regular{} symmetric kernel.
  \begin{enumerate}[(i)]
    \item\label{cor:regular-conv} If the eigenvalues of~$T_\kk$ are non-negative, then the
      uniform vaccination strategies are Pareto optimal and~$\cmir=1$
      (\emph{i.e.}~$\cpu\subset \cp$). 
    \item\label{cor:regular-conc} If $R_0$ is a simple  eigenvalue
      of~$T_\kk$ and the others eigenvalues are non-positive, then the kernel $k$ is
      irreducible,  the uniform vaccination strategies are
      anti-Pareto optimal and~$\cmar=0$ (\emph{i.e.}~$\cpu\subset \cpa$). 
  \end{enumerate}
\end{corollary}

\begin{remark}[Equivalent conditions]\label{rem:semi-definite}
  Let~$\kk$ be a \regular{} symmetric kernel. The eigenvalues of the operator~$T_\kk$ are
  non-negative if and only if~$T_\kk$ is semi-definite positive, that is:
  \begin{equation}
    \int_{\Omega \times \Omega} \kk(x,y) g(x) g(y) \mu(\mathrm{d}x) \mu(\mathrm{d}y) \geq
    0 \quad \text{for all}\quad g \in L^2.
  \end{equation}
  Similarly, the condition given in Corollary~\ref{cor:regular}~\ref{cor:regular-conc}
  that implies the concavity of~$R_e$ is equivalent to the semi-definite negativity
  of~$T_\kk$ on the orthogonal of~$\un$:
  \begin{equation}
    \int_{\Omega \times \Omega} \kk(x,y) g(x) g(y) \mu(\mathrm{d}x) \mu(\mathrm{d}y) \leq
    0 \quad \text{for all}\quad g \in L^2 \quad\text{such that}\quad
    \int_\Omega g\, \mathrm{d} \mu = 0. 
  \end{equation}
\end{remark}

\begin{remark}[Comparison with a result from \cite{poghotanyan_constrained_2018}]
  Poghotanyan, Feng, Glasser and Hill \cite[Theorem~4.7]{poghotanyan_constrained_2018}
  obtained a similar result in finite dimension using a result from Friedland
  \cite{Fried80}: if the next-generation non-negative matrix~$K$ of size~$N \times N$
  satisfies the following conditions
  \begin{enumerate}[(i)]
    \item\label{item:po1} $\sum_{j=1}^{N} K_{ij}$ does not depend on~$i\in \lb 1, N
     \rb$ (which corresponds the parameters~$a_i$ in
      \cite[Equation~(2.4)]{poghotanyan_constrained_2018} being all equal),
    \item \label{item:po2} $\mu_i K_{ij} = \mu_j K_{ji}$ for all~$i,j \in \lb 1, N \rb$
      where~$\mu_i$ denote the relative size of population~$i$ (which corresponds to
      \cite[Equation~(2.4)]{poghotanyan_constrained_2018}),
    \item \label{item:po3} $K$ is not singular and its inverse is an M-matrix
      (\textit{i.e.}, its non-diagonal coefficients are non-positive),
  \end{enumerate}
  then the uniform strategies are Pareto optimal (\emph{i.e.}, they minimize the
  reproduction number among all strategies with same cost). Actually, this can be seen as
  a direct consequence of Corollary~\ref{cor:regular}~\ref{cor:regular-conv}. Indeed, the
  corresponding kernel~$\kk_\mathrm{d}$ defined by~\eqref{eq:next-kernel} in the discrete
  probability space~$\Omega= \lb 1, N \rb$ endowed with the discrete probability
  measure~$\mu_\mathrm{d}$ also defined by~\eqref{eq:next-kernel} has constant degree
  thanks to Point~\ref{item:po1} and is symmetric thanks to Point~\ref{item:po2}.
  Since~$K^{-1}$ is an M-matrix, its real eigenvalues are positive according to
  \cite[Chapter~6 Theorem~2.3]{NonnegativeMatBerman1994}. The eigenvalues of
  ~$T_{\kk_\mathrm{d}}$ and $K$ are actually the same as~$K$ is the representation matrix
  of~$T_{\kk_\mathrm{d}}$ in the canonic basis of~$\R^N$. We conclude that the
  operator~$T_{\kk_\mathrm{d}}$ is positive definite. Hence
  Corollary~\ref{cor:regular}~\ref{cor:regular-conv} can be applied to recover that the
  uniform strategies are Pareto optimal. \medskip
  
  However, the converse is not true. As a counter-example, consider a population divided
  in~$N=3$ groups of same size (\textit{i.e},~$\mu_1 = \mu_2 = \mu_3 = 1/3$) and the
  following next-generation matrix:
  \begin{equation*}
    K = 
    \begin{pmatrix}
      3 & 2 & 0 \\
      2 & 2 & 1 \\
      0 & 1 & 4
    \end{pmatrix}
    \qquad \text{with inverse} \quad
    K^{-1} = 
    \begin{pmatrix}
      1.4 & -1.6 & 0.4 \\
      -1.6 & 2.4 & -0.6 \\
      0.4 & - 0.6 & 0.4 
    \end{pmatrix}.
  \end{equation*}
  Clearly Points~\ref{item:po1} and~\ref{item:po2} hold and Point~\ref{item:po3} fails
  as~$K^{-1}$ is not an M-matrix. Nevertheless, the matrix~$K$ is definite positive as its
  eigenvalues~$\sigma(K) = \{ 5, 2 + \sqrt{3}, 2 - \sqrt{3} \}$ are positive. And thus,
  thanks to Corollary~\ref{cor:regular}~\ref{cor:regular-conv}, we get that the uniform
  strategies are Pareto optimal. Hence, Corollary~\ref{cor:regular}~\ref{cor:regular-conv}
  is a strict generalization of \cite[Theorem~4.7]{poghotanyan_constrained_2018} even for
  finite metapopulation models.
\end{remark}

\begin{remark}
  We also refer the reader to the paper of Friedland and Karlin \cite{FK75}: from the
  Inequality~(7.10) therein, we can obtain
  Corollary~\ref{cor:regular}~\ref{cor:regular-conv} when~$\Omega$ is a compact set
  of~$\R^n$,~$\mu$ is a finite measure,~$\kk$ is a continuous symmetrizable kernel such
  that~$\kk(x, x)>0$ for all~$x\in \Omega$.

  Further comments on related results may be found in the discussion of~\cite[Theorem
  5.1]{ddz-Re}. 
\end{remark}

Below, we give examples of metapopulation models from the previous sections where
Proposition~\ref{prop:regular} applies. For continuous models, we refer the reader to
Sections~\ref{sec:rank-2-reg} and~\ref{sec:geometric}.

\begin{example}[Fully asymmetric cycle model]\label{ex:asym-cyclic}
  We consider the fully asymmetric circle model with~$N \geq 3$ vertices developed in
  Section~\ref{sec:asym-circle}. Since the in and out degree of each vertex is exactly
  one, the adjacency matrix is \regular{} according to
  Example~\ref{ex:regular}~\ref{ex:regular-asym-circle}.
  
  The spectrum of the adjacency matrix is given by the~$N$th roots of unity, so  for~$N
  \geq 3$ it does   not lie in~$\R_-\cup \{R_0\}$, so Corollary~\ref{cor:regular} does not
  apply.  However, in this case the effective spectral radius~$R_e$ is given by formula
  \eqref{eq:geom-mean}, which corresponds to the geometric mean. According to
  \cite[Section~3.1.5]{bv2009}, the map~$\eta \mapsto R_e(\eta)$ is concave, so
  Proposition~\ref{prop:regular}~\ref{prop:reg-cave} applies. This proves that the
  spectral condition given in Corollary~\ref{cor:regular} and \cite[Section~5.2]{ddz-Re}
  to get the concavity of $R_e$ is only sufficient.
\end{example}

\begin{example}[Finite assortative and disassortative model]\label{ex:cyclic-group}
  Let~$\Omega = \{ 1,2, \ldots, N \}$ and~$\mu$ be the uniform probability on~$\Omega$.
  Let~$a, b\in \R_+$. We consider the kernel from the models developed in
  Section~\ref{sec:ass-disass}: 
  \[
    \kk(i,j) = a \ind{i=j} + b \ind{i \neq j}. 
  \]
  Since~$\mu$ is uniform, the kernel~$\kk$ is \regular; provided its spectral radius is
  positive, \textit{i.e.},~$a$ or~$b$ is positive.

  In the assortative model~$0<b \leq a$, according to
  Proposition~\ref{prop:disass-ass-eig}~\ref{pt:ass-Re}, the eigenvalues of the symmetric
  operator~$T_\kk$ are non-negative. Hence,
  Corollary~\ref{cor:regular}~\ref{cor:regular-conv} applies: the
  uniform strategies are Pareto optimal. This is consistent with
  Theorem~\ref{th:dis-assortative}~\ref{pt:ass}.

  In the  dissortative model,  we have~$0  \leq a \leq  b$ and~$b  > 0$.
  According to Proposition~\ref{prop:disass-ass-eig}~\ref{pt:disass-Re},
  the  eigenvalues of~$T_\kk$  different  from its  spectral radius  are
  non-positive.                                                   Hence,
  Corollary~\ref{cor:regular}~\ref{cor:regular-conc}     applies:    the
  uniform  strategies   are  anti-Pareto.    This  is   consistent  with
  Theorem~\ref{th:dis-assortative}~\ref{pt:disass} and~\ref{pt:multi}.
\end{example}

\subsection{Proof of Proposition \ref{prop:regular}}\label{sec:proof-regular}

By analogy with \cite{LineSumSymmetEaves1985}, we consider the following definition.

\begin{definition}[Completely reducible kernels]
  \label{def:comp-irr}
  A kernel $\kk$  is said to be \emph{completely reducible}  if there
  exist an at most countable index set $I$, and 
  measurable sets $\Omega_0$ and $(\Omega_i, i\in I)$, such
  that~$\Omega$ is the disjoint union $\Omega = \Omega_0
  \sqcup(\bigsqcup_{i\in I} \Omega_i)$, the kernel $\kk$ decomposes
  as $\kk=\sum_{i\in   I} \ind{\Omega_i}
  \kk  \ind{\Omega_i}$ a.e., and, for all $i\in I$,  the  kernel $\kk$
  restricted to~$\Omega_i$ is irreducible. 
\end{definition}

  As in the discrete case for so-called line sum symmetric matrices,
  see~\cite[Lemma 1]{LineSumSymmetEaves1985},
  kernels for which for any $x$ the out-degree is equal to the
  in-degree
  are necessarily completely reducible; the fact that these degrees
  do not depend on $x$ impose further constraints.

\begin{lemma}[Complete reduction]\label{lem:reduction}
  If~$\kk$ is a \regular{} kernel on~$\Omega$, then~$\kk$ is completely
  reducible. Furthermore, the
  set $\Omega_0$ from Definition~\ref{def:comp-irr} is empty,
   the   cardinal of the partition $(\Omega_i, i\in I)$
  is equal to the multiplicity of
  $R_0$ and thus is finite; and,  for all $i \in I$, the kernel $\kk$
  restricted to $\Omega_i$ is a  \regular{} irreducible kernel with spectral
  radius equal to $R_0$.
\end{lemma}

\begin{proof}
  We recall that a set~$A\in \cf$ is invariant if~$\kk(A^c, A)=0$, where for~$A, B \in
  \cf$:
  \[
    \kk(B, A)=\int_{B\times A} \kk(x, y)\, \mu(\rd x)\mu(\rd
    y).
  \]
  Since for each $x$, the in-degree $\kk(x,\Omega)$ is equal to the
  out-degree
  $\kk(\Omega,x)$, we get by integration $\kk(A,\Omega) =
  \kk(\Omega,A)$, so
  \[
    \kk(A^c,A) = \kk(A^c,\Omega) - k(A^c,A^c) = k(\Omega,A^c) - k(A^c,A^c)
    = k(A,A^c).
  \]
  Therefore if $A$ is invariant, then so is its complement $A^c$. 
  According    to    \cite[Section~7]{ddz-Re}   and    more    precisely
  Remark~7.1(iv),  there  exists then  an  at  most countable  partition
  of~$\Omega$ made of $\Omega_0$ and~$(\Omega_i, i\in I)$
   such that~$\kk=\sum_{i\in I} \kk_i$, with $\kk_i=\ind{\Omega_i}\kk
   \ind{\Omega_i}$, $\mu(\Omega_i)>0$ and~$\kk_i$ restricted to~$\Omega_i$ is irreducible.
  Since~$\un$ is an eigenvector of~$T_\kk$ associated to the eigenvalue~$R_0$ and the
  sets~$\Omega_0$ and $(\Omega_i, i\in I)$ are pairwise disjoint, we
  deduce that~$\Omega_0$ is of zero measure and $\ind{\Omega_i}$ is an
  eigenvector of~$T_{\kk_i}$ with eigenvalue~$R_0>0$, for all~$i \in I$. Hence, all the
  kernels~$\kk_i$ restricted to~$\Omega_i$ are irreducible \regular{}
  kernels with spectral radius equal to $R_0$. Thus, the cardinal
  of~$I$ is equal to the multiplicity of~$R_0$ (for~$T_\kk$). Since~$\kk$ has finite~$L^2$
  double-norm, the operator~$T_\kk$ is compact, and the multiplicity of~$R_0>0$, and thus
  the cardinal of~$I$, is finite. 
\end{proof}

\begin{lemma}\label{lem:uni-crit}
  Let~$\kk$ be a \regular{} irreducible kernel on~$\Omega$. Then the uniform strategy is a
  critical point for~$R_e$ among all the strategies with the same cost in~$(0, 1)$, and
  more precisely: for all~$\eta$ with the same cost in~$(0, 1)$
  as~$\etau\in \cpu$
  and~$\varepsilon>0$ small enough, we have:
  \[
    R_e((1-\varepsilon) \etau + \varepsilon \eta)= R_e(\etau)+O(\varepsilon^2).
  \]
\end{lemma}

\begin{proof}
  Let~$\etau$ be the uniform strategy with cost~$c\in (0, 1)$. Since~$\kk$ is irreducible,
  we get that~$(1-c)R_0$ is a simple isolated eigenvalue of~$\kk \etau$, whose
  corresponding left and right eigenvector are~$\un$ as~$\kk \etau$ is also \regular.
  For~$\eta\in \Delta$, we get that~$T_{\kk((1-\varepsilon) \etau + \varepsilon \eta)}$
  converges to~$T_{\kk \etau}$ (in operator norm, thanks to \eqref{eq:double-norm})
  as~$\varepsilon$ goes down to~$0$. Notice that:
  \[
    \norm{ T_{\kk (\etau +
  \varepsilon(\eta - \etau))} - T_{\kk \etau}}_{L^2}^2=
O(\varepsilon^2).
\]
According to
  \cite[Theorem~2.6]{EffectivePertuBenoit}, we get that for any~$\eta\in \Delta$
  and~$\varepsilon>0$ small enough:
  \begin{align*}
    R_e((1-\varepsilon) \etau + \varepsilon \eta)- R_e(\etau)
    &= \varepsilon \int_\Omega \kk(x,y) (\eta(y) - \etau(y)) \, \mu(\mathrm{d}x) \mu(\mathrm{d}y) + O(\varepsilon^2) \\
    &= \varepsilon R_0 \int_\Omega (\eta(y) - \etau(y)) \,\mu(\mathrm{d}y) + O(\varepsilon^2), 
  \end{align*}
  where for the last equality we used that~$\kk$ is \regular. In particular, if~$\eta$
  and~$\etau$ have the same cost~$c\in (0, 1)$, then \( R_e((1-\varepsilon) \etau +
  \varepsilon \eta) - R_e(\etau)=O(\varepsilon^2)\), which means that the uniform strategy
  is a critical point for~$R_e$ among all the strategies with cost~$c\in (0, 1)$.
\end{proof}

\begin{proof}[Proof of Proposition~\ref{prop:regular}]
  We prove \ref{prop:reg-cvxe}, and thus consider~$\kk$ \regular{} and~$R_e$ convex. We
  first consider the case where~$\kk$ is irreducible. For any~$\eta$,
  Lemma~\ref{lem:uni-crit} and the convexity of~$R_e$ imply that
  \[
    R_e(\etau) + O(\varepsilon^2)
    = R_e((1-\varepsilon)\etau + \varepsilon \eta)
    \leq (1-\varepsilon)R_e(\etau) + \varepsilon R_e(\eta),
  \]
  where~$\etau$ the uniform strategy with the same cost as~$\eta$. Letting~$\varepsilon$
go  to~$0$, we get~$R_e(\eta)\geq R_e(\etau)$, so~$R_e$ is minimal at~$\etau$.

  Since~$C(\etau)=c$ and~$R_e(\etau)=(1-c) R_0$, we deduce that~$\mir(c) = (1 -
  c) R_0$ and thus, the Pareto frontier is a segment given by~$\F = \{ (c, (1-c)R_0) \,
  \colon \, c\in [0, 1] \}$.

  \medskip  In what  follows, we  write  $R_e[\kk]$ to  stress that  the
  reproduction function on $\Delta$ defined by~\eqref{eq:def-Re} depends
  on   the  kernel   $\kk$:   $R_e[\kk](\eta)=\rho(T_{\kk  \eta})$   for
  $\eta\in  \Delta$.    If~$\kk$  is  not  irreducible,   then  use  the
  representation     from      Lemma~\ref{lem:reduction},     to     get
  that~$R_e[\kk]=\max_{i\in I} R_e[\kk_i]$. Since the cost is affine, we
  deduce that a strategy~$\eta$  with~$R_e[\kk](\eta)=\ell \in [0, R_0]$
  is   Pareto  optimal   if  and   only  if,   for  all~$i\in   I$,  the
  strategies~$\eta_i=\eta\ind{\Omega_i}$    are    Pareto optimal    for    the
  kernel~$\kk$  restricted to~$\Omega_i$  and~$R_e[\kk_i](\eta_i)=\ell$;
  see also~\cite[Corollary~7.4]{ddz-Re}.   
  Then     the      first     step      of     the      proof     yields
  that~$\eta_i   =   \ell   \ind{\Omega_i}$   and   thus   the   uniform
  strategy~$\etau=\ell\ind{\Omega}$ is  Pareto optimal. This ends the  proof of
  \ref{prop:reg-cvxe}.

  \medskip

  We now prove \ref{prop:reg-cave}.  We first check that
 the kernel $\kk$ is irreducible.
Thanks to Lemma~\ref{lem:reduction},
  the kernel $\kk$ is  completely reducible with a zero measure $\Omega_0$.
  However, \cite[Lemma~7.3]{ddz-Re} also implies that it is
  \emph{monatomic}, a notion introduced in~\cite[Section~3.4]{ddz-Re}
  which intuitively states that $\kk$ has only one irreducible
  component. Together with complete reducibility, this implies that
  $\kk$ is  irreducible.
  The rest of the proof is then similar to the proof
  of~\ref{prop:reg-cvxe}
  under the irreducibility assumption.
\end{proof}

\section{\Regular{} symmetric kernels of rank two}\label{sec:rank-2-reg}

\subsection{Pareto and anti-Pareto frontiers}
\label{sec:pp-f-rk2}
Any \regular{} symmetric kernel may be decomposed spectrally in
$L^2(\Omega^2,\mu^{\otimes 2})$  as $k(x,y) = R_0+ \sum_{n\in
\N^*} \varepsilon_n \alpha_n(x)\alpha_n(y)$, with~$\varepsilon_n\in \{-, +\}$,~$(\alpha_n,
n\in \N^*)$ an orthogonal family of~$L^2$ also orthogonal to~$\un$. As an application
of the results from the previous section, we will treat the case of symmetric \regular{}
kernel  whose associated operator is  of rank~$2$, where one can explicitly minimize and maximize~$R_e$ among
all strategies of a given cost.

We suppose that~$\Omega = [0,1)$ is equipped with the Borel $\sigma$-field~$\cf$ and a
probability measure~$\mu$ whose cumulative distribution function~$\varphi$, defined
by~$\varphi(x)=\mu([0,x])$ for~$x\in \Omega$, is continuous and increasing.  We consider
the following two kernels on~$\Omega$:
\begin{equation}\label{eq:k+}
  \kk^\varepsilon(x,y) = R_0 + \varepsilon \alpha(x) \alpha(y),
  \quad\text{with}\quad \varepsilon\in \{-, +\},
\end{equation}
where $R_0 > 0$ and~$\alpha\in L^2$ is strictly increasing and satisfies:
\begin{equation}
  \label{eq:alpha-cond}
  \sup_{\Omega} \alpha^2 \leq R_0
  \quad\text{and}\quad
  \int_\Omega \alpha \,\mathrm{d} \mu = 0. 
\end{equation}

\begin{remark}[Generality]
  We note that this particular choice of~$\Omega$ may be made without loss of generality,
  and that the strict monotonicity assumption on~$\alpha$ is almost general: we refer the
  interested reader to Section~\ref{sec:wlog} for further discussion of this point.
\end{remark}

For~$\varepsilon\in \{-, +\}$, the kernel~$\kk^\varepsilon$ is symmetric and \regular.
Furthermore, we have that~$R_0$ and~$\varepsilon \int_\Omega \alpha^2 \,\mathrm{d} \mu$ are the only
non-zero eigenvalues (and their multiplicity is one) of~$T_{\kk^\varepsilon}$ with corresponding
eigenvector~$\un$ and $\alpha$. Since~$\alpha^2 \leq R_0$, we also get that~$R_0$ is
indeed the spectral radius of~$T_{\kk^\varepsilon}$. 

The Pareto (resp.\ anti-Pareto) frontier is already greedily parametrized by the uniform
strategies for the kernel~$\kk^+$ (resp.~$\kk^-$), see Corollary~\ref{cor:regular}. The
following result restricts the choice of anti-Pareto (resp. Pareto) optimal strategies to
two extreme strategies. Hence, in order to find the optima, it is enough to compute and compare
the two values of~$R_e$ for each cost.

We recall the set of uniform strategies $ \cpu=\{ t\un\, \colon\, t\in [0,1]\}$ and 
consider the following set of extremal strategies:
\[
  \leftStrats = \left\{ \ind{[0, t)}\, \colon\, t \in [0,1]\right\}
  \quad\text{and}\quad
  \rightStrats = \left\{ \ind{[t,1)}\, \colon\, t \in [0, 1]\right\}
\]
as well as the following set of strategies which contains~$ \cpu$ thanks to
\eqref{eq:alpha-cond}:
\[
  \orthStrats=\left\{\eta\in \Delta\, \colon \, \int_\Omega \alpha \, \eta\,
  \mathrm{d} \mu=0\right\}.
\]
Recall that strategies are defined up to the a.s.\ equality. The proof of the next
proposition is given is Section~\ref{sec:proof-kk+-}

\begin{proposition}[Optima are uniform or on the sides]\label{prop:rank2}
  Let $[0,1)$ be endowed with a probability measure whose cumulative
  distribution function is increasing and continuous.
  Let~$\kk^\varepsilon$ 
  be given by
  \eqref{eq:k+} with~$R_0>0$ and~$\alpha$ a strictly increasing
  function on~$[0, 1)$ such that~\eqref{eq:alpha-cond} holds.
\begin{enumerate}[(i)]
    \item \textbf{The kernel~$\kk^+$.} A strategy is Pareto optimal if and only if it belongs to
      $\orthStrats$. In particular, for any~$c\in[0,1]$, the strategy $(1-c)\un$ costs~$c$
      and is Pareto optimal. The only possible anti-Pareto strategies of cost~$c$
      are~$\ind{[0, 1-c)}$ and~$\ind{[c, 1)}$. In other words, 
      \[
	\cp=\orthStrats
	\quad\text{and}\quad
	\cpa\subset \leftStrats \cup \rightStrats.
      \]
    \item \textbf{The kernel~$\kk^-$.} A strategy is anti-Pareto optimal if and only if it belongs
      to~$\orthStrats$. In particular, for any~$c\in[0,1]$,  the strategy~$(1-c)\un$
      costs~$c$ and is anti-Pareto optimal. The only possible Pareto strategies of
      cost~$c$ are~$\ind{[0, 1-c)}$ and~$\ind{[c, 1)}$. In other words, 
      \[
	\cp\subset \leftStrats\cup \rightStrats \quad\text{and}\quad \cpa=\orthStrats.
      \]
  \end{enumerate}
  In both cases, we have~$\cmar=0$ and~$ \cmir =1$. 
\end{proposition}

\begin{remark}\label{rem:ass-alpha}
  Intuitively, the populations~$\{\alpha<0\}$ and~$\{\alpha>0\}$ behave in an assortative
  way for~$\kk^+$ and in a disassortative way for $\kk^-$. As in
  Section~\ref{sec:ass-disass}, the uniform strategies are Pareto optimal in the
  ``assortative''~$\kk^+$ case and anti-Pareto optimal in the ``disassortative''~$\kk^-$ 
  case. 
\end{remark}

\begin{remark}\label{rem:alpha-sym}
  Under the assumptions of Proposition~\ref{prop:rank2}, if furthermore~$\alpha$ is
  anti-symmetric with respect to~$1/2$, that is~$\alpha(x)=-\alpha(1-x)$ for~$x\in (0,
  1)$, and~$\mu$ is symmetric with respect to~$1/2$, that is~$\mu([0, x])=\mu([1-x,1))$,
  then it is easy to check from the proof of Proposition~\ref{prop:rank2} that the
  strategies from~$\leftStrats$ and~$\rightStrats$ are both optimal:~$\cpa=
  \leftStrats\cup\rightStrats$ for~$\kk^+$ and~$\cp= \leftStrats\cup\rightStrats$
  for~$\kk^-$. We plotted such an instance of~$\kk^+$ and the corresponding Pareto and
  anti-Pareto frontiers in Figure~\ref{fig:k+}. We refer to Section~\ref{ssec:infinite}
  for an instance where~$\alpha$ is not symmetric and~$\cp\neq \leftStrats\cup
  \rightStrats$ for~$\kk^-$.
\end{remark}

\begin{figure}
  \begin{subfigure}[T]{.5\textwidth} \centering
    \includegraphics[page=5]{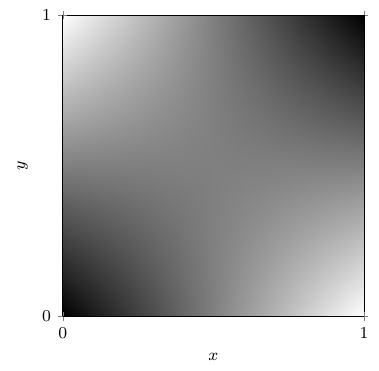}
    \caption{Grayplot of the kernel~$\kk^+(x,y) = 1 + (2x -1)(2y-1)$ on~$[0, 1)$.}
    \label{fig:k+-kernel}
  \end{subfigure}%
  \begin{subfigure}[T]{.5\textwidth} \centering
    \includegraphics[page=6]{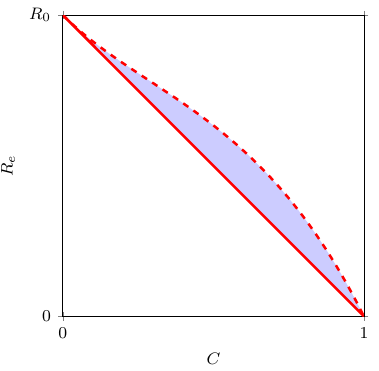}
    \caption{Solid line: the Pareto frontier~$\F$; dashed line: the
      anti-Pareto frontier~$\AF$ (which corresponds to the uniform strategies); blue
    region: all possible outcomes~$\FF$.}
  \label{fig:k+-frontier}
\end{subfigure}%
  \caption{An example of a \regular{} kernel operator of rank 2.}
  \label{fig:k+} 
\end{figure}

\subsection{On the choice of~\texorpdfstring{$\Omega=[0, 1)$}{Omega=[0, 1)} and on the monotonicity assumption}\label{sec:wlog}

Using a reduction model technique from \cite[Section~7]{ddz-theo}, let us first see that
there is no loss of generality by considering the kernel~$\kk^\varepsilon=R_0 +
\varepsilon \alpha \otimes \alpha$ on~$\Omega=[0, 1)$ endowed with the Lebesgue measure
$\mu$ and with~$\alpha$ non-decreasing.

Suppose that the function~$\alpha$ in \eqref{eq:k+} is replaced by  an~$\R$-valued
measurable function~$\alpha_0$ defined on a general probability space~$(\Omega_0,
\mathscr{F}_0, \mu_0)$ such that \eqref{eq:alpha-cond} holds. Thus, with obvious
notations, for~$\varepsilon\in \{-, +\}$, the kernel $R_0+\varepsilon \alpha_0\otimes
\alpha_0$ is a kernel on $\Omega_0$. Denote by~$F$ the repartition function of~$\alpha_0$
(that is,~$F(r)=\mu_0(\alpha_0\leq r)$ for~$r\in \R$) and take~$\alpha$ as the quantile
function of~$\alpha_0$, that is, the right continuous inverse of~$F$. Notice the
function~$\alpha$ is defined on the probability space~$(\Omega, \mathscr{F}, \mu)$ is
non-decreasing and satisfies \eqref{eq:alpha-cond}. Consider the probability kernel
$\kappa \, \colon \, \Omega_0 \times \cf \to [0,1]$ defined by $\kappa(x,
\cdot)=\delta_{F(\alpha_0(x))}(\cdot)$, with~$\delta$ the Dirac mass, if~$\alpha$ is
continuous at~$\alpha_0(x)$ (that is, $F(\alpha_0(x)-)=F(\alpha_0(x))$) and the uniform
probability measure on~$[F(\alpha_0(x)-), F(\alpha_0(x))]$ otherwise. On the measurable
space~$(\Omega_0 \times \Omega, \cf_0 \otimes \cf)$, we consider the probability
measure~$\nu(\mathrm{d}x_1,\mathrm{d}x_2) = \mu_0(\mathrm{d}x_1) \kappa(x_1,
\mathrm{d}x_2)$, whose marginals are exactly~$\mu_0$ and~$\mu$. Then, for~$\varepsilon\in
\{-, +\}$, we have that~$\nu(\mathrm{d}x_1, \mathrm{d}x_2) \otimes \nu(\mathrm{d}y_1,
\mathrm{d}y_2)$-a.s.:
\[
  R_0 +\varepsilon \alpha_0(x_1) \alpha_0(y_1) = R_0 + \varepsilon \alpha(x_2) \alpha(y_2).
\]
According to \cite[Section~7.3]{ddz-theo}, see in particular Proposition~7.3 therein, the
kernels~$R_0+\varepsilon \alpha_0\otimes \alpha_0$ and~$R_0+\varepsilon \alpha\otimes
\alpha$ are coupled and there is a correspondence between the corresponding
(anti-)\-Pareto optimal strategies and their (anti-)Pareto frontiers are the same.

Hence, there is no loss in generality in assuming that the function~$\alpha$ in
\eqref{eq:k+} is indeed defined on~$[0, 1)$ and is non-decreasing.

\bigskip
  
On the contrary, one cannot assume in full generality that~$\alpha$ is \emph{strictly}
increasing, as when it is only non-decreasing, the situation is more complicated. Indeed,
let us take the parameters $R_0 = 1$ and~$\alpha = \ind{[0,0.5)} - \ind{[0.5,1)}$. Then,
the kernel~$\kk^-$ is complete bi-partite:~$\kk^- = \ind{[0,0.5) \times [0.5,1)} +
\ind{[0.5,1) \times [0,0.5)}$. Hence, according to
Theorem~\ref{th:dis-assortative}~\ref{pt:multi}, we have~$\cmir = 0.5$ for the
kernel~$\kk^-$. In a similar fashion, one can see that~$\kk^+ = \ind{[0,0.5) \times
[0,0.5)} + \ind{[0.5,1) \times [0.5,1)}$ is assortative and reducible; it is then easy to
check that~$\cmar = 0.5$ for the kernel~$\kk^+$. However, it is still true that, for all
costs $c$:
\begin{itemize}
  \item~$\ind{[0,1-c)}$ or~$\ind{[c,1)}$ is solution of Problem~\eqref{eq:prob-max-Re}
    when the kernel~$\kk^+$ is considered,
  \item~$\ind{[0,1-c)}$ or~$\ind{[c,1)}$ is solution of Problem~\eqref{eq:prob-min-Re}
    when the kernel~$\kk^-$ is considered.
\end{itemize}
From the proof of Proposition~\ref{prop:rank2}, we can not expect to have strict
inequalities in \eqref{eq:comparison} if~$\alpha$ is only non-decreasing, and thus one can
not expect~$\leftStrats \cup \rightStrats$ to contain~$\cpa$ for the kernel~$\kk^+$
or~$\cp$ for the kernel~$\kk^-$.

\subsection{Proof of Proposition~\ref{prop:rank2}}\label{sec:proof-kk+-}

We assume that~$R_0>0$ and~$\alpha$ is a strictly increasing function
defined on~$\Omega=[0, 1)$
such that~\eqref{eq:alpha-cond} holds. Without loss of generality, we shall assume
that~$R_0=1$ unless otherwise specified. We write~$R_e^\varepsilon$ for the effective
reproduction function associated to the kernel~$\kk^\varepsilon$. We shall also write
$\varepsilon a$ for~$a$ if~$\varepsilon=+$ and~$-a$ if $\varepsilon=-$. We first
rewrite~$R_e^\varepsilon$ in two different ways in Section~\ref{sec:2formula}. Then, we consider the
kernel~$\kk^-$ in Section~\ref{sec:k-} and the kernel $\kk^+$ in Section~\ref{sec:k+}.

\subsubsection{Two expressions of the effective reproduction function}\label{sec:2formula}

We provide an explicit formula for the function~$R_e^\varepsilon$,
and an alternative variational formulation, both of which will
be needed below.

\begin{lemma}\label{lem:2formula}
  Assume~$R_0=1$ and~$\alpha$ is a strictly increasing function defined
  on~$\Omega=[0, 1)$ such
  that~\eqref{eq:alpha-cond} holds. We have for~$\varepsilon\in \{+, -\}$ and~$\eta\in
  \Delta$:
  \begin{equation}\label{eq:re-explicit}
    2R_e^\varepsilon(\eta)= \int \eta\, \mathrm{d} \mu +\varepsilon \int \alpha^2\, \eta\, \mathrm{d} \mu
    + \sqrt{\left(\int \eta\, \mathrm{d} \mu -\varepsilon \int \alpha^2\, \eta\, \mathrm{d} \mu  \right)^2
    + 4\varepsilon \left( \int \alpha\, \eta\, \mathrm{d} \mu\right)^2}.
  \end{equation}
  Alternatively,~$R_e^\varepsilon(\eta)$ is the solution of the variational problem:
  \begin{equation}\label{eq:variational-formula}
    R_e^\varepsilon(\eta) = \sup_{h\in B^\eta_+} \left( \int_0^1 h\, \eta \, \mathrm{d} \mu \right)^2
    + \varepsilon \left( \int_0^1 h\, \alpha\, \eta \, \mathrm{d} \mu \right)^2,
  \end{equation}
  where
  \[
    B_+^\eta=\left\{h\in L^2_+\, \colon\, \int_0^1 h^2\, \eta \,
    \mathrm{d} \mu = 1\right\}.
  \]
  The supremum in \eqref{eq:variational-formula} is reached for the right Perron
  eigenfunction of~$T_{\kk \eta}$ chosen in ~$ B_+^\eta$. 
\end{lemma}

\begin{proof}
  We first prove~\eqref{eq:re-explicit}. For all~$\eta \in \Delta$,
  the rank of the kernel operator~$T_{\kk^\varepsilon \eta}$ is
  smaller or equal to~$2$ and~%
  \(
  \mathrm{Im} (T_{\kk^\varepsilon \eta})
  \subset \mathrm{Vect}(\un, \alpha)
  \).
  The matrix of~$T_{\kk^\varepsilon \eta}$ in the basis~$(\un, \alpha)$ of the
  range of~$T_{\kk^\varepsilon \eta}$ is given by:
  \begin{equation}\label{eq:matrix-representation}
    \begin{pmatrix}
      \int \eta\, \mathrm{d} \mu & \int \alpha\, \eta\, \mathrm{d} \mu \\
   \varepsilon \int \alpha\, \eta\, \mathrm{d} \mu & \varepsilon \int
   \alpha^2\, \eta\, \mathrm{d} \mu 
    \end{pmatrix}.
  \end{equation}
  An explicit computation of the spectrum of this matrix yields
  Equation~\eqref{eq:re-explicit} for its largest eigenvalue.

  \medskip

The variational formula~\eqref{eq:variational-formula} is a direct
consequence of general Lemma~\ref{lem:min-max} below.
\end{proof}

 \begin{lemma}[Variational formula for~$R_e$ when~$\kk$ is symmetric]
   \label{lem:min-max}
   Suppose that~$\kk$ is a symmetric kernel on~$\Omega$ with a finite
   double norm in~$L^2$. Then, we have that for all~$\eta\in \Delta$:
  \begin{equation}
    \label{eq:Re-var}
    R_e(\eta) = \sup_{h\in B^\eta_+} \, \int_{\Omega \times \Omega}
  h(x) \eta(x)\, \kk(x,y) \, h(y) \eta(y) \, \mu(\mathrm{d}x) 
    \mu(\mathrm{d}y),
  \end{equation}
  where
 \[
   B_+^\eta=\left\{h\in L^2_+\, \colon\, \int_\Omega h^2\, \eta \,
       \mathrm{d} \mu = 1\right\}.
   \]
The supremum in \eqref{eq:Re-var} is reached for the right Perron
eigenfunction of~$T_{\kk \eta}$ chosen in
~$ B_+^\eta$. 
\end{lemma}

\begin{proof}
  For a finite measure~$\nu$ on~$(\Omega, \cf)$, as usual, we denote
  by~$L^2(\nu)$ the set of measurable real-valued functions~$f$ such
  that~$\int_\Omega f^2 \mathrm{d} \nu<+\infty$ endowed with the
  usual scalar product, so that~$L^2(\nu)$ is an Hilbert space.
  Let~$\eta\in \Delta$. We denote by~$\mathcal{T}_{\kk \eta}$ the
  integral operator associated to the kernel~$\kk \eta$ seen as an
  operator on the Hilbert space~$L^2(\eta \mathrm{d} \mu)$:
  for~$g\in L^2(\eta \mathrm{d} \mu)$ and~$x\in \Omega$ we
  have~$\mathcal{T}_{\kk \eta}(g)(x)= \int_\Omega \kk(x,y)\,
  \eta(y)\, g(y)\,\mu(\mathrm{d}y)$. The operator~$\mathcal{T}_{\kk \eta}$ is
  self-adjoint and compact since the double-norm of~$\kk$
  in~$L^2(\eta \mathrm{d} \mu)$ is finite. It follows from the
  Krein-Rutman theorem and the Courant–Fischer–Weyl min-max principle
  that its spectral radius is given by the variational formula:
  \[
\rho(\mathcal{T}_{\kk \eta})= \sup_{h\in B^\eta_+} \, \int_{\Omega \times \Omega}
  h(x) \, \kk(x,y) \, h(y) \, \eta(x)\mu(\mathrm{d}x) \, 
   \eta(y)\mu(\mathrm{d}y).
\]
Besides, the set~$L^2(\mu)$ is
  densely and continuously embedded in~$L^2( \eta \mathrm{d} \mu)$ and the
  restriction of~$\mathcal{T}_{\kk \eta}$ to~$L^2(\mu)$ is equal to~$T_{\kk
  \eta}$. Thanks to
  \cite[Lemma~3.2~(iii)]{ddz-theo}, we deduce that~$\rho(T_{\kk \eta})$ is equal to~$\rho(\mathcal{T}_{\kk \eta})$, which gives~\eqref{eq:Re-var}.

Let~$h_0$ be the right Perron eigenfunction of~$T_{\kk \eta}$ chosen
such that~$h_0\in B_+^\eta$. We get:
\[
  \int_{\Omega\times \Omega} \eta(x)h_0(x)\, \kk(x,y)\, \eta(y) h_0(y)\,
  \mu(\mathrm{d}x)\mu(\mathrm{d}y)= R_e(\eta)
  \int_{\Omega} \eta(x)h_0(x)^2\,
  \mu(\mathrm{d}x)=R_e(\eta).
\]
  Thus, the supremum in \eqref{eq:Re-var} is reached for~$h=h_0$. 
\end{proof}

\subsubsection{The kernel~$\kk^-$}\label{sec:k-}

Since~$\alpha$ is increasing, we have~$\mu(\alpha^2 = R_0) = 0$ and
thus the symmetric kernel~$\kk^-$ is~$\mu^{\otimes 2}$-a.s.\ positive.
It follows from Remark~\ref{lem:k>0-c} that~$\cmar=0$ and~$\cmir=1$,
and the strategy~$\un$ (resp.~$\zero$) is the only Pareto optimal as
well as the only anti-Pareto optimal strategy with cost~$c=0$
(resp.~$c=1$). Since the kernel~$\kk^-$ is \regular{} and symmetric,
and the non-zero eigenvalues of~$T_{\kk^-}$ are given by~$R_0=1$
and~$-\int \alpha^2\, \mathrm{d}\mu$, the latter being negative, we
deduce from Corollary~\ref{cor:regular}~\ref{cor:regular-conc}
that~$\cpu\subset \cpa$. On the one hand, if~$\eta$ is anti-Pareto
optimal with the same cost as $\etau\in \cpu$, one can use that~$R_e^-(\eta)=\int \eta \, \mathrm{d}\mu$
(as~$R_e^-(\etau)=\int \etau \, \mathrm{d}\mu$)
and~\eqref{eq:re-explicit} to deduce that~$\eta\in \orthStrats$. On
the other hand, if~$\eta$ belongs to~$\orthStrats$, we deduce
from~\eqref{eq:re-explicit}
that~$R_e(\eta)=\int \eta \, \mathrm{d}\mu$, and thus~$\eta$ is
anti-Pareto optimal. In conclusion, we get~$\cpa= \orthStrats$.
\medskip

  We now study the Pareto optimal strategies. We first introduce a notation inspired
  by the stochastic order of real valued random variables: we say
 that~$\eta_1, \eta_2 \in
\Delta$ with the same cost are in \emph{stochastic order}, and we write
$\eta_1 \leq_{\mathrm{st}} \eta_2$ if: 
\begin{equation}
   \label{eq:ineg-stoch}
\int _0^t \eta_1 \, \mathrm{d} \mu 
  \geq \int _0^t \eta_2 \, \mathrm{d} \mu 
  \quad\text{for all}\quad
  t\in [0, 1].
\end{equation}
We also write~$\eta_1 <_{\mathrm{st}} \eta_2$ if the inequality
in~\eqref{eq:ineg-stoch} is strict for at least one~$t\in (0, 1)$. If
$\eta_1 <_{\mathrm{st}} \eta_2$ and~$h$ is an increasing bounded
function defined on~$[0, 1)$, then we have:
\begin{equation}
   \label{eq:monot-stoch}
  \int _\Omega h\, \eta_1 \, \mathrm{d} \mu 
  < \int _\Omega h\, \eta_2 \, \mathrm{d} \mu .
\end{equation}
 Let~$c\in (0, 1)$ be fixed. Define the vaccination strategies with
 cost~$c$:
 \begin{equation}
   \label{eq:def-eta-gd}
\etag =\ind{[0,1-c)}
\quad\text{and}\quad
\etad = \ind{[c,1)}.
 \end{equation}
 In particular we have~$\etag <_{\mathrm{st}}\etad$ as~$\mu$ has no atom
 and~$\Omega$ as full support. Let~$\eta\notin\{\etag,\etad\}$
 be a vaccination strategy with cost~$c$; necessarily 
\[
  \etag <_{\mathrm{st}}\eta<_{\mathrm{st}}\etad.
\]

We now rewrite the function~$R_e^-$ in order to use the stochastic order on the
vaccination strategies.  We deduce from~\eqref{eq:re-explicit} that:
\begin{equation}\label{eq:re-explicit-}
  4 R_e^-(\eta)=4 \int \eta\, \mathrm{d} \mu - H(\eta)^2
  \quad\text{with}\quad
  H(\eta)= \sqrt{\int (1+\alpha)^2\eta\, \mathrm{d} \mu} -
  \sqrt{\int (1-\alpha)^2\eta\, \mathrm{d} \mu} .
\end{equation}
Then, using that~$\alpha$ is increasing and~$[-1,1]$-valued, we deduce
from~\eqref{eq:monot-stoch} (with~$h=(1+\alpha)^2$ and $h=-(1-\alpha)^2$) and the
definition of~$H$ in~\eqref{eq:re-explicit-} that:
\[
  H(\etag) < H(\eta) < H(\etad).
\]
This readily implies that 
$  R_e^-(\eta)> \min \left( R_e^-(\etag), R_e^-(\etad)\right)
$. Thus, among strategies of cost~$c$, the only possible Pareto optimal
ones are~$\etag$ and~$\etad$. We deduce that 
$ \cp\subset \leftStrats\cup\rightStrats$.

\subsubsection{The kernel~$\kk^+$}\label{sec:k+}

Arguing as for~$\kk^-$, we get that~$\cmar=0$ and~$\cmir=1$, and the strategy~$\un$
(resp.~$\zero$) is the only Pareto optimal as well as the only anti-Pareto optimal
strategy with cost~$c=0$ (resp.~$c=1$). Since the kernel~$\kk^+$ is \regular{} and
symmetric, and the non-zero eigenvalues of~$T_{\kk^+}$ given by~$R_0$ and~$\int_\Omega
\alpha^2\, \mathrm{d}\mu$ are positive, we deduce from
Corollary~\ref{cor:regular}~\ref{cor:regular-conv} that ~$\cpu\subset \cp$.

\medskip

Arguing as in Section~\ref{sec:k-} for the identification of the anti-Pareto optima based
on~\eqref{eq:re-explicit} (with~$\varepsilon=+$ instead of~$\varepsilon=-$) and using
that~$\cpu\subset \cp$ (instead of~$\cpu\subset\cpa$), we deduce that~$\cp=
\orthStrats $.

\medskip

We now consider the anti-Pareto optima. Let~$c\in (0, 1)$. We first start with some
comparison of integrals with respect to the vaccination strategies, with cost~$c$,~$\etag$
and~$\etad$ defined by \eqref{eq:def-eta-gd}. Let~$\eta$ be a strategy of cost~$c$ not
equal to~$\etag$ or~$\etad$ (recall that a strategy is defined up to the a.s.\ equality).
Consider the monotone continuous non-negative functions defined on~$[0,1]$:

\begin{equation*}
  \fg \, \colon \, x \mapsto \varphi^{-1} \left(\int_{[0,x)}
  \eta \, \mathrm{d} \mu\right),
  \quad \text{and}\quad
  \fd \, \colon \, x \mapsto \varphi^{-1} \left(1 - \int_{[x,1)} \eta
  \, \mathrm{d}\mu\right). 
\end{equation*}
 
Let~$i\in \{0,1\}$. Let~$\phi_i^{-1}$ denote the generalized left-continuous inverse
of~$\phi_i$. Notice that~$\eta(x)\, \mu(\mathrm{d} x)$-a.s.~$\phi^{-1}_i \circ
\phi_i(x)=x$.  The measure~$\eta_i\, \mathrm{d} \mu$ is the push-forward of~$\eta\,
\mathrm{d} \mu$ through~$\phi_i$, so that for~$h$ bounded measurable:
\begin{equation}
  \label{eq:push-forward4}
  \int h \, \eta\, \mathrm{d} \mu
  = \int h_i\, \eta_i\, \mathrm{d} \mu
  \quad\text{with}\quad
  h_i=h\circ \phi_i^{-1}.
\end{equation} 

Since~$\eta$ is not equal to~$\etag$ a.s., there exists~$x_0 < 1 - c$ such that,~$\fg(x) =
x$ for~$x \in [0,x_0]$ and~$\fg(x) < x$ for~$x \in (x_0, 1]$. Thus, we deduce
that~$\fg^{-1}(y) = y$ for all~$y \in [0, x_0]$ and~$\fg^{-1}(y)>y$ for all~$y \in (x_0,
1-c]$. Similarly, since~$\eta$ is not equal to~$\etad$ almost surely, there exists~$x_1 >
c$ such that~$\fd^{-1}(y) = y$ for all~$y \in (x_1,1]$ and~$\fd^{-1}(y) < y$ for all~$y
\in [c, x_1)$. Since~$\alpha$ is increasing and~$\mu$ has no atom and full support
in~$\Omega$, we deduce from from~\eqref{eq:push-forward4}, applied to~$h\alpha$, that
if~$h$ is positive a.s., then:
\begin{equation}\label{eq:comparison}
  \int \hg \, \alpha\, \etag\, \mathrm{d} \mu
  <
  \int h\, \alpha\, \eta\, \mathrm{d} \mu
  <
  \int \hd \, \alpha\, \etad\, \mathrm{d}\mu.
\end{equation}

\medskip

Let~$h$ be the right Perron eigenfunction of~$T_{\kk^+ \eta}$ chosen such that~$h\in
B_+^\eta$. Since~$\kk^+$ is positive a.s.\ and thus, irreducible, we have that~$h$ is
positive a.s. Thanks to Lemma~\ref{lem:2formula}, we have:
\begin{equation}\label{eq:eigenvaria}
  R_e^+(\eta) = \left( \int h\, \eta \,
  \mathrm{d} \mu \right)^2
  + \left( \int h\, \alpha\, \eta \, \mathrm{d} \mu \right)^2
  \quad\text{and}\quad
  \int h^2 \, \eta \, \mathrm{d} \mu =1 .
\end{equation}
We deduce from~\eqref{eq:push-forward4} that for~$i\in \{0,1\}$:
\[
  \int h\, \eta \, \mathrm{d} \mu =\int h_i\, \eta_i\, \mathrm{d} \mu
  \quad\text{and}\quad
  1=\int h^2\, \eta \, \mathrm{d} \mu =\int h_i^2\, \eta_i\,
  \mathrm{d} \mu.
\]
In particular~$h_i$ belongs to~$B_+^{\eta_i}$. Using that a.s.~$h>0$, we then deduce
from~\eqref{eq:eigenvaria} and~\eqref{eq:comparison} that:
\[
  R_e^+(\eta) < \max_{i\in \{0,1\}} 
  \left( \int h_i\, \eta_i \,
  \mathrm{d} \mu \right)^2
  + \left( \int h_i\, \alpha \, \eta_i\, \mathrm{d} \mu
  \right)^2
  \leq \max_{i\in \{0,1\}} R_e(\eta_i). 
\]
We conclude that only~$\etag$ or~$\etad$ can maximize~$R_e^+$ among the strategies of
cost~$c\in (0,1)$. We deduce that~$ \cpa\subset \leftStrats\cup\rightStrats$.

\subsection{An example where all parametrizations of the Pareto
  frontier have an infinite number of discontinuities}
\label{ssec:infinite}

The purpose of this section is to give a particular example of kernel on a continuous model where we
rigorously prove that the Pareto frontier cannot be greedily
parametrized, that is, parametrized by a continuous path in~$\Delta$ (as in the fully
symmetric circle), and that all the parametrizations have an arbitrary large
number of discontinuities (possibly countably infinite).
\medskip

We keep the setting from Section~\ref{sec:pp-f-rk2}. Without loss of
generality, we assume that~$R_0 = 1$, and we consider the kernel
$\kk^-=1- \alpha\otimes \alpha$ on~$\Omega=[0, 1)$ endowed with its
Lebesgue measure. We know from the previous section that, for any cost,
either~$\etag$ or~$\etad$ are Pareto optimal, and that all other
strategies are non-optimal. The idea is then to build an instance of the
function~$\alpha$ in
such a way that for some costs, one must vaccinate ``on the left'' and
for other costs ``on the right''. \medskip

Let~$N \in \lb 2, + \infty \rb$.
Consider an increasing sequence~$(x_n, \, n\in \lb 0, N\rb )$ such that
$x_0 = 1/2$,~$x_N = 1$ and~$\lim_{n\rightarrow \infty } x_n = 1$ if~$N = \infty$. For~$0
\leq n < N$, let~$p_n = x_{n+1} - x_n$ and assume that~$p_{n+1} < p_n$
for~$n\in \lb 0, N\lb$. For~$n
\geq 1$, let~$x_{-n}$ be the symmetric of~$x_n$ with respect to~$1/2$, \textit{i.e.},~$x_{-n} = 1 -
x_n$. The function~$\alpha$ is increasing piecewise linear defined on~$(0, 1)$ by:
\begin{equation}\label{eq:infinite-alpha}
  \alpha(x) = 
  \begin{cases}
    2x - 1, &\quad \text{for} \; x \in [x_{2m}, x_{2m + 1}), \\
    \\
    x - 1 + \frac{x_{2m-1} + x_{2m}}{2} &\quad \text{for} \; x \in [x_{2m-1}, x_{2m}). \\
  \end{cases}
\end{equation}
See Figure~\ref{fig:alpha} for an instance of the
graph of~$\alpha$ given in Example~\ref{ex:alpha-exple}. 
Note that for all~$n\in \lb 0, N\lb$, we have:
\begin{equation}\label{eq:int-equalities}
  \int_{x_{n}}^{x_{n + 1}} \alpha \, \mathrm{d} \mu =
 -\int_{x_{-n-1}}^{x_{-n}} \alpha \, \mathrm{d} \mu.
\end{equation}
This proves that the integral of~$\alpha$ over~$[0,1)$ is equal to~$0$. Of course,
$\sup_{[0,1)} \alpha^2 = 1 = R_0$. Hence,~$\alpha$ satisfies
Condition~\eqref{eq:alpha-cond}.

\begin{figure}
  \begin{subfigure}[T]{.5\textwidth}
    \centering
    \includegraphics[page=2]{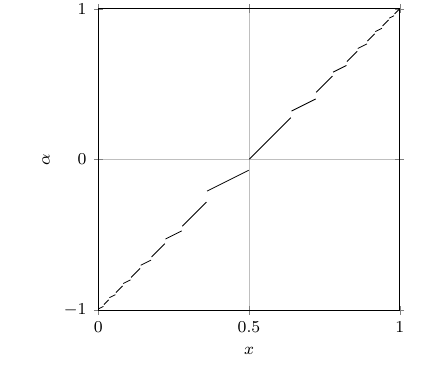}
    \caption{Graph of the function~$\alpha$ defined by
    Equation~\eqref{eq:infinite-alpha} and
    Example~\ref{ex:alpha-exple}.}
  \label{fig:alpha} 
  \end{subfigure}%
  \begin{subfigure}[T]{.5\textwidth}
    \centering
    \includegraphics[page=3]{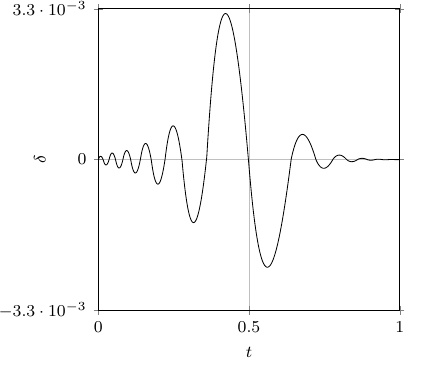}
    \caption{Graph of the corresponding function~$\delta$ defined in Equation~\eqref{eq:delta}.}
    \label{fig:delta}
  \end{subfigure}%
  \caption{Plots of the functions of interest in Section~\ref{ssec:infinite}.}
  \label{fig:functions}
\end{figure}

We recall that a function~$\gamma: [0, \cmir] \mapsto \Delta$ is a
parametrization of the Pareto frontier if for all~$c\in
[0, \cmir]$ the strategy~$\gamma(c)$ is Pareto optimal with cost
$C(\gamma(c))=c$. Now we can prove there exists no greedy
parametrization of the Pareto frontier of the kernel~$\kk^-$ and even
impose an arbitrary large lower bound for the number of discontinuities. 
  
\begin{proposition}\label{prop:infinite2}
  Let~$N \in \lb 2, + \infty \rb$. Consider the kernel~$\kk^- = 1 - \alpha\otimes \alpha$ from \eqref{eq:k+} on~$\Omega=[0, 1)$ endowed with its Lebesgue measure, with~$\alpha$
  given by~\eqref{eq:infinite-alpha}. Then, any parametrization of
  the Pareto frontier has at least~$2N-2$ and at most~$20N-2$
  discontinuities.
\end{proposition}

The proof is given at the end of this section, and relies on the following
technical lemma based on the comparison of the following monotone paths
$\gag$ and~$\gad$ from~$[0, 1]$ to~$\Delta$:
\begin{equation}
  \gag(t) = \ind{[0,t)}, \quad \text{and} \quad
  \gad(t) = \ind{[1-t,1)}, \quad t \in [0,1]
\end{equation}
which parameterizes~$\leftStrats$ and~$\rightStrats$ as
$\gag([0, 1])=\leftStrats$ and~$\gad([0, 1])=\rightStrats$. Notice
that strategies~$\gag(t)$ and~$\gad(t)$ have the same cost~$1-t$. 

Consider the function~$\delta:[0, 1] \rightarrow \R$ which, according
to Proposition~\ref{prop:rank2}, measures the
difference between the effective reproduction numbers at the extreme
strategies:
\begin{equation}\label{eq:delta}
  \delta(t) = R_e(\gag(t)) - R_e(\gad(t)).
\end{equation}
The function~$\delta$ is continuous and~$\delta(0) = \delta(1) = 0$; see
for example Figure~\ref{fig:delta}  for  its  graph  when  $\alpha$  is  taken  from
Example~\ref{ex:alpha-exple}.  We  say  that~$t\in  (0, 1)$  is  a  zero
crossing of~$\delta$  if~$\delta(t)=0$ and there  exists $\varepsilon>0$
such  that~$\delta(t+r)\delta(t-r)<0$ for  all~$r\in (0,  \varepsilon)$.
The  following  result  gives  some  information on  the  zeros  of  the
function~$\delta$.

\begin{lemma}\label{lem:infinite}
  Under the assumptions of Proposition~\ref{prop:infinite2}, the
  function~$\delta$ defined in~\eqref{eq:delta} has at least~$2N -2$
  zero-crossings in~$(0,1)$ and at most~$20N$ zeros in~$[0, 1]$. Besides,
  if~$N = \infty$,~$0$ and~$1$ are the only accumulation points of the
  set of zeros of~$\delta$.
\end{lemma}

\begin{proof}
  Using the explicit representation of~$R_e^-$ from
  Lemma~\ref{lem:2formula}, see~\eqref{eq:re-explicit} with~$\varepsilon=-$, we get the function~$\delta$ can be expressed as:
  \begin{equation}
 2 \delta(t) = \Vd(t) - \Vg(t) + \sqrt{
	\Vg(t) ^2 - \Mg(t)^2} - \sqrt{
   \Vd(t) ^2 - \Md(t)^2} ,
  \end{equation}
  where, as~$\int \alpha \, \mathrm{d}\mu=0$:
  \begin{equation*}
    \Mg(t) = 2 \int_0^t \!\!\alpha \, \mathrm{d} \mu,
    \,\,\,
    \Vg(t) = t+\int_0^t \!\!\alpha^2\, \mathrm{d} \mu,
    \, \,\, 
    \Md(t)=\Mg(1-t)\,\, \,
    \text{and} \quad \Vd(t) = t+ \int_{1-t}^1 \!\!\alpha^2
    \, \mathrm{d} \mu.
  \end{equation*}
Elementary computations give that for all~$n\in \lb 0, N\lb$:
  \begin{equation}
    \int_{x_{n}}^{x_{n + 1}} \alpha(x)^2 \, \mathrm{d}x -
    \int_{x_{-n-1}}^{x_{-n}} \alpha(x)^2 \, \mathrm{d} x =
    \frac{(-1)^n p_n^3}{4},
  \end{equation}
where we recall that~$p_n=x_{n+1} - x_n$. Hence, we obtain that for all~$n \in \lb -N, N \rb$:
  \begin{equation}
\Vd(x_n) - \Vg(x_n) =
    \frac{1}{4} \sum\limits_{i=\abs{n}}^\infty (-1)^i p_i^3.
  \end{equation}
  Since the sequence~$(p_n, \, n \in \lb 0, N \lb)$ is decreasing, we
  deduce that the sign of~$\Vd(x_n) - \Vg(x_n)$ alternates depending
  on the parity of~$n\in \rb -N, N \lb$: it is positive for odd~$n$
  and negative for even~$n$. The same result holds for the numbers~$\delta(x_n)$ since~$\Mg(x_n) =\Md(1-x_n)$ for all~$n \in \lb -N, N \rb$ according to~\eqref{eq:int-equalities} (use
  that, with~$b>0$, the function~$x\mapsto x - \sqrt{x^2 - b^2}$ is
  decreasing for~$x\geq \sqrt{b}$ as its derivative is
  negative). This implies that~$\delta$ has at least~$2N-2$
  zero-crossings in~$(0, 1)$. \medskip

  We now prove that~$\delta$ has at most~$20N$ zeros in~$[0, 1]$ and
  that~$0$ and~$1$ are the only possible accumulation points of the set
  of zeros of~$\delta$. It is enough to prove that~$\delta$ has at most
  10 zeros on~$[x_n, x_{n+1}]$ for all finite~$n\in \lb -N, N \lb$. On
  such an interval~$[x_n, x_{n+1}]$, the function~$\alpha$ is a
  polynomial of degree one. Consider first~$n$ odd and non-negative, so that for~$t\in [x_n, x_{n+1}]$, we get that with~$a=1-(x_n+x_{n+1})/2$:
  \begin{align*}
  \Mg(t)=2t^2 - 2t + b_1, \quad
    & \Vg(t)=\frac{4}{3} t^3 - 2 t^2 + 2t + b_2, \\
\Md(t)=t^2 -2a t + b_3 , \quad
        & \Vd(t)=- \frac{1}{3} t^3 + a t^2 + (1- a^2) t + b_4, 
  \end{align*}
  where~$b_i$ are constants. If~$t$ is a zero of~$\delta$, then it is
  also a zero of the polynomial~$P$ given by: 
  \[
P= 4(\Vd -\Vg)\left(\Vg \Md^2 - \Vd \Mg^2\right)- \left( \Mg^2 -
    \Md^2\right)^2. 
  \]
Since the degree of~$P$ is exactly 10, it has at most 10 zeros. Thus~$\delta$ has at most
10 zeros on~$[x_n, x_{n+1}]$. This ends the proof. 
\end{proof}

\begin{proof}[Proof of Proposition~\ref{prop:infinite2}]
  According to Proposition~\ref{prop:rank2}, the only
  possible Pareto
  strategies of cost~$c=1-t\in [0, 1]$ are $\gag(t)$ and~$\gad(t)$,
  and only one of them is optimal when $\delta\neq 0$. 
   A zero crossing of the function~$\delta$ on~$(0, 1)$ therefore corresponds to a
   discontinuity of any parametrization of the Pareto frontier. We
   deduce from Lemma~\ref{lem:infinite} that in~$(0, 1)$ there are at least~$2N-2$
   and at most~$20N- 2$ zeros crossing and thus
   discontinuities of any parametrization of the Pareto frontier. 
\end{proof}

\begin{figure}
  \begin{subfigure}[T]{.5\textwidth} \centering
    \includegraphics[page=1]{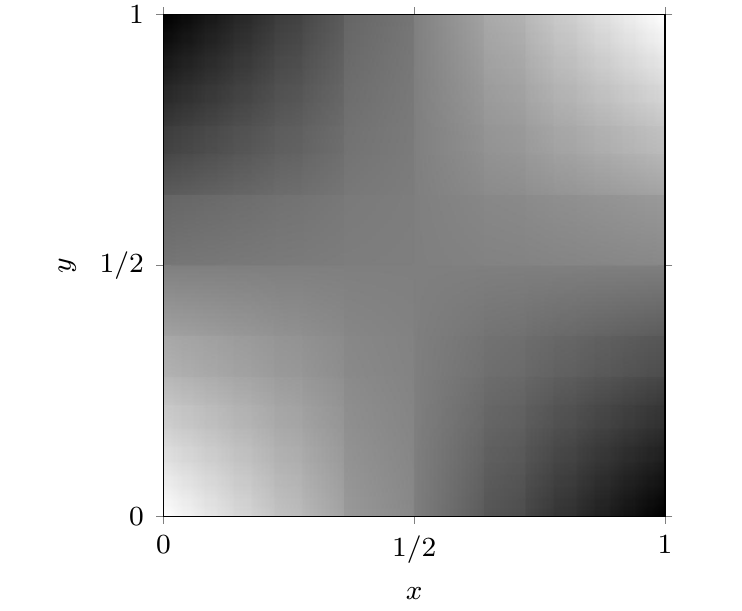}
    \caption{Grayplot of the kernel~$\kk^-(x,y) = 1 - \alpha(x) \alpha(y)$ where~$\alpha$ is plotted
   in Figure~\ref{fig:alpha}.}\label{fig:regular_kernel_1} 
  \end{subfigure}%
  \begin{subfigure}[T]{.5\textwidth} \centering
    \includegraphics[page=4]{disconnected}
    \caption{Solid line: the Pareto frontier~$\F$; dashed line: the
      anti-Pareto frontier~$\AF$ (which corresponds to the uniform strategies); blue
    region: all possible outcomes~$\FF$.}
    \label{fig:pareto_frontier_regular_1}
  \end{subfigure}%
  \caption{An example of a \regular{} kernel operator of rank 2.}\label{fig:infinite-component}
\end{figure}

\begin{example}
  \label{ex:alpha-exple}
  In Figure~\ref{fig:alpha}, we have represented the function~$\alpha$ defined by~\eqref{eq:infinite-alpha}
  where:
  \begin{equation*}
    x_n = \frac{1}{2} \log_{12}(12(n+1)), \quad 0 \leq n \leq N=11.
  \end{equation*}
  Hence, the mesh~$(x_n, \, -N \leq n \leq N)$ is composed by~$2N +1 = 23$ points. The
  graph of the corresponding function~$\delta$ defined in~\eqref{eq:delta} is
  drawn in Figure~\ref{fig:delta}. The grayplot of the kernel~$\kk^-=1-
  \alpha \otimes \alpha$ is given in
  Figure~\ref{fig:regular_kernel_1} and the associated Pareto and
  anti-Pareto frontiers are plotted in
  Figure~\ref{fig:pareto_frontier_regular_1}.
\end{example}

\section{Geometric kernels on the sphere}\label{sec:geometric}

A geometric random graph is an undirected graph constructed by assigning a random point in
a latent metric space to each node and by connecting two nodes according to a certain
probability that depends on the distance between their latent point. Because of its
geometric structure, this model is appealing for a wide-range of applications such as
wireless networks modelling \cite{DegreeDistribuHekmatNone}, social networks
\cite{LatentSpaceApHoff2002} and biological networks \cite{FittingAGeomeHigham2008}. A
geometric random graph model can be represented as a symmetric kernel defined on the
latent space (also called \emph{graphon}) according to \cite{lovasz_large_2012}.

In this section, we focus our study on the latent space given by the unit sphere. In 
Section~\ref{sec:sphere} we present the mathematical model, and give in
Section~\ref{sec:sphere-uni} sufficient conditions on the kernel for uniform  strategies
to be Pareto or anti-Pareto optimal. Section~\ref{sec:sphere-aff} is  devoted to the
explicit descriptions of the Pareto and anti-Pareto optimal  vaccination strategies in the
affine case.

\subsection{The model}\label{sec:sphere}

Let $d\geq 2$.  Let~$\Omega = \Sd$ be the unit sphere of the Euclidean~$d$-dimensional
space~$\R^d$ endowed with the usual Borel $\sigma$-field and the uniform probability
measure~$\mu$. Let~$\langle \cdot, \cdot \rangle$ denote the usual scalar product on
$\R^d$ and let
\[
  \delta(x,y) = \arccos(\braket{x,y})
\]
denote  the  geodesic  distance  between~$x,y\in  \Sd$.  By  symmetry,  the distribution 
on~$[-1, 1]$  of  the scalar  product  of two  independent uniformly  distributed  random 
variables   in~$\Sd$  is  equal  to  the distribution of  the first  coordinate of  a
uniformly  distributed unit vector: it  is the probability  measure on  $[-1, 1]$ with 
density with respect       to      the       Lebesgue      measure       proportional to
the function $w_d$ defined on $[-1, 1]$ by:
\[
  w_d(t)=(1 -  t^2)^{(d-3)/2}\, \ind{(-1, 1)}(t).
\]
In  particular, we  deduce from  the  Funk-Heck formula  (take $n=0$  in
\cite[Theorem~1.2.9]{ApproximationTDaiF2013}) that  for any non-negative measurable
function~$h$ defined on~$[-1, 1]$ and $x\in \Sd$, we have:
\begin{equation}\label{eq:h-Sd2}
  \int_{\Sd} h(\langle x,y \rangle )\, 
  \mu(\mathrm{d} y)
  = c_d\int_{-1}^ 1 h(t) \, w_d(t) \, \mathrm{d}t
  \quad\text{with}\quad
  c_d=\frac{\Gamma(\frac{d}{2})}{\Gamma(\frac{d-1}{2}) \sqrt{\pi}} \cdot
\end{equation}

We consider a symmetric kernel~$\kk$ on $\mathbb{S}^{d-1}$ corresponding to a geometric
random graph model on~$\mathbb{S}^{d-1}$, given by:
\begin{equation}
  \label{eq:kernel-sphere}
  \kk(x,y) = p(\braket{x,y})=f\circ \delta (x,y), \quad x,y \in \mathbb{S}^{d-1},
\end{equation}
where~$p \, \colon \, [-1,1] \to \R_+$ is a measurable function and~$f = p \circ \cos \,
\colon \, [0, \pi] \to \R_+$. We assume that~$\kk$ has finite double norm on~$L^2$; thanks
to~\eqref{eq:h-Sd2}, this is equivalent to:
\begin{equation}\label{eq:p-L2}
  \int_{-1}^1 p(t)^2 \, w_d(t) \, \mathrm{d}t
  = \int_0^\pi f(\theta)^2 \sin(\theta)^{d-2} \mathrm{d}\theta < \infty.
\end{equation}
By symmetry, using that the scalar product and the measure~$\mu$ are invariant by
rotations, we deduce that the kernel~$\kk$ is a
\regular{} kernel. According to~\eqref{eq:r0-regular} and using~\eqref{eq:h-Sd2}, we get
that the basic reproduction number is given by:
\begin{equation}
  R_0 = c_d \int_{-1}^{1} p(t) \, w_d(t) \, \mathrm{d}t
  = c_d \int_0^\pi f(\theta) \sin(\theta)^{d-2}\, \mathrm{d}\theta.
\end{equation}

By \cite[Theorem~1.2.9]{ApproximationTDaiF2013}, the eigenvectors of the
symmetric operator~$T_\kk$  on $L^2(\Sd)$  are the  spherical harmonics,
and  in  particular,   they  don't  depend  on   the  function~$p$.
We recall the linear subspace of spherical
   harmonics of degree~$n$ for $n \in \N$ has dimension $d_n$ given by
   $d_0=1$ and for $n\in \N^*$:
  \[
     d_n = \frac{2n+ d-2}{n+d-2} \binom{n + d - 2}{n}.
   \]
   The  corresponding eigenvalues~$(\lambda_n,  n\in \N)$  are real  and
   given by:
  \begin{equation}\label{eq:vpn-sphere}
     \lambda_n =c_d
     \int _{-1}^1 p(t) \, \frac{G_n(t)}{ G_n(1) } \, w_d(t) \, \mathrm{d}t
     = c_d
	     \int_0^\pi f(\theta)\, \frac{G_n(\cos(\theta))}{ G_n(1) }
             \sin(\theta)^{d-2}\, \mathrm{d}\theta,
  \end{equation}
  where~$G_n$ is the Gegenbauer polynomial of degree~$n$ and
  parameter~$(d-2)/2$
       (see
   \cite[Section~B.2]{ApproximationTDaiF2013} with~$G_n=C_n^{(d-1)/2}$).
   We simply recall that $G_0=\un$ and that for~$d=2$, the Gegenbauer polynomials are,
   up to  a multiplicative  constant, the  Chebyshev polynomials  of the
   first kind:
   \[
     G_n(\cos(\theta))=\frac{2}{n} \, \cos(n\theta) \quad\text{for
       $\theta\in [0, \pi]$ and $n\in \N^*$};
   \]
and that for $d\geq 3$, $r\in (-1,1)$ and~$\theta\in [0, \pi]$:
   \[
     \sum_{n=0}^\infty r^n G_n(\cos(\theta))=(1+r^2 - 2r
     \cos(\theta))^{-(d-2)/2}
     \quad\text{and}\quad
     G_n(1)=\binom{n+d-3}{n} \quad\text{for~$n\in \N^*$}. 
   \]

   Thus,  if~$\lambda\neq  0$ is  an  eigenvalue  of $T_\kk$,  then  its
   multiplicity  is  the  sum  of all  the  dimensions~$d_n$  such  that
   $\lambda_n = \lambda$.  The eigenvalue  $R_0$ (associated to
   the eigenvector $\un$ which is the spherical harmonic of degree 0) is
   in fact simple according to the next Lemma.

   \begin{lemma}
Let $\kk$ be a kernel  on $\Sd$ given
by~\eqref{eq:kernel-sphere}, with finite double norm and such that $R_0>0$. Then the 
kernel $\kk$ is \regular{} and irreducible, and  its eigenvalue $R_0$ is simple.
   \end{lemma}
   \begin{proof}
     The kernel $\kk$ is trivially a \regular{}  kernel. 
     Since~$d_0=1$, we only need to prove that $\lambda_n < \lambda_0 =
     R_0$ for all $n \in \N^*$ to get that $R_0$ is simple, and then use
     Lemma~\ref{lem:reduction} to get 
     that $\kk$ is irreducible.

     According to \cite[Equation~22.14.2]{abramowitz1972handbook} or
     \cite[Section~3.7.1]{atkinson2012}, we get that
     $|G_n(t)|\leq G_n(1)$ for $t\in [-1, 1]$.  Since $G_n$ is a
     polynomial,  the inequality is strict for a.e.
     $t\in [-1, 1]$.  Using~\eqref{eq:vpn-sphere}, we obtain that
     $\lambda_n < \lambda_0$ for all $n \in \N^*$.
   \end{proof}

\begin{example}[The circle:~$d=2$]
  \label{ex:sphere=circle}
  In   case~$d    =   2$,   we   identify    the   circle~$\mathbb{S}^1$
  with~$\Omega=\R/{2\pi\Z}$           and           the           scalar
  product~$\langle  \theta,\theta'  \rangle=\cos(\theta-\theta')$.
  The kernel~$\kk$ from~\eqref{eq:kernel-sphere}
  is           the           convolution          kernel           given
  by~$k(\theta,   \theta')=p  (\cos(\theta-\theta'))=f(\theta-\theta')$,
  where $f$ is symmetric non-negative and $2\pi$ periodic and its
  restriction to $[0, \pi]$ is square integrable.  Then,
  we can consider the development in~$L^2([0, \pi])$ of $f$ as a Fourier series:
  \begin{equation}
    \label{eq:fourier}
    f(\theta) = \sum_{n = 0}^\infty a_n(f) \cos(n \theta), \quad \theta \in [0, \pi],
\end{equation}
  where:
  \begin{equation}
    \label{eq:coef-fourier}
    a_0(f) = \frac{1}{\pi} \int_0^\pi f(\theta) \, \mathrm{d} \theta
    \quad\text{and}\quad 
    a_n(f) = \frac{2}{\pi} \int_0^\pi \cos(n \theta) f(\theta) \,
    \mathrm{d} \theta \quad\text{for}\quad n \geq 1.
\end{equation}
  It follows from Equation~\eqref{eq:fourier} that the kernel has the following
  decomposition in~$L^2([0, 2\pi)^2)$:
  \begin{equation}
    \label{eq:sphere-circle:k}
    \kk(\theta,\theta') = a_0(f) + \sum_{n = 1}^\infty a_n(f) \, \big(\cos(n
    \theta) \cos(n \theta') +
    \sin(n \theta) \sin(n \theta')\big), \quad \theta, \theta' \in [0,
    2\pi). 
  \end{equation}
  Assume that~$a_0(f)>0$, that  is, $f$ is non-zero.  Then, the spectral
  radius~$R_0=a_0(f)$ is an eigenvalue  with multiplicity one associated
  to the eigenfunction~$\un$  (and thus~$\kk$  is a \regular{}  kernel). The
  other   eigenvalues  are   given  by~$\lambda_{n}   =  a_n(f)/2$   for
  all~$n  \geq 1$  and, when  non zero and distinct,  have
  multiplicity~$2$.
\end{example}

\subsection{Sufficient condition for convexity or concavity}
\label{sec:sphere-uni}

We would like to provide conditions on the function~$f$ or~$p$ that
ensure that the eigenvalues~$(\lambda_n, \, n \geq 1)$ given by
\eqref{eq:vpn-sphere} of the operator~$T_\kk$ with the kernel~$\kk$
defined by~\eqref{eq:kernel-sphere} are all non-negative or all
non-positive so that~$R_e$ is convex or concave according to
Corollary~\ref{cor:regular}. Schoenberg's theorem, see
\cite[Theorem~14.3.3]{ApproximationTDaiF2013} or
\cite[Theorem~1]{StrictlyAndNoGneiti2013}, characterizes the continuous
function~$f$ such that the kernel~$\kk$ is positive semi-definite (and
thus the eigenvalues~$(\lambda_n, \, n \geq 1)$ are all non-negative) as
those with non-negative Gegenbauer coefficients:
$f=\sum_{n=0}^\infty a_n \, G_n$, where the convergence is
uniform on~$[-1, 1]$, with~$a_n\geq 0$ for all~$n\in \N$ and
$\sum_{n=0}^\infty a_n \, G_n(1)$ finite. When~$d=2$, this corresponds
to the Böchner theorem.
We refer to \cite{StrictlyAndNoGneiti2013} and references therein for
some characterization of functions~$f$ such that the kernel~$\kk$
from~\eqref{eq:kernel-sphere} is definite positive. 
We end this section with some examples. 

\begin{example}
  \label{ex:square}
  We give an elementary example in the
  setting of Example~\ref{ex:sphere=circle} when $d=2$.
  Set
  \[
    f_+(\theta)= (\pi - \theta)^2
    \quad\text{and}\quad
    f_-(\theta)= \pi^2- (\pi - \theta)^2
    \quad\text{for~$\theta\in [0, \pi]$.}
  \]
   We can compute the Fourier coefficients of~$f_+$ and~$f_-$ as:
  \[
    (\pi - \theta)^2 = \frac{\pi^2}{3} +
    \sum_{n = 1}^\infty \frac{4}{n^2} \cos(n \theta), \quad \theta \in [0,\pi].
  \]
  Using Corollary~\ref{cor:regular} and \cite[Theorem~5.5]{ddz-Re}, we
  deduce that the function~$R_e$ associated to the convolution
  kernel~$\kk=f_+\circ \delta$ is convex and~$\cpu\subset \cp$;
  whereas the function~$R_e$ associated to the convolution
  kernel~$\kk=f_-\circ \delta$ is concave and~$\cpu\subset \cpa$.
\end{example}

\begin{example}[Kernel from a completely monotone function]
  \label{ex:comp}
  Let~$\varphi$ be a continuous non-negative function defined
  on~$\R_+$, such that~$\varphi$ is completely monotone, that
  is,~$\varphi$ is infinitely differentiable on~$(0, +\infty )$ and~$(-1)^n \varphi^{(n)} \geq 0$ on~$(0,+\infty)$ for all~$n \geq 1$.
  Using \cite[Theorem~7]{StrictlyAndNoGneiti2013}, we get that the
  geometric kernel~$\kk = f\circ \delta$ on~$\Sd$, with~$d=2$, where~$f = \varphi_{[0,\pi]}$ is positive definite (thus all the
  eigenvalues of~$T_\kk$ are non-negative). Thanks to
  Corollary~\ref{cor:regular} and \cite[Theorem~5.5]{ddz-Re}, we
  deduce that~$R_e$ is convex and the uniform strategies are Pareto
  optimal:~$\cpu\subset \cp$.
\end{example}

\begin{example}[Kernel from a Bernstein function]
  \label{ex:Bern}
  Let~$\varphi$ be a Bernstein function, that is a non-negative~$C^1$
  function defined on~$\R_+$ such that~$\varphi^{(1)}$ is completely
  monotone. Assume furthermore that~$\sup_{\R_+} \varphi < \infty$.
  This gives that the function~$t \mapsto (\sup_{\R_+} \varphi) - \varphi(t)$ defined on~$\R_+$
  is continuous non-negative and completely monotone. Consider the
  geometric kernel~$\kk = f\circ \delta$ on~$\Sd$, with~$d=2$,
  where~$f = \varphi_{[0,\pi]}$. We deduce from
  \cite[Theorem~7]{StrictlyAndNoGneiti2013}, see also the previous
  example, that all the eigenvalues of the integral operator~$T_\kk$,
  but for $R_0$, are non-positive. Then, using
  Corollary~\ref{cor:regular} and \cite[Theorem~5.5]{ddz-Re}, we get
  that~$R_e$ is concave and the uniform strategies are anti-Pareto
  optimal:~$\cpu\subset \cpa$.
\end{example}

\begin{example}[Kernel from a power function] \label{ex:puissance} Let~$m \geq 1$ be an
  integer and~$\theta_0 \geq \pi$ a real number.
  Using \cite[Lemma~4]{StrictlyAndNoGneiti2013}, we get that for 
  $f(\theta) = (\theta_0 - \theta)^m$, $R_e$ is convex and the uniform
  vaccination strategies are Pareto optimal; and that
  for~$f(\theta) = \theta_0^m - (\theta_0 - \theta)^m$,
  $R_e$ is concave and the uniform strategies are anti-Pareto optimal.
\end{example}

\begin{example}[The function $p$ is a series] According  to  
  \cite[Theorem~1]{StrictlyAndNoGneiti2013},  if  the function $p$  can be  written as
  $p(t)=\sum_{n\in  \N} b_n  \, t^n$ with $b_n$ non-negative and $\sum_{n\in  \N} b_n$
  finite, then, for all $d\geq 2$, the kernel $\kk$ defined by~\eqref{eq:kernel-sphere} on
  $\Sd$  is semi definite  positive (and definite positive  if the coefficients  $b_n$ 
  are  positive  for infinitely  many  even  and infinitely many odd  integers $n$), and
  thus the function $R_e$ is convex and  the uniform  vaccination strategies are  Pareto
  optimal thanks          to         Corollary~\ref{cor:regular}          and
  \cite[Theorem~5.5]{ddz-Re}. \medskip

  Consider    the     kernel    $\kk(x,y)=|x-y|^\nu$,     that    is, $p(t)=2^{\nu/2}  
  |1-t|^{\nu/2}$,   with   $\nu>  (1-d)/2$,   so   that condition~\eqref{eq:p-L2}         
  holds.              According to~\cite[Section~3.7.1]{atkinson2012} and  Equation~(3.74)
  therein, for $n\geq  1$, the eigenvalues  $\lambda_n$ have the same  sign as
  $\prod_{k=0}^{n-1}   (-\nu  +2k)$.    So,   we   deduce  that   for $\nu\in ((1-d)/2,
  0)$ all the eigenvalues are positive  and thus $R_e$ is  convex  and  the  uniform 
  vaccination  strategies  are  Pareto optimal;  and  for  $\nu\in  (0,   2)$  all  the 
  eigenvalues  (but $\lambda_0=R_0>0$) are negative  and thus $R_e$ is  concave and the
  uniform strategies are anti-Pareto optimal. The latter case is also a consequence of
  \cite[Theorem~1]{StrictlyAndNoGneiti2013}, whereas the former case is not a direct
  consequence of \cite[Theorem~1]{StrictlyAndNoGneiti2013} as $\sum_{n\in \N} b_n$ is not
  finite when $\nu$ is negative. 
\end{example}

\subsection{The affine model}\label{sec:sphere-aff}

Recall $\Omega= \Sd\subset  \R^d$, with $d\geq 2 $, is  endowed with the uniform
probability measure $\mu$. In  this section, we suppose that the model    is   affine,
that is,   the    kernel    $\kk$    given by~\eqref{eq:kernel-sphere}, \emph{i.e.}  $
\kk(x,y) = p(\braket{x,y})$, has a linear envelope:
\[
  p(t)=a + b t \quad\text{for}\quad t\in [-1, 1].
\]
The kernel~$\kk$ being  non-negative non-constant with~$R_0>0$ is equivalent to the
condition  $a\geq |b| > 0$ on the parameter~$(a,b)$.  This model corresponds to 
$f(\theta) =  a + b \cos(\theta)$  for $\theta\in [0, \pi]$.   Since the Gegenbauer
polynomials~$(G_n,  n\in \N)$ are orthogonal  with respect to the  measure~$ w_d(t)\, 
\mathrm{d} t$,  we  easily  deduce from~\eqref{eq:vpn-sphere} that the non-zero
eigenvalues of the integral operator~$T_\kk$      are~$R_0=a$       (with     
multiplicity~$d_0=1$) and~$\lambda_1= b/d$ (with multiplicity~$d_1=d$).  \medskip

For~$x\in \Sd$ and~$t\in [-1,1]$, we consider the following balls centered at~$x$:
\[
  B(x,t)=\{y\in \Sd\, \colon\, \langle x,y \rangle \geq t\}.
\]
Recall that strategies are defined up to equality almost surely. We 
consider the
following sets of extremal strategies, for~$x\in \Sd$:
\[
  \Sball= \left\{ \ind{B(x,t)}\, \colon\, x\in \Sd, \, t \in [-1,1]\right\},
\]
as well as the following set of strategies which contains the
set of uniform strategies
$ \cpu=\{ t\un\, \colon\, t\in [0,1]\}$:
\[
\So=\left\{\eta\in \Delta\, \colon \, \int_\Sd x \, \eta(x) \, 
\mu(\mathrm{d} x) =0\right\}.
\]

\begin{proposition}
  \label{prop:rank2-sphere}
  Let~$a\geq |b|>0$ and the kernel~$\kk$ on~$\Sd$, with~$d\geq 2$, be given by: \[
  \kk(x,y)=a+ b \langle x,y \rangle. \]
  \begin{enumerate}[(i)]
    \item\label{item:r2-k+} \textbf{The case~$b>0$.} A strategy is Pareto optimal if and
      only if it belongs to $\So$. In particular, for any~$c\in[0,1]$, the strategy
      $(1-c)\un$ costs~$c$ and is Pareto optimal. The  anti-Pareto optimal strategies
      are~$\ind{B(x,t)}$ for~$x\in \Sd$ and $t\in [-1,1]$.  In other words: \[ \cp=\So
      \quad\text{and}\quad \cpa= \Sball. \]
    \item\label{item:r2-k-} \textbf{The case~$b<0$.} A strategy is anti-Pareto optimal if
      and only if it belongs to~$\So$. In particular, for any~$c\in[0,1]$,  the
      strategy~$(1-c)\un$ costs~$c$ and is anti-Pareto optimal. The Pareto optimal
      strategies are~$\ind{B(x,t)}$ for~$x\in \Sd$ and $t\in [-1,1]$. In other words: \[
      \cpa= \Sball \quad\text{and}\quad \cpu=\So. \]
  \end{enumerate}
  In both cases, we have~$ \cmir =1$ and~$\cmar=0$. 
\end{proposition}

\begin{example}
  We consider the kernel~$\kk=1 +b \langle \cdot, \cdot \rangle$ on the sphere~$\Sd$,
  with~$d=2$. This model has the same Pareto and anti-Pareto frontiers as the equivalent
  model  given by~$\Omega = [0,1)$ endowed with the Lebesgue measure and the kernel~$(x,y)
  \mapsto 1 + b \cos(\pi(x-y))$, where the equivalence holds  in the sense of
  \cite[Section~7]{ddz-theo}, with an obvious deterministic coupling $\theta\mapsto
  \exp(2i\pi \theta)$.  We provide the Pareto and anti-Pareto frontiers in
  Figure~\ref{fig:optim_regular} with~$b=1$ (top) and with~$b=-1$ (bottom).
\end{example}
 
\begin{proof}
  The proof of Proposition~\ref{prop:rank2-sphere} is decomposed in four steps.
  
  \textit{Step 1:}~$R_e(\eta)$ is the eigenvalue of a~$2\times 2$
  matrix~$M(\eta)$. Without loss of generality, we shall assume
  that~$R_0=a=1$. Since~$\kk$ is positive a.s., we deduce
  that~$\cmir=1$ and~$\cmar=0$ thanks to Lemma~\ref{lem:k>0-c}; and
  the strategy~$\un$ (resp.~$\zero$) is the only Pareto optimal as
  well as the only anti-Pareto optimal strategy with cost 0 (resp.
  1). So we shall only consider strategies~$\eta\in \Delta$ such
  that~$C(\eta)\in (0, 1)$. \medskip

  Let~$z_0\in \Sd$. Write~$b=\varepsilon \lambda^2$ with $\varepsilon\in \{-1,+1\}$
  and~$\lambda\in (0, 1]$, and define the function $\alpha$ on~$\Sd$ by:
  \[
    \alpha=\lambda \, \langle \cdot, z_0 \rangle.
  \]
  Let~$\eta\in \Delta$ with cost~$c\in (0, 1)$. As~$\cmir=1> C(\eta)$,
  we get that~$R_e(\eta)>0$. We deduce
  from the special form of the kernel~$\kk$ that the eigenfunctions of $T_{\kk \eta}$ are
  of the form~$\zeta +\beta \lambda \langle \cdot, y \rangle$ with~$\zeta, \beta\in \R$
  and~$y\in \Sd$. Since~$R_e(\eta)>0$, the right Perron eigenfunction, say~$h_\eta$, being
  non-negative, can be chosen such that~$h_\eta=1+\beta_\eta \lambda\langle \cdot, y_\eta
  \rangle$ with~$\beta_\eta\geq 0$ and~$\beta_\eta\lambda\leq 1$. Up to a rotation on the
  vaccination strategy, we shall take $y_\eta=z_0$, that is:
  \[
    h_\eta=1+ \beta_\eta \, \alpha.
  \]
  From the equality~$R_e(\eta) h_\eta= T_{\kk \eta} h_\eta$, we deduce that:
  \begin{align}
    \label{eq:Re-pour-1}
    R_e(\eta)&= \int_\Sd \eta(y) \, \mu( \mathrm{d}y) + \beta_\eta
    \,\lambda\int_\Sd \eta(y) \,
    \langle y, z_0 \rangle \, \mu(\mathrm{d} y),\\
    \label{eq:Re-pour-alpha}
    \beta_\eta R_e(\eta) \langle \cdot, z_0 \rangle
	     &=\varepsilon \lambda \int_\Sd \eta(y) \, \langle \cdot, y \rangle\,
	     \mu( \mathrm{d}y)+ \beta_\eta\, \varepsilon \lambda^2\int_\Sd \eta(y)
	     \, \langle \cdot, y \rangle\langle y, 
	     z_0 \rangle\, \mu( \mathrm{d}y) .
  \end{align} 
  Evaluating the latter equality at~$x=z_0$, we deduce that $R_e(\eta)$ is a positive
  eigenvalue of the matrix~$M(\eta)$ associated to the eigenvector~$(1, \beta_\eta)$,
  where:
  \begin{equation}\label{eq:def-Mh-sp}
    M(\eta)=
    \begin{pmatrix}
      \int \eta \, \mathrm{d} \mu & \int \alpha\,\eta \, \mathrm{d} \mu\\
      \varepsilon \int \alpha\,\eta \, \mathrm{d} \mu & \varepsilon\int
      \alpha^2\, \eta \, \mathrm{d} \mu\\ 
    \end{pmatrix}.
  \end{equation}
  We end this step by proving the following equivalence:
  \begin{equation}
    \label{eq:ba=0}
    \beta_\eta=0
    \Longleftrightarrow
    \int \alpha\,\eta \, \mathrm{d} \mu = 0. 
  \end{equation}
  Indeed, if~$\beta_\eta =0$, then the vector~$(1, 0)$ is an eigenvector of $M(\eta)$
  associated to the eigenvalue~$R_e(\eta)$. We deduce from~\eqref{eq:def-Mh-sp} that $\int
  \alpha\,\eta \, \mathrm{d} \mu= 0$. Conversely, if  $\int \alpha\,\eta \, \mathrm{d}
  \mu= 0$, then the matrix~$M(\eta)$ is diagonal with eigenvalues~$\int \eta \,
  \mathrm{d}\mu$ and $\int \alpha^2\, \eta \, \mathrm{d} \mu$. As~$\alpha^2\leq 1$ with
  strict inequality on a set of positive~$\mu$-measure, we deduce that:
  \begin{equation}\label{eq:h>ha2} \int \eta \, \mathrm{d}\mu > \int \alpha^2\, \eta \,
    \mathrm{d} \mu.
  \end{equation}
  Since~$(1, \beta_\eta)$ is an eigenvector of~$M(\eta)$, this implies
  that~$\beta_\eta=0$. This proves~\eqref{eq:ba=0}.

  \medskip

  \textit{Step 2:} $R_e(\eta)$ is the spectral radius of the matrix $M(\eta)$, that
  is,~$R_e(\eta)=\rho(M(\eta))$. We first consider the case $\varepsilon=-1$.
  Since~$\alpha$ is non constant as~$\lambda>0$, we deduce from the Cauchy-Schwarz
  inequality, that the determinant of $M(\eta)$ is negative. As a.s.~$\cmir=1$, we deduce
  that~$R_e(\eta)>0$, and thus the other eigenvalue is negative. Since~$\alpha^2\leq 1$,
  the trace of~$M(\eta)$ is non-negative, thus~$R_e(\eta)$ is the spectral radius of the
  matrix~$M(\eta)$.

  \medskip

  We now consider the case~$\varepsilon=+1$.  Let $\etau$ be the uniform
  strategy with the same cost as~$\eta$. Thanks to~\eqref{eq:Re-pour-1},
  we get $R_e(\etau)=\int \etau\, \rd \mu=\int \eta\, \rd \mu$. 
  Since the non-zero eigenvalues of $T_\kk$,
  that is,~$1$ and~$\lambda^2/d$, are positive, we deduce from
  Corollary~\ref{cor:regular}~\ref{cor:regular-conv}, that the uniform strategies are
  Pareto optimal ($\cpu\subset \cp$), so we have:
  \[
    R_e(\eta)\geq R_e(\etau)=\int \eta \, \mathrm{d} \mu. 
  \]
  We           deduce           from~\eqref{eq:Re-pour-1}           that
  $\beta_\eta \int \alpha\,\eta \, \mathrm{d} \mu\geq 0$.

  On the one hand, if  $\beta_\eta \int \alpha \,\eta \, \mathrm{d} \mu=
  0$, then, 
  by~\eqref{eq:ba=0}, the matrix~$M(\eta)$ is diagonal. Using~\eqref{eq:h>ha2}, we obtain
  that~$R_e(\eta)=\rho(M(\eta))$. On the other hand, if $\beta_\eta \int \alpha\,\eta \,
  \mathrm{d} \mu>0$, then the matrix $M(\eta)$ has positive entries. Since the
  eigenvector~$(1, \beta_n)$ also has positive entries, it is the right Perron eigenvector
  and the corresponding eigenvalue is the spectral radius of~$M(\eta)$, that
  is,~$R_e(\eta)=\rho(M(\eta))$.
  To conclude, the equality~$R_e(\eta)=\rho(M(\eta))$ holds in all cases. 

  \medskip

  \textit{Step 3:} $R_e(\eta)=\int \eta \, \rd \mu \Longleftrightarrow \eta\in
  \So$.  Let~$\eta\in \Delta$ such that~$R_e(\eta)=\int \eta\, \rd \mu$. We deduce
  from~\eqref{eq:Re-pour-1} that~$ \beta_\eta \int \alpha\,\eta \, \mathrm{d} \mu= 0$.
  Thanks to~\eqref{eq:ba=0}, this implies that $\beta_\eta=0$.
  Using~\eqref{eq:Re-pour-alpha}, we obtain that $\int y\eta(y) \, \mu(\mathrm{d} y)=0$
  and thus~$\eta\in \So$. Clearly if~$\eta\in \So$, we deduce from~\eqref{eq:Re-pour-1}
  that $R_e(\eta)=\int \eta \, \mathrm{d} \mu$.

  As a consequence and since~$\cpu\subset\So$, we deduce from Corollary~\ref{cor:regular}
  that if~$\varepsilon=+1$, then~$\cpu \subset \cp$ and thus~$\cp=\So$; and that
  if~$\varepsilon=-1$, then~$\cpu\subset \cpa$ and thus~$\cpa=\So$. \medskip

  \textit{Step 4:} A relation with the \regular{} symmetric kernels of rank two from
  Section~\ref{sec:rank-2-reg}. This step is in the spirit of \cite[Section~7]{ddz-theo}
  on coupled models.  Let~$X$ be a uniform random variable on~$\Sd$. Let~$\Omega_0=[-1,1]$
  endowed with the probability measure $\mu_0(\mathrm{d} t)= c_d\,
  w_d(t)\, \mathrm{d} t$, and set~$\Delta_0$ the set of~$[0, 1]$-valued measurable
  functions defined on~$\Omega_0$. According to \cite[Theorem 8.9]{Kal21}, there exists
  $\eta_0 \in \Delta_0$ such that a.s.:
  \begin{equation}\label{eq:def-h0}
    \eta_0(\langle X, z_0 \rangle)=\E\left[\eta(X)\, |\, \langle X, z_0 \rangle\right].
  \end{equation}
  Set~$\alpha_0=\lambda t$, and define the matrix:
  \[
    M_0(\eta_0)=
    \begin{pmatrix}
      \int_{\Omega_0} \eta_0 \, \mathrm{d} \mu_0 & \int_{\Omega_0}
      \alpha_0\,\eta_0 \, \mathrm{d} \mu_0\\ 
      \varepsilon \int_{\Omega_0} \alpha_0\,\eta_0 \, \mathrm{d} \mu_0 &
      \varepsilon\int_{\Omega_0}
      \alpha_0^2\, \eta_0 \, \mathrm{d} \mu_0\\ 
    \end{pmatrix}.
  \] 
  By construction of~$\eta_0$, we have~$ M_0(\eta_0)=M(\eta)$. Thanks to
  Section~\ref{sec:rank-2-reg}, see Lemma~\ref{lem:2formula} (but for the fact
  that~$\Omega$ therein in replaced by $[-1, 1]$), we get that~$M_0(\eta_0)$ is exactly
  the matrix in~\eqref{eq:matrix-representation}, and thus the spectral radius of $
  M_0(\eta_0)$ is the effective reproduction number of the model associated to the
  \regular{} symmetric kernel of rank two $\kk_0^\varepsilon=1+ \varepsilon
  \alpha_0\otimes \alpha_0$ given in~\eqref{eq:k+} (with~$\Omega$,~$\mu$,~$\alpha$
  replaced by~$\Omega_0$, $\mu_0$ and~$\alpha_0$). We deduce that: if~$\eta$ is Pareto or
  anti-Pareto optimal for the model~$(\Sd, \mu, \kk)$ then so is~$\eta_0$ for the
  model~$(\Omega_0, \mu_0, \kk_0^\varepsilon)$; and if~$\eta_0$ is Pareto or anti-Pareto
  optimal for the model $(\Omega_0, \mu_0, \kk_0^\varepsilon)$, so is any strategy~$\eta$
  such that~\eqref{eq:def-h0} holds.

  \medskip

  We first consider the case~$\varepsilon=+1$. According to Proposition~\ref{prop:rank2},
  we get that the anti-Pareto optimal strategies are~$\eta_0=\ind{[-1, t)}$
  or~$\eta_0=\ind{[-t, 1)}$ for a given cost~$c$ (with~$t$ uniquely characterized by~$c$).
  Using that $0\leq \eta \leq 1$, we deduce that the only possible choice for $\eta$ such
  that~\eqref{eq:def-h0} holds is to take $\eta=\ind{B(-z_0,t)}$
  or~$\eta=\ind{B(-z_0,t)}$. Since~$z_0$ was arbitrary, we get that the only possible
  anti-Pareto optimal strategies belong to~$\Sball$. Notice that anti-Pareto optimal
  strategies exist for all cost~$c\in [0,1]$ as~$\kk>0$ a.s., see Lemma~\ref{lem:k>0-c}
  and \cite[Section~5.4]{ddz-theo} for irreducible kernels, loss function~$R_e$ and
  uniform cost function~$C$ given by~\eqref{def:cost}. Since the set of anti-Pareto
  optimal strategies is also invariant by rotation, we deduce that~$\cpa=\Sball$.

  The case~$\varepsilon=-1$ is similar and thus~$\cp=\Sball$ in this case. (Notice the
  irreducibility of the kernel~$\kk$ is only used in \cite[Lemma~5.13]{ddz-theo} for the
  study of anti-Pareto frontier.)
\end{proof}

\begin{figure}
  \begin{subfigure}[T]{.5\textwidth}
    \centering
    \includegraphics[page=1]{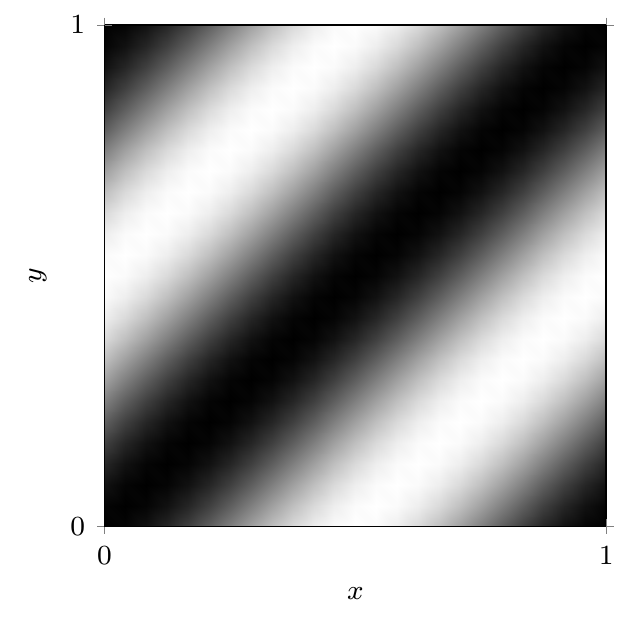}
    \caption{Grayplot of the kernel
     $\kk = 1 +\cos(\pi(x-y))$ on~$[0, 1)$.}\label{fig:geometric+} 
  \end{subfigure}%
  \begin{subfigure}[T]{.5\textwidth} \centering
    \includegraphics[page=2]{regular} \caption{Solid line: the Pareto frontier~$\F$ (which corresponds to the uniform strategies); dashed line: the
      anti-Pareto frontier~$\AF$; blue
    region: all possible outcomes~$\FF$.}
    \label{fig:frontier-geometric+}
  \end{subfigure}
  \begin{subfigure}[T]{.5\textwidth}
    \centering
    \includegraphics[page=3]{regular}
    \caption{Grayplot of the kernel~%
      $\kk = 1 - \cos(\pi(x-y))$ on~$[0, 1)$.}\label{fig:geometric-} 
  \end{subfigure}%
  \begin{subfigure}[T]{.5\textwidth} \centering
    \includegraphics[page=4]{regular} \caption{Solid line:
      the Pareto frontier~$\F$; dashed line: the 
      anti-Pareto frontier~$\AF$ (which corresponds to the uniform strategies); blue
    region: all possible outcomes~$\FF$.}
    \label{fig:frontier-geometric-}
  \end{subfigure}%
  \caption{Two examples of a geometric kernel on the circle~$\R\setminus
    \Z$.} \label{fig:optim_regular}
\end{figure}

\printbibliography

\end{document}